\documentclass[10pt,reqno]{amsart}
\usepackage{amssymb,mathrsfs,color}

\usepackage{wrapfig}
\usepackage{tikz}
\usetikzlibrary{arrows,calc,decorations.pathreplacing}
\definecolor{light-gray1}{gray}{0.90}
\definecolor{light-gray2}{gray}{0.80}

\definecolor{deepgreen}{cmyk}{1,0,1,0.5}

\newcommand{\B}{\mathcal{B}}
\newcommand{\E}{\mathcal{E}}

\newcommand{\HH}{\mathcal{H}}

\newcommand{\NN}{\mathcal{N}}
\newcommand{\M}{\mathcal{M}}
\newcommand{\cS}{\mathcal{S}}

\newcommand{\Cc}{\mathcal{C}}

\newcommand{\CC}{\mathscr{C}}

\newcommand{\Sm}{\mathscr{S}}

\newcommand{\Hp}{\mathbb{H}}
\newcommand{\N}{\mathbb{N}}
\newcommand{\R}{\mathbb{R}}
\newcommand{\Sp}{\mathbb{S}}
\newcommand{\Z}{\mathbb{Z}}

\newcommand{\al}{\alpha}
\newcommand{\be}{\beta}
\newcommand{\ga}{\gamma}
\newcommand{\de}{\delta}
\newcommand{\e}{\varepsilon}
\newcommand{\fy}{\varphi}
\newcommand{\om}{\omega}
\newcommand{\la}{\lambda}
\newcommand{\te}{\theta}
\newcommand{\s}{\sigma}

\newcommand{\De}{\Delta}

\newcommand{\p}{\partial}
\newcommand{\na}{\nabla}

\newcommand{\supp}{\operatorname{supp}}

\makeatletter

\newcommand{\Rmnum}[1]{\expandafter\@slowromancap\romannumeral #1@}
\makeatother

\newcommand{\I}{\infty}

\newcommand{\ti}{\widetilde}

\newcommand{\ang}[1]{\left\langle{#1}\right\rangle}
\newcommand{\abs}[1]{\left\lvert{#1}\right\rvert}

\newcommand{\ali}[1]{\begin{align}\begin{split} #1 \end{split}\end{align}}
\newcommand{\ant}[1]{\begin{align*}\begin{split} #1 \end{split}\end{align*}}
\newcommand{\EQ}[1]{\begin{equation}\begin{split} #1 \end{split}\end{equation}}
\newcommand{\pmat}[1]{\begin{pmatrix} #1 \end{pmatrix}}

\newcommand{\Del}[1]{}

\def\ti{\tilde}

\numberwithin{equation}{section}

\newtheorem{thm}{Theorem}[section]
\newtheorem{cor}[thm]{Corollary}
\newtheorem{lem}[thm]{Lemma}
\newtheorem{prop}[thm]{Proposition}

\newtheorem{claim}[thm]{Claim}
\theoremstyle{remark}
\newtheorem{rem}{Remark}
\newtheorem{defn}{Definition}

\newcommand{\mand}{{\ \ \text{and} \ \  }}

\newcommand{\mif}{{\ \ \text{if} \ \ }}
\newcommand{\mfor}{{\ \ \text{for} \ \ }}
\newcommand{\mas}{{\ \ \text{as} \ \ }}

\def\glei{\mathrm{eq}}
 \def\Id{\mathrm{Id}}

\begin{document}

\title[Scattering for super-critical waves]{Scattering for radial, semi-linear, super-critical wave equations with bounded critical norm}
\author{Benjamin Dodson and Andrew Lawrie}
\begin{abstract} In this paper we study the \emph{focusing} cubic wave equation in $1+5$ dimensions with radial initial data as well as the one-equivariant wave maps equation in $1+3$ dimensions with the model  target  manifolds  $\Sp^3$ and $\Hp^3$. In both cases the scaling for the equation leaves the $\dot{H}^{\frac{3}{2}} \times \dot{H}^{\frac{1}{2}}$-norm of the solution invariant, which means that the equation is \emph{super-critical} with respect to the conserved energy.  Here we prove a conditional scattering result: If the critical norm of the solution stays bounded on its maximal time of existence, then the solution is global in time and scatters to free waves as $t \to \pm \infty$. The methods in this paper also apply to all supercritical power-type nonlinearities for both the focusing and defocusing radial semi-linear equation in $1+5$ dimensions, yielding analogous results. 

\end{abstract}

\thanks{Support of the National Science Foundation, DMS-1103914  for the first author, and DMS-1302782 for the second author, is gratefully acknowledged. }

\maketitle

\section{Introduction} 
In this paper we study three super-critical semi-linear wave equations, namely the focusing cubic wave equation in $\R^{1+5}$ with radial initial data and the one-equivariant wave maps equations from $\R^{1+3} \to \Sp^3$  and from $\R^{1+3}  \to \Hp^3$.  Under certain conditions the former equation serves as a good model for the first of the latter two, which have nonlinearities that arise  naturally from the geometry of the target manifold.  

\subsection{The cubic wave equation in $\R^{1+5}$}
Consider first the Cauchy problem for the focusing cubic  semi-linear wave equation in $\R^{1+5}$,  
\EQ{\label{u gen eq}
&u_{tt}- \Delta u - u^3 = 0,\\
&\vec u(0) = (u_0, u_1), 
}
restricted to the radial setting.  
The conserved energy for solutions, $$\vec u(t):= (u(t), u_t(t)),$$ to~\eqref{u gen eq} is given by 
\ant{
E( \vec u(t)) := \int_{\R^5}  \left[\frac{1}{2}( \abs{u_t(t)}^2 + \abs{\nabla u(t)}^2) - \frac{1}{4} \abs{u(t)}^4\right] \, dx = \textrm{constant}.
}
As we will only be considering radial solutions to~\eqref{u gen eq}, we slightly abuse notation by often writing $u(t, x) = u(t, r)$ where here  $(r, \om)$ with $r= \abs{x}$, $x= r  \om$,  $\om \in  \Sp^4$ are polar coordinates on $\R^5$. In this setting we can rewrite the 	equation~\eqref{u gen eq} as
\EQ{ \label{u eq}
&u_{tt}- u_{rr}- \frac{4}{r} u_r - u^3 = 0,\\
&\vec u(0) = (u_0, u_1),
}
and the conserved energy (up to a constant multiple) by
 \EQ{\label{E}
E( \vec u(t)) := \int_0^{\infty} \left[ \frac{1}{2}( u_t^2(t,r) + u_r^2(t,r)) - \frac{1}{4} u^4(t,r) \right] \, r^4 \, dr.
}
The Cauchy problem~\eqref{u eq} is invariant under the scaling
\EQ{
\vec u(t, r) \longmapsto   \vec u_{\la}(t, r) := (\la^{-1} u( t/ \la, r/ \la), \la^{-2} u_t( t/ \la, r/ \la)).
}
One can also check that this scaling leaves unchanged the $\dot{H}^{\frac{3}{2}} \times \dot{H}^{\frac{1}{2}}(\R^5)$-norm of the initial data. It is for this reason that~\eqref{u eq} is called {\em energy-supercritical}. 

The standard argument based on Strichartz estimates shows that~\eqref{u eq} is  locally well-posed in $\dot{H}^{\frac{3}{2}} \times \dot{H}^{\frac{1}{2}}(\R^5)$. This  means that for all initial data $\vec u(0) =(u_0, u_1) \in \dot{H}^{\frac{3}{2}} \times \dot{H}^{\frac{1}{2}}$, there is a unique solution, $\vec u(t)$, defined on a maximal interval of existence $I_{\max} = I_{\max}( \vec u)$ with $\vec u  \in C^0\left(I_{\max} ; \dot{H}^{\frac{3}{2}} \times \dot{H}^{\frac{1}{2}}(\R^5)\right)$. Moreover, for every compact time interval $J \subset I_{\max}$ we have 
\EQ{
u \in S(J):= L^{2}_t(J; L^{10}_x(\R^5)).
} 
The Strichartz norm $S(J)$ determines a criterion for both scattering and finite time blow up, see Proposition~\ref{prop: small data}.  In particular, one can show that if the initial data $\vec u(0)$ is sufficiently  small  in $ \dot{H}^{\frac{3}{2}} \times \dot{H}^{\frac{1}{2}}$, then the corresponding solution $\vec u(t)$ has finite $S(\R)$-norm and hence scatters to free waves as $t \to \pm \infty$. 

The theory for solutions to~\eqref{u eq} with initial data that is small in $\dot{H}^{\frac{3}{2}} \times \dot{H}^{\frac{1}{2}}$ is thus well understood -- all solutions are global in time and scatter to free waves as $t \to \pm \infty$. However, much less is known regarding the  \emph{asymptotic dynamics} of solutions once one leaves the perturbative regime. 

There are solutions to the focusing problem that blow-up in finite time. For example,   
\ant{
\fy_T(t, x) = \frac{ \sqrt{2}}{T-t}
}
solves the ODE,  $\fy_{tt} = \fy^3$. By the finite speed of propagation, one can construct from $\fy_T$ a compactly supported (in space) self-similar blow-up solution to~\eqref{u eq},  $\vec{u}_T(t)$,  with blow-up time $t=T$. However, such a self-similar solution  must have 
\begin{align*}
\lim_{t \to T} \| \vec{u}_T(t)\|_{\dot{H}^{\frac{3}{2}} \times \dot{H}^{\frac{1}{2}}(\R^5)} =  \infty.
\end{align*}
Such behavior is typically referred to as type-I blow-up. One the other hand, type-II solutions, $\vec u(t)$, are those whose critical norm remains bounded on their maximal interval of existence, $I_{\max}$, i.e., 
\EQ{ \label{type2}
\sup_{t  \in I_{\max}} \| \vec{u}(t)\|_{\dot{H}^{\frac{3}{2}} \times \dot{H}^{\frac{1}{2}}(\R^5)} <  \infty.
}
In this paper we restrict our attention to type-II solutions, i.e., those which satisfy the bound~\eqref{type2}. We prove that if a solution  $\vec u(t)$ to~\eqref{u eq} satisfies~\eqref{type2}, then it  must exist globally in time and scatter to free waves in both time directions. We establish the following result. 

\begin{thm}\label{thm: main}
 Let $\vec u(t) \in \dot{H}^{\frac{3}{2}} \times \dot{H}^{\frac{1}{2}}(\R^5)$ be a radial solution to~\eqref{u eq} defined on its maximal interval of existence $I_{\max}=(T_-, T_+)$. Suppose in addition that  
\EQ{\label{type ii}
\sup_{t \in I_{\max}} \| \vec u(t) \|_{\dot{H}^{\frac{3}{2}} \times \dot{H}^{\frac{1}{2}}(\R^{5})} < \infty.
}
Then,  $I_{\max} = \R$, that is, $\vec u(t)$ is defined globally in time. Moreover,  
\EQ{
\| u \|_{S(\R)}<  \infty, 
}
which means that $\vec u(t)$ scatters to free waves as $t \to  \pm \infty$, i.e., there exist radial solutions $\vec u_{L}^{\pm}(t)  \in \dot{H}^{\frac{3}{2}} \times \dot{H}^{\frac{1}{2}}(\R^5)$ to the free wave equation, $\Box u_L = 0$,  so that 
\EQ{
\| \vec u(t)- \vec u_L^{\pm}(t) \|_{\dot{H}^{\frac{3}{2}} \times \dot{H}^{\frac{1}{2}}(\R^5)} \longrightarrow 0 \mas t \to   \pm \infty.
}
\end{thm}

\begin{rem}
Theorem~\ref{thm: main} is a conditional result. Other than the requirement that the initial data be small in $\dot H^{\frac{3}{2}} \times \dot H^{\frac{1}{2}}$, there is no known general criterion to determine when~\eqref{type ii} is satisfied by the evolution.  We remark that that this type of result is analogous to the work by Duyckaerts, Kenig, and Merle in $3$ dimensions in~\cite{DKM5} and also bears similarity to the $L^{3, \infty}$ result of Escauriaza, S\"eregin, and \v{S}verak for the Navier-Stokes equation, see~\cite{ESS03}. 
\end{rem}

\begin{rem} The proof of Theorem~\ref{thm: main} readily generalizes to all supercritical powers $p  > \frac{7}{3}$ in dimension $d = 5$ for both the focusing and defocusing equations with radial initial data.  As we demonstrate in the next subsection where we consider the case of one-equivariant wave maps, the techniques in this paper are very flexible with regards to  particular algebraic structure of the nonlinearity. For power-type nonlinearities we have chosen to present the details for only the focusing cubic equation to keep the exposition as simple as possible.  Another reason for choosing to study the \emph{focusing} cubic equation as opposed to other power-type nonlinearities is this equation requires several new techniques which fall outside the scope of what has previously been developed in the literature, such as~\cite{DKM5}, which is the first conditional result for a focusing supercritical equation, and~\cite{KM11a, KM11b, KV10CPDE, KV11TAMS, Bul12b}, which address defocusing supercritical waves. 
\end{rem}

\subsection{One-Equivariant Wave Maps} Next we consider one-equivariant wave maps in $1+3$ dimensions taking values in $3$-dimensional rotationally symmetric manifolds $\M$. Let $(r, \theta) \in \R^* \times \Sp^2$ be polar coordinates on $\R^3$ and let $(\psi, \om)$ be geodesic polar coordinates on $\M$, where the metric takes the form 
\ant{
ds^2  = dr^2 + g^2(\psi) d \om^2
}
and $d\om^2$ denotes the  round metric on $\Sp^2$.  Maps $U: \R^{1+3} \to  \M$ can then be written in the form $U(t, r, \theta)  = ( \psi(t,r , \theta), \om(t, r, \te))$.  In the usual one-equivariant (or co-rotational) reduction, one makes the ansatz
\EQ{
U(t, r, \theta)  = ( \psi(t,r), \te)
}
and the wave maps system reduces to a Cauchy problem for the coordinate function $\psi$, viz., 
\EQ{ \label{eq wm}
&\psi_{tt} - \psi_{rr} - \frac{2}{r} \psi_r + \frac{f(\psi)}{r^2} =  0 \\
&\vec \psi(0) = (\psi_0, \psi_1)
}
where $f(\psi) := g(\psi) g'(\psi)$.  The conserved energy is given by 
\EQ{
\E( \vec \psi)(t):= \frac{1}{2} \int_0^{\infty}   \left[ \psi_t^2 + \psi_r^2 + \frac{ 2 g^2( \psi)}{r^2} \right] \, r^2 \, dr = \textrm{constant}. 
}
We also note the scaling invariance, 
\EQ{
\psi(t, r) \longmapsto   \psi_{\la}(t, r):= \psi( t/ \la, r/ \la)
}
Note that it is energetically favorable for a solution to concentrate to a point  by the rescaling  above and sending $\la \to 0$ as we have 
\EQ{
\E( \vec \psi_\la)  = \la \E( \vec \psi)  \to 0 \mas \la \to 0
}
This means that $1+3$ dimensional wave maps are \emph{super-critical} with respect to the conserved energy. It is well know that $3d$ wave maps into positively curved targets such as the $3$-sphere, $\M = \Sp^3$, can blow up in finite time in a self-similar fashion. This was proved by Shatah in~\cite{Shatah} and an explicit example was given by Turok, Spergel \cite{TS} with 
\EQ{
\psi(t, r)  = 2 \arctan( r/t).
}
For negatively curved targets such as $3d$-hyperbolic space, $\M = \Hp^3$, it is not know whether singularities can develop in finite time. 

Wave maps arise as a model in particle physics, particularly with dimension $d=3$, and in this setting are referred to as nonlinear $\s$-models. Here for simplicity we will consider two model targets, namely $\M = \Sp^3$ and $\M = \Hp^3$. 

\subsubsection{Wave maps into $\Sp^3$} The case of the target $ \M = \Sp^3$ has traditionally been of interest due to the possibility of solutions with nontrivial topology.  In our equivariant formulation with the $\Sp^3$ target, we have $g( \psi) = \sin \psi$ and the equation and conserved energy become 
\EQ{\label{wm s3}
&\psi_{tt} - \psi_{rr} - \frac{2}{r} \psi_r + \frac{\sin(2\psi)}{r^2} =  0, \quad \vec \psi(0) = (\psi_0, \psi_1)\\
& \E( \vec \psi)(t):= \frac{1}{2} \int_0^{\infty}   \left[ \psi_t^2 + \psi_r^2 + \frac{ 2 \sin^2( \psi)}{r^2} \right] \, r^2 \, dr 
}
From the above it is clear that for initial data $\vec \psi(0) = (\psi_0, \psi_1)$ to have finite energy one requires that $\lim_{r \to \infty} \psi_0(r) = n \pi$ for some $n \in \N$,  and by a continuity argument, this endpoint is fixed by the evolution. The energy allows for more flexible behavior of $\psi_0(r)$  at $r = 0$ in contrast to the case of $2d$ wave maps into the $2$-sphere, which have topological degree which is fixed by the evolution. Here we see that simply requiring that the solution has finite energy allows for any finite limit $\lim_{r \to 0} \psi (t, r) = \al(t) \in \R$ and this limit can possibly change with the evolution.  Here, our techniques force us to ignore this subtlety as will impose the priori assumption that for all $t \in I_{\max}( \vec \psi)$ we have 
\EQ{ \label{h 32 wm}
 \vec u(t, r)  := (u(t, r), u_t(t, r)) := \left(\frac{\psi(t, r)}{r}, \frac{\psi_t(t, r)}{r} \right) \in  \dot{H}^{\frac{3}{2}}  \times \dot{H}^{\frac{1}{2}} (\R^5)
 }
 With this assumption we can recast the Cauchy problem~\eqref{wm s3} in terms of $\vec u(t)$, which solves 
 \EQ{ \label{u eq wms}
 &u_{tt} - u_{rr} - \frac{4}{r} u_r  = \frac{ 2ru - \sin(2ru)}{r^3}: = F_{\Sp^3}(r, u)\\
 & \vec u(0) = (u_0, u_1)
 }
By Sobolev embedding we must then have $u(t) \in L^5(\R^5)$ for all $t \in I_{\max}$, which in turn implies that 
\ant{
\int_0^{\infty} \frac{\psi^5(t, r)}{r} \, dr <\infty.
}
Hence we must have $\psi(t, 0) = 0$ and $\psi(t, \infty) = 0$ for all $t \in I_{\max}$ once we make the a priori assumption~\eqref{h 32 wm}. Note also that the nonlinearity satisfies 
\EQ{ \label{F S}
&F_{\Sp^3}(r, u) = u^3 Z_{\Sp^3}( r u) \\
&\abs{F_{\Sp^3}(r, u)  } \lesssim \abs{u}^3
}
where $Z_{\Sp^3}( \rho)  =8 \frac{ \rho -  \sin \rho}{ \rho^3}$ is a smooth bounded, \emph{nonnegative} function.

Thus, in the $5d$ formulation~\eqref{u eq wms}, the topologically trivial equivariant wave maps problem into $\Sp^3$ bears many similarities to the \emph{focusing} cubic equation~\eqref{u eq}. We establish a conditional result, which is completely analogous to Theorem~\ref{thm: main}.  Before stating our main theorem regarding equivariant wave maps, we first introduce the case the of the $\Hp^3$ target.

 \subsubsection{Wave maps into $\Hp^3$}
 The case of the negatively curved target $ \M = \Hp^3$ bears many similarities to the defocusing radial cubic wave equation in $\R^{1+3}$ after the reduction performed below.  In the equivariant formulation with the $\Hp^3$ target, we have $g( \psi) = \sinh \psi$ and the equation and conserved energy become 
\EQ{\label{wm h3}
&\psi_{tt} - \psi_{rr} - \frac{2}{r} \psi_r + \frac{\sinh(2\psi)}{r^2} =  0, \quad \vec \psi(0) = (\psi_0, \psi_1)\\
& \E( \vec \psi)(t):= \frac{1}{2} \int_0^{\infty}   \left[ \psi_t^2 + \psi_r^2 + \frac{ 2 \sinh^2( \psi)}{r^2} \right] \, r^2 \, dr 
}
From the above it is clear that for initial data $\vec \psi(0) = (\psi_0, \psi_1)$ to have finite energy one requires that $\lim_{r \to \infty} \psi_0(r) = 0$  and this endpoint is fixed by energy conservation.  As in the case of the $\Sp^3$ target, the energy allows for more flexible behavior of $\psi_0(r)$  at $r = 0$. Here we see that simply requiring that the solution has finite energy allows for any finite limit $\lim_{r \to 0} \psi (t, r) = \al(t) \in \R$ and this limit can possibly change with the evolution.  Here, again we ignore this subtlety and impose the priori assumption that for all $t \in I_{\max}( \vec \psi)$ we have 
\EQ{ \label{h 32 wm h}
 \vec u(t, r)  := (u(t, r), u_t(t, r)) := \left(\frac{\psi(t, r)}{r}, \frac{\psi_t(t, r)}{r} \right) \in  \dot{H}^{\frac{3}{2}}  \times \dot{H}^{\frac{1}{2}} (\R^5)
 }
 With this assumption we can recast the Cauchy problem~\eqref{wm h3} in terms of $\vec u(t)$, which solves 
 \EQ{ \label{u eq wmh}
 &u_{tt} - u_{rr} - \frac{4}{r} u_r  = \frac{ 2ru - \sinh(2ru)}{r^3}: = F_{\Hp^3}(r, u)\\
 & \vec u(0) = (u_0, u_1)
 }
Assuming~\eqref{h 32 wm h} and by Sobolev embedding we must then have $u(t) \in L^5(\R^5)$ for all $t \in I_{\max}$, which in turn implies that 
\ant{
\int_0^{\infty} \frac{\psi^5(t, r)}{r} \, dr <\infty.
}
Hence we must have $\psi(t, 0) = 0$ once we make the a priori assumption~\eqref{h 32 wm}. Note also that the nonlinearity satisfies 
\EQ{
&F_{\Hp^3}(r, u) = u^3 Z_{\Hp^3}( r u) \\
}
where $Z_{\Hp^3}( \rho)  =8 \frac{ \rho -  \sinh \rho}{ \rho^3}$ is a smooth \emph{nonpositive} function.  If we assume an a priori uniform bound 
\EQ{ \label{type ii wm}
\sup_{t \in I_{\max}} \|  \vec u(t) \|_{\dot H^{\frac{3}{2}} \times \dot{H}^{\frac{1}{2}}(\R^5)} \le C< \infty,
}
then we have $L^{\infty}$ control on $ \psi = r u$ by radial Sobolev embedding, 
\EQ{ 
 \abs{ r u(t, r)} \lesssim \sup_{t \in I_{\max}} \|  \vec u(t) \|_{\dot H^{\frac{3}{2}} \times \dot{H}^{\frac{1}{2}}(\R^5)} \lesssim C,
 } 
which follows from Lemma~\ref{lem: rad se}. With the assumption~\eqref{type ii wm}, we thus have a uniform bound on $Z_{\Hp^3} ( ru)$ and hence 
\EQ{ \label{F H bound}
\abs{F_{\Hp^3}(r, u)} \lesssim  \abs{u}^3
}
making the analogy with the defocusing cubic equation in $\R^{1+5}$ clear.

Finally, we state our main result for wave maps, which holds for both the $\Sp^3$ target,~\eqref{u eq wms},  and for the $\Hp^3$ target,~\eqref{u eq wmh}. 
\begin{thm}\label{thm: wm}
Let $\vec u(t) \in \dot{H}^{\frac{3}{2}} \times \dot{H}^{\frac{1}{2}}(\R^5)$ be a radial solution to either~\eqref{u eq wms} or to~\eqref{u eq wmh} defined on its maximal interval of existence $I_{\max}=(T_-, T_+)$. Suppose in addition that  
\EQ{\label{type ii1}
\sup_{t \in I_{\max}} \| \vec u(t) \|_{\dot{H}^{\frac{3}{2}} \times \dot{H}^{\frac{1}{2}}(\R^{5})} < \infty.
}
Then,  $I_{\max} = \R$, that is, $\vec u(t)$ is defined globally in time. Moreover,  
\EQ{
\| u \|_{S(\R)}<  \infty, 
}
which means that $\vec u(t)$ scatters to free waves as $t \to  \pm \infty$, i.e., there exist radial solutions $\vec u_{L}^{\pm}(t)  \in \dot{H}^{\frac{3}{2}} \times \dot{H}^{\frac{1}{2}}(\R^5)$ to the free wave equation, $\Box u_L = 0$,  so that 
\EQ{
\| \vec u(t)- \vec u_L^{\pm}(t) \|_{\dot{H}^{\frac{3}{2}} \times \dot{H}^{\frac{1}{2}}(\R^5)} \longrightarrow 0 \mas t \to   \pm \infty.
}
\end{thm}

\subsection{History of the problems}
There is very little known about energy super-critical semi-linear wave equations on $\R^{1+d}$, at least when compared to the vast body of literature devoted to their sub-critical and critical counterparts. 

There are several conditional results in the same vein as Theorem~\ref{thm: main} and Theorem~\ref{thm: wm} in the case of  \emph{defocusing} super-critical equations, where Morawetz type identities can be used;  see for example \cite{KM11a, KM11b, KV11TAMS, KV11a, Bul12a, Bul12b}. The only such results for  \emph{focusing} type equations as considered here are the work of Duyckaerts, Kenig, and Merle,~\cite{DKM5}, concerning super-critical power-type nonlinearities in dimension $d = 3$, and the work of the second author on a semi-linear Skyrme-type equation in~\cite{L13}. 

There is much less know in the way of unconditional results.  Recently,  in an exciting new direction,  Krieger and Schlag,~\cite{KS14}, have constructed a  family of solutions to the super-critical power type equation in $d=3$, which are smooth, global-in-time,  have ~\emph{infinite} critical norm, and are stable under small perturbations. 

For the focusing NLS and focusing wave equations in  high dimensions, $d \ge 11$, there have recently been  blow-up constructions based on the bubbling off of a solution to the underlying elliptic equation; see Merle, Raphael, Rodnianski,~\cite{MRR} and Collot,  \cite{Coll} (we note that in the aforementioned works something different is meant by  ``type II" than how this phrase is used in this paper). 

Super-critical equivariant wave maps in $d=3$ as considered here are called nonlinear $\s$-models in particle physics and have been extensively studied. As we mentioned above, self-similar blow-up was demonstrated by Shatah,~\cite{Shatah}, and Turok-Spergel in the case of the $\Sp^3$ target.  Donninger  has established stability  for such self-similar solutions,~\cite{Don10, Don11}; see also the work of Bizon,~\cite{Biz00}. Similar results have been established for power-type nonlinearities, see~\cite{DS14} for a stability result, and~\cite{BMW} for a construction of an infinite family of smooth, self-similar solutions.  Surprisingly, blow-up can occur even in the case of wave maps into negatively curved targets in high enough dimensions, as shown by Cazenave, Shatah, and Tahvildar-Zadeh, \cite{CST}; see  the book of Shatah and Stuwe~\cite{SSbook} for more.  
 
 Equations of the form 
\EQ{ \label{sl}
\Box u  =  \pm \abs{u}^{p-1} u
} 
for energy critical and sub-critical  values of $p$ have been  extensively studied. When say,  $d=3$, the energy critical power, $p=5$,  exhibits markedly  different phenomena than both the subcritical and supercritical problems. Global existence and scattering for all finite energy data was established by Struwe,~\cite{Struwe88}, for the radial defocusing equation and by Grillakis,~\cite{Gri90}, in the nonradial, defocusing case. 

In the case of  the focusing energy critical equation, type-II blow up does occur,  as explicitly demonstrated by Krieger, Schlag, and Tataru~\cite{KST3}, by way of  an energy concentration scenario resulting in the bubbling  off of the unique radial  ground state solution, $W$, for the underlying elliptic equation; see also \cite{KS12, DHKS, DK}. 

In~\cite{KM08}, Kenig and Merle initiated an extremely effective program of attack for semilinear equations such as~\eqref{sl} with the concentration compactness/rigidity method based on the fundamental profile decompositions of Bahouri and G\'erard,~\cite{BG}. There they gave a characterization of possible dynamics for solutions with energy strictly below the threshold energy of the ground state elliptic solution, $W$. The seminal work of Duyckaerts, Kenig, and Merle \cite{DKM1,DKM3, DKM2, DKM4} gave a classification of possible dynamics for large energies. To be more precise,  all type-II radial solutions asymptotically resolve into a sum of rescaled solitons  plus a radiation term. Dynamics at the threshold energy of $W$ have been studied by Duyckaerts and Merle~\cite{DM} and slightly above the threshold energy by Krieger, Nakanishi, and Schlag in~\cite{KNS13AJM, KNS13DCDS, KNS14CMP}. 

For analogues of Theorem~\ref{thm: main} and Theorem~\ref{thm: wm} for energy sub-critical equations we refer the reader to~\cite{Shen}, and to the recent work of the authors,~\cite{DL14a}. We also mention the remarkable works of Merle and Zaag,~\cite{MZ03ajm, MZ05ma} where it was determined that all subcritical blow-up for the focusing equation occurs at the self-similar rate. 

\subsection{Outline of the proofs of Theorem~\ref{thm: main} and Theorem~\ref{thm: wm}}

The proofs of Theorem~\ref{thm: main} and Theorem~\ref{thm: wm} proceed via the concentration compactness/ rigidity method developed by Kenig and Merle in~\cite{KM06, KM08}. The method is  based around an elaborate contradiction argument -- if Theorem~\ref{thm: main} (respectively Theorem~\ref{thm: wm}) were false, the linear and nonlinear profile decompositions of Bahouri and G\'erard~\cite{BG} allow for a construction of a minimal  non-scattering solution to~\eqref{u eq}, called the critical element. Here  minimality refers to the size of the norm in~\eqref{type ii} (resp.~\eqref{type ii1}). This construction is standard in the field and we give a brief outline in Section~\ref{sec: cc}. The crucial property of the critical element that drives the contradiction argument is that its trajectory is \emph{pre-compact} up to modulation in the space $\dot{H}^{\frac{3}{2}} \times \dot{H}^{\frac{1}{2}}$. The goal is then to prove that this compactness property is too restrictive  for a nonzero solution and therefore the critical element does not exist.  

For the rigidity argument that rules out the critical element, we roughly follow the strategy implemented in~\cite{DKM5} for the focusing super-critical wave equation in $3$-dimensions.  Given the \emph{exterior energy estimates} for the free radial wave equation proved in~\cite{KLS}, the precise strategy in~\cite{DKM5} can be adapted to super-critical equations of the form 
\EQ{
\Box u = \pm \abs{u}^{p-1} u
}
in higher dimensions, in particular $d=5$, \emph{but only when $p>5$}, i.e., for  critical regularity levels $\dot{H}^s \times  \dot{H}^{s-1}$ with $s \ge 2$.  There are several instances where the techniques used in ~\cite{DKM5} break down for supercritical powers $ p<5$, and in particular for the cubic-type equations considered here. Here we build on the strategy developed in~\cite{DKM5}, by developing collection of robust new techniques that work for all powers $p > \frac{7}{3}$ as well as for more complicated nonlinearities such as those which arise in the context of equivariant wave maps, as in~\eqref{u eq wms} and~\eqref{u eq wmh}. 

As in~\cite{DKM5} we reduce to two scenarios for the critical element $ \vec u^*(t)$, namely, 
\begin{enumerate}
\item $\vec u^*(t)$ is a self-similar blow-up solution with pre-compact trajectory, see Proposition~\ref{prop: r2}.
\item $\vec u^*(t)$ is pre-compact on its entire interval of existence $I$ up to a modulation parameter $N(t)$ that is bounded away from $0$ on $I$, see Proposition~\ref{prop: r1}. 
\end{enumerate}

In the first situation, $(1)$, we follow~\cite{DKM5} by introducing self-similar coordinates and using compactness to produce a stationary solution to a non-degenerate  elliptic equation with zero boundary data.  However, a straightforward application of their approach adapted to $5d$ breaks down for the cubic equation and the regularity $\dot{H}^{\frac{3}{2}} \times \dot{H}^{\frac{1}{2}}$. We overcome this difficulty  by proving, in Section~\ref{sec: dd},  a bound on the $\dot{H}^{\frac{5}{2}} \times \dot{H}^{\frac{3}{2}}$ norm of $\vec u^*(t)$, i.e., we show that solutions with pre-compact trajectories are in fact more regular than what is given by scaling; see Proposition~\ref{prop: u 52} and Proposition~\ref{prop: wm 52}. This bound on the $\dot{H}^{\frac{5}{2}} \times \dot{H}^{\frac{3}{2}}$ norm of $\vec u^*(t)$ then implies (by interpolation) that a self-similar compact blow-up solution is in fact also compact in $\dot{H}^2 \times \dot{H}^1 \cap \dot{H}^{\frac{3}{2}} \times \dot{H}^{\frac{1}{2}} $. With this additional regularity we are able to rule out compact self-similar blow up with a somewhat simpler implementation of the argument in~\cite{DKM5}. The crucial gain of regularity in Section~\ref{sec: dd} is established using the so-called double Duhamel trick which we will describe briefly below.  We note that an analogous implementation of the double Duhamel trick was performed by the authors  in~\cite{DL14a} for the cubic equation in $d=3$. 

There are several  major difficulties  when trying to rule out a critical element $\vec u^*(t)$ as in $(2)$ for supercritical equations as considered here. The first is that  $\vec u^*(t)$ is constructed in the homogeneous space $\dot{H}^{\frac{3}{2}} \times \dot{H}^{\frac{1}{2}}(\R^5)$. Although having a pre-compact trajectory up to modulation in this space is a strong property, there is very little in practice that one can do at this regularity, since useful global quantities such as the conserved energy are at the level of $(\dot H^1 \cap \dot{H}^{\frac{5}{4}}) \times L^2$, where here the intersection with $\dot{H}^{\frac{5}{4}}$ arises due to the $L^4(\R^5)$ term in the conserved energy associated to~\eqref{u eq} and Sobolev embedding.  In order to resolve this first issue, we prove in Section $4$ that solutions $\vec u^*(t)$ as in $(2)$ have more decay than what is given by scaling. In particular, we use the another implementation of the double Duhamel trick to prove that in fact a compact solution as in $(2)$ above must be uniformly bounded in $\dot{H}^{\frac{3}{4}} \times \dot{H}^{-\frac{1}{4}}$. Interpolating with the critical $ \dot{H}^{\frac{3}{2}} \times \dot{H}^{\frac{1}{2}}$  bound then implies compactness in the space $\dot{H}^1 \times L^2$, see Lemma~\ref{lem: compact}. An alternative method for establishing additional spacial decay based on delicate flux-type estimates  was used in \cite{KM11a, KM11b, DKM5}. However, this method does not seem to work for the cubic type nonlinearities  considered here, thus motivating the different approach based on the double-Duhamel trick. 

 Second, even once the decay issue is resolved, there are no known useful viral or Morawetz type inequalities for super-critical \emph{focusing}-type equations such as~\eqref{u eq} and~\eqref{u eq wms}.  Here we rely on a new rigidity argument, see~\cite{KLS, L13},  that is  based on the exterior energy estimates for the underlying free equation, see Proposition~\ref{prop: ext en}, and a good understanding of the underlying elliptic equations, see Lemma~\ref{lem: hm} and Lemma~\ref{lem: ell}. We refer the reader to Section~\ref{sec: rig} for the details of the argument, which was inspired by the so-called \emph{channels of energy} method developed by Duyckaerts, Kenig, and Merle in the seminal papers,~\cite{DKM4, DKM5}. We emphasize that this rigidity argument does not require the use of any monotone quantities at the level of the nonlinear dynamical equation and is thus very flexible when it comes to the structure of the nonlinearity.  Our implementation of this method relies crucially on the additional decay for compact trajectories proved in Section~\ref{sec: dd}; see in particular Proposition~\ref{prop: u 34} and Proposition~\ref{prop: wm 34}.

In general, solutions to~\eqref{u eq},~\eqref{u eq wms} or~\eqref{u eq wmh} are only as regular and only decay as much as their initial data as evidenced by the presence of the free propagator $S(t)$ in the Duhamel representation for the solution 
\EQ{\label{duhamel}
\vec u^*(t_0)  = S(t_0-t) \vec u^*(t)  + \int_{t}^{t_0} S(t_0-s) (0, \pm F(u^*)) \, ds.
}
The critical element is different however since  the pre-compacntess of its trajectory is at odds with the dispersive properties of the free part, $S(t_0-t) \vec u^*(t)$. It follows that  the first term on the right-hand-side above must  vanish weakly as $t  \to \sup I_{\max}, \inf I_{\max} $. The second term on the right-hand-side of~\eqref{duhamel} with $t = T_+$ or $t=T_-$ therefore  encodes both the regularity and the spacial decay of the critical element, and in fact gains can be expected due to the presence of the nonlinear term $F(u)$. The additional regularity and decay are extracted by way of the ``double Duhamel trick," which refers to the consideration of the pairing of 
\ant{
\ang{ \int_{T_1}^{t_0} S(t_0-s) (0, F(u)) \, ds, \int^{T_2}_{t_0} S(t_0-\tau) (0, F(u)) \, d\tau }
} 
where $T_1<t_0$ and $T_2>t_0$.  This technique was introduced  by Colliander, Keel, Staffilani, Takaoka, Tao \cite{ITeamAnnals}, Tao~\cite{Tao07}  and later utilized in the Kenig-Merle framework for non-linear Schr{\"o}dinger equations by Killip, Visan~\cite{KV10CPDE, KV10AJM, KVClay},  and for semi-linear wave equations~\cite{KV11TAMS, Bul12a, Bul12b}. This method is related to the in/out decomposition used by Killip, Tao, Visan~\cite[Section~$6$]{KTV09}. For more details on how to exploit the fact that differing time directions are chosen above we refer the reader to Section~\ref{sec: dd}.

\section{Preliminaries}

\subsection{Some facts from harmonic analysis}
We will denote by $P_N$ the usual Littlewood-Paley projections onto frequencies of size $ \abs{ \xi} \simeq N$ and by $P_{ \le N}$ the projection onto frequencies $ \abs{ \xi} \lesssim N$. The frequency size $N$ will often be a dyadic number $N = 2^k$ for some $k \in \Z$ and  in this case we will also  employ the following abuse of notation: When we write $P_k$ with a \emph{lowercase} subscript $k$, this will mean projection onto frequencies $ \abs{\xi} \simeq 2^k$. 

We recall the Bernstein inequalities.  
\begin{lem}[Bernstein's inequalities]\cite[Appendix A]{Taobook} \label{lem bern} Let $1 \le p \le q \le \infty$ and $s \ge 0$. Let $ f: \R^d \to \R$. Then 
\EQ{ \label{bern}
&\|P_{\ge N} f\|_{L^p} \lesssim N^{-s} \| \abs{\na}^s P_{\ge N} f\|_{L^p},\\
&\|P_{\le N} \abs{\na}^s f\|_{L^p} \lesssim N^{s} \|  P_{\le N} f\|_{L^p}, \, \quad
\|P_{ N} \abs{\na}^{\pm s} f\|_{L^p} \simeq N^{ \pm s} \|  P_{ N} f\|_{L^p}\\
&\|P_{\le N} f\|_{L^q} \lesssim N^{\frac{d}{p}- \frac{d}{q}} \| P_{\le N} f\|_{L^p}, \, \quad
\|P_{ N} f\|_{L^q} \lesssim N^{\frac{d}{p}- \frac{d}{q}} \| P_{N} f\|_{L^p}.
}
\end{lem} 
In what follows we will also require the the notion of a frequency envelope. 
\begin{defn}\cite[Definition~$1$]{Tao1} \label{def: freq en}We define a \emph{frequency envelope} to be a sequence  $\be = \{\be_k\}$ of positive real numbers with $\be \in \ell^2$ and 
\ant{
\| \be\|_{\ell^2}  \lesssim B.
}
 If $\be$ is a frequency envelope and $(f, g) \in \dot{H}^s \times \dot{H}^{s-1}$ then we say that \emph{$(f, g)$  lies underneath $\be$} if 
\ant{
\| (P_k f,P_k g)\|_{\dot H^s \times \dot{H}^{s-1}} \le \be_k \, \quad \forall k \in \Z,
} 
and we note that if $(f, g)$ lie underneath $\be$ then we have 
\ant{
\| (f, g)\|_{\dot H^s \times \dot{H}^{s-1}} \lesssim B.
}
\end{defn}

Next we recall a refinement of the Sobolev embedding for radially symmetric functions, which follows from the Hardy-Littlewood-Sobolev inequality. 
\begin{lem}[Radial Sobolev Embedding]\cite[Corollary A.3]{TVZ}\label{lem: rad se} Let  $0 <  \ga <5$ and suppose $f \in \dot{W}^{\ga, p}(\R^5)$ is a radial function. Suppose that 
\ant{
\beta > -\frac{5}{q}, \quad 1 \le  \frac{1}{p} + \frac{1}{q} \le 1+ \ga, \quad  \frac{5}{p'} + \frac{5}{q'}  =  5 - \be - \ga
}
and at most one of the equalities 
$
 q=1, \, \, \, q= \infty, \, \, p =1, \, \, p = \infty,\, \,    \frac{1}{p} + \frac{1}{q} = 1+ \ga,
 $
 holds. Then 
 \EQ{
 \| r^{\be} f \|_{L^{q'}(\R^5)} \le C \|f\|_{\dot{W}^{\ga, p}(\R^5)}.
 }
\end{lem}
\subsection{Strichartz estimates}
For the small data theory we require Strichartz estimates for the linear wave equation in $\R^{1+5}$, 
\EQ{ \label{lin eq}
&v_{tt} - \Delta v = G,\\
&\vec v(0) = (v_0, v_1).
}
A free wave will mean a solution to~\eqref{lin eq} with $G=0$ and will be denoted by $\vec u(t) = S(t) \vec u(0)$. We say that a triple $(p, q, \ga)$ is admissible  if 
\EQ{ \label{adm} 
p, q \ge 2,  \, \, \frac{1}{p} + \frac{2}{q}   \le 1, \quad
\frac{1}{p} + \frac{5}{q}  = \frac{5}{2} - \ga.
}
The Strichartz estimates below are standard and we refer the reader to~\cite{Kee-Tao, LinS} or the book~\cite{Sogge} as well as the references therein for proofs. 

\begin{prop}\cite{Kee-Tao, LinS, Sogge} \label{prop stric}Let $\vec v(t)$ be a solution to~\eqref{lin eq} with initial data $\vec v(0) \in \dot{H}^s \times \dot{H}^{s-1}(\R^5)$ for $s >0$. Let $(p, q, \ga)$, and $(a, b, \rho)$ be admissible triples. Then, for any time interval $I \ni 0$ we have the estimates 
\EQ{\label{stric}
\| v \|_{L^{p}_t(I; W^{s- \ga, q}_x)} \lesssim  \| (v_0, v_1) \|_{\dot H^s \times \dot H^{s-1}} + \|G\|_{L^{a'}_t(I; W^{s-1+ \rho, b'}_x)}.
}
\end{prop}

\subsection{Small data theory: global existence, scattering and the perturbation lemma}
The usual argument based on Proposition~\ref{prop stric} with $s = \frac{3}{2}$, $(p, q, \ga) = (2, 10, 3/2)$, and $(a', b', \rho) = (1, 2, 0)$ yields the following standard small data result. 
\begin{prop}[Small data theory]\label{prop: small data}
Let $\vec u(0) = (u_0, u_1) \in \dot{H}^{\frac{3}{2}} \times \dot{H}^{\frac{1}{2}}(\R^5)$ be initial data for either~\eqref{u eq},~\eqref{u eq wms} or~\eqref{u eq wmh}. Then there is a unique, solution $\vec u(t) \in \dot{H}^{\frac{3}{2}} \times \dot{H}^{\frac{1}{2}}(\R^5)$ on a maximal interval of existence $I_{\max}( \vec u )= (T_-(\vec u), T_+( \vec u))$.  Moreover, for any compact interval $J \subset I_{\max}$ we have 
\ant{ 
\| u\|_{L^2_t(J; L^{10}_x)(\R^5)} < \infty.
}
A globally defined solution $\vec u(t)$ for $t\in [0, \infty)$ scatters as $ t \to \infty$ to a free wave, i.e., a solution $\vec u_L(t)$ of $\Box u_L =0$
if and only if  $ \|u \|_{L^2_t([0, \infty), L^{10}_x)}< \infty$. In particular, there exists a constant $\de_0>0$ so that  
\EQ{ \label{global small}
 \| \vec u(0) \|_{\dot{H}^{\frac{3}{2}} \times \dot{H}^{\frac{1}{2}}} < \de_0 \Longrightarrow  \| u\|_{L^2_t(\R; L^{10}_x)} \lesssim \|\vec u(0) \|_{\dot{H}^{\frac{3}{2}} \times \dot{H}^{\frac{1}{2}}} \lesssim \de_0
 }
and therefore $\vec u(t)$  scatters to free waves as $t \to \pm \infty$. Lastly, we recall  the standard finite time blow-up criterion: 
\EQ{ \label{ft bu}
T_+( \vec u)< \infty \Longrightarrow \|u \|_{L^2_t([0, T_+( \vec u));L^{10}_x)} = + \infty
}
A nearly identical statement holds when $- \infty< T_-( \vec u)$. 
\end{prop}

Another standard result is the Perturbation Lemma for approximate solutions to~\eqref{u eq},~\eqref{u eq wms}, or~\eqref{u eq wmh} which is needed in the concentration compactness procedure in Section~\ref{sec: cc}.  
\begin{lem}[Perturbation Lemma] \label{lem: pert}
There are continuous functions \\$\e_0,C_0:(0,\I)\to(0,\I)$ such that the following holds:
Let $I\subset \R$ be an open interval (possibly unbounded), $\vec u, \vec v\in C(I; \HH)$  satisfying for some $A>0$
\ant{ 
 \|\vec v\|_{L^\infty(I;\dot{H}^{\frac{3}{2}} \times \dot H^{\frac{1}{2}})(\R^5)} +   \|v\|_{L^{2}_t(I;L^{10}_x(\R^5))} & \le A \\
 \|\glei(u)\|_{L^{1}_t(I; \dot{H}^{\frac{1}{2}}_x)}
   + \|\glei(v)\|_{L^{1}_t(I; \dot{H}^{\frac{1}{2}}_x)} + \|w_0\|_{L^2_t(I;L^{10}_x)} &\le \e \le \e_0(A),
   }
where $\glei(u)$ is either $ \glei(u):=\Box u - u^3$, $\glei(u) := \Box u - Z_{\Sp^3}(ru) u^3$ or $ \glei(u) := \Box u - Z_{\Hp^3} (ru) u^3$  in the sense of distributions, and $\vec w_0(t):=S(t-t_0)(\vec u-\vec v)(t_0)$ with $t_0\in I$ arbitrary but fixed.  Then
\ant{ 
  \|\vec u-\vec v-\vec w_0\|_{L^\infty\left(I;\dot{H}^{\frac{3}{2}} \times \dot H^{\frac{1}{2}}\right)}+\|u-v\|_{L^2_t(I;L^{10}_x)} \le C_0(A)\e.}
  In particular,  $\|u\|_{S(I)}<\I$.
\end{lem}


\section{Concentration compactness}\label{sec: cc}

\subsection{Existence and compactness of a critical element}
We begin the proofs of Theorem~\ref{thm: main} and Theorem~\ref{thm: wm}. In both cases we follow the concentration-compactness/rigidity method introduced by Kenig and Merle in~\cite{KM06, KM08}. The concentration compactness aspect of the argument  is based on the fundamental linear and nonlinear profile decompositions of Bahouri and Gerard,~\cite{BG},  and is by now standard. We essentially use the scheme from~\cite{KM10}, which is a refinement of the techniques  used in~\cite{KM06, KM08} to extract a critical element. Indeed, the conclusion of this section is that  in the event that Theorem~\ref{thm: main} or Theorem~\ref{thm: wm} fails, there exists a minimal, nontrivial, non-scattering solution to~\eqref{u eq}, that  is referred to as the critical element. 

First some notation, and  we follow~\cite{KM10} for convenience.  Given initial data $ (u_0, u_1) \in \dot H^{\frac{3}{2}} \times \dot H^{\frac{1}{2}}(\R^5)$ we denote by  $\vec u(t) \in \dot H^{\frac{3}{2}} \times \dot H^{\frac{1}{2}}(\R^5)$ the unique solution to either~\eqref{u eq},~\eqref{u eq wms}, or ~\eqref{u eq wmh}  with initial data $\vec u(0) = (u_0, u_1)$ defined on the maximal interval of existence $I_{\max}(\vec u) := (T_-( \vec u), T_+( \vec u))$.  For $A>0$ define 
\ant{
\B(A):= \left\{ (u_0, u_1) \in\dot H^{\frac{3}{2}} \times \dot H^{\frac{1}{2}}   \, \,  \, : \, \,\,        \|\vec u(t)\|_{L^{\infty}_t([0, T_+(\vec u)); \dot H^{\frac{3}{2}} \times \dot H^{\frac{1}{2}})} \le A\right\}.
}
\begin{defn} We say that $\mathcal{SC}(A)$ holds if for all $ \vec u = (u_0, u_1) \in \B(A)$ we have $T_+(\vec u) = + \infty$ and $ \|u \|_{S([0, \infty))} < \infty$. We also say that $\mathcal{SC}(A;  \vec u)$ holds if $\vec u \in \B(A)$, $T_+ (\vec u) = +\infty$ and $ \|u \|_{S([0, \infty))} < \infty$. 
\end{defn} 
\begin{rem}
Recall from Proposition~\ref{prop: small data} that $\| u \|_{S([0, \infty))}< \infty$ if and only if $\vec u$ scatters to a free waves as $t \to + \infty$. It follows that  both Theorem~\ref{thm: main}  and Theorem~\ref{thm: wm} are equivalent to the statement that $\mathcal{SC}(A)$ holds for all $A>0$. 
\end{rem} 

Now suppose that Theorem~\ref{thm: main} (resp. Theorem~\ref{thm: wm}) {\em is false}. By  Proposition~\ref{prop: small data}, there is an $A_0>0$  such that $\mathcal{SC}(A_0)$ holds.  As we are assuming that Theorem~\ref{thm: main} (resp. Theorem~\ref{thm: wm}) fails, there exists a threshold value $A_C$ so that for $A<A_C$, $\mathcal{SC}(A)$ holds, and for $A>A_C$, $\mathcal{SC}(A)$ fails. It is clear  that $0<A_0<A_C$. The standard conclusion of the supposed failure of the Theorem~\ref{thm: main} (resp. Theorem~\ref{thm: wm}) is that there then must exist a non-scattering solution $\vec u(t)$ to~\eqref{u eq} so that $\mathcal{SC}(A_C, \vec u)$ fails, which enjoys certain minimality and compactness properties. 

 We state a refined version of this result below, and we refer the reader to~\cite{KM10, KM11a, KM11b} for the details of the argument. As usual, the main ingredients are the linear and nonlinear Bahouri-Gerard type profile decompositions from~\cite{BG} used together with the nonlinear perturbation theory, Lemma~\ref{lem: pert}. 

\begin{prop} \label{prop: ce} Suppose that Theorem~\ref{thm: main} (resp. Theorem~\ref{thm: wm}) is false. Then, there exists a solution $\vec u^*(t)$, referred to as a critical element  such that $\mathcal{SC}(A_C;  \vec u^*)$ fails. Moreover, we can assume that $\vec u^*(t)$ does not scatter in either time direction, which means that  
\EQ{\label{blow up}
  \|u^*\|_{S((T_-(\vec u), 0])}=\|u^*\|_{S([0, T_+(\vec u)))} =  \infty.
}
In addition, there exists a continuous  function $N: I_{\max}(\vec u) \to (0, \infty)$ so that the set 
\EQ{\label{K}
K:= \left\{ \left(\frac{1}{N(t)} u^*\left( t, \frac{\cdot}{N(t)} \right), \, \frac{1}{N^2(t)} u_t^*\left( t, \frac{\cdot}{N(t)} \right) \right) \mid t \in I_{\max}(\vec u)\right\}
} 
is pre-compact in $\dot{H}^\frac{3}{2} \times \dot H^{\frac{1}{2}} (\R^3)$ and we have 
\EQ{ \label{inf N0}
\inf_{t \in [0, T_+(\vec u))} N(t) >0.
}
\end{prop}

In what follows it will be convenient to give a name to the compactness property~\eqref{K} satisfied by the critical element. 
\begin{defn}[The Compactness Property] Let $I \ni 0$ be a time interval and suppose $\vec u(t)$ be a solution to either ~\eqref{u eq},~\eqref{u eq wms}, or~\eqref{u eq wmh}  on  an interval $I$. We will say $\vec u(t)$ has the \emph{compactness property on $I$} if there exists a continuous  function $N: I \to (0, \infty)$ so that the set 
\ant{
K:= \left\{ \left(\frac{1}{N(t)} u\left( t, \frac{\cdot}{N(t)} \right), \, \frac{1}{N^2(t)} u_t\left( t, \frac{\cdot}{N(t)} \right) \right) \mid t \in I\right\}
} 
is pre-compact in $\dot{H}^\frac{3}{2} \times \dot H^{\frac{1}{2}} (\R^5)$. 
\end{defn} 

\begin{rem}[Uniformly Small Tails]\label{Ceta}
A straightforward consequence of a solution having the \emph{compactness property} on an interval $I$ is that, after modulation,  we can control the $\dot{H}^{\frac{3}{2}} \times \dot{H}^{\frac{1}{2}}$ tails uniformly in $t \in I$. In fact, by the Arzela-Ascoli theorem,  for any $\eta_0 > 0$ there exists  $ 0 < c(\eta_0)< C(\eta_0) < \infty$   such that
\EQ{\label{small tails}
&\int_{\abs{x} \geq \frac{C(\eta)}{N(t)}}\abs{|\nabla|^{3/2} u(t,x)}^{2} dx + \int_{\abs{\xi} \geq C(\eta)N(t)} \abs{\xi}^3 \abs{\hat{u}(t,\xi)}^{2}\,  d\xi  \le \eta_0, \\
&\int_{\abs{x} \le \frac{c(\eta)}{N(t)}}\abs{|\nabla|^{3/2} u(t,x)}^{2} dx + \int_{\abs{\xi} \le c(\eta)N(t)} \abs{\xi}^3 \abs{\hat{u}(t,\xi)}^{2}\,  d\xi  \le \eta_0, \\
&\int_{\abs{x} \geq \frac{C(\eta)}{N(t)}}\abs{|\nabla|^{1/2} u_t(t,x)}^{2} dx + \int_{\abs{\xi} \geq C(\eta)N(t)} \abs{\xi} \abs{\hat{u_t}(t,\xi)}^{2} \, d\xi  \le  \eta_0, \\
&\int_{\abs{x} \le \frac{c(\eta)}{N(t)}}\abs{|\nabla|^{1/2} u_t(t,x)}^{2} dx + \int_{\abs{\xi} \le c(\eta)N(t)} \abs{\xi} \abs{\hat{u_t}(t,\xi)}^{2} \, d\xi  \le  \eta_0, 
}
for all $t \in I$. 
\end{rem}

Another standard fact about solutions to~\eqref{u eq} with the compactness property on an open interval $I = (T_-, T_+)$ is that any Strichartrz norm of the linear part of the evolution vanishes  as $t \to T_- $ and as $t \to T_+$. A concentration compactness argument then implies that the linear part of the evolution must vanish weakly in $\dot{H}^{\frac{3}{2}} \times \dot{H}^{\frac{1}{2}}$, see  \cite[Section $6$]{TVZ08}, \cite[Proposition $3.6$]{Shen}. 
This implies the following lemma, which is essential to  the  proofs of Theorem~\ref{thm: main} and Theorem~\ref{thm: wm}. 
\begin{lem}[Weak Limits] \cite[Section $6$]{TVZ08}, \cite[Proposition $3.6$]{Shen} \label{lem: weak}Let $\vec u(t)$ be a  solution to either~\eqref{u eq},~\eqref{u eq wms}, or~\eqref{u eq wmh} with the compactness property on an interval $I = (T_-, T_+)$. Then for any $t_0 \in I$ we have 
\EQ{
-\int_{t_0}^T S(t_0 - s) (0, F(u))  \, ds \rightharpoonup \vec u(t_0)  \mas T \nearrow T_+ \quad \textrm{weakly in} \, \, \dot{H}^{\frac{3}{2}} \times \dot{H}^{\frac{1}{2}}\\
+\int^{t_0}_T S(t_0 - s) (0, F(u))  \, ds \rightharpoonup \vec u(t_0)  \mas T \searrow T_- \quad \textrm{weakly in} \, \, \dot{H}^{\frac{3}{2}} \times \dot{H}^{\frac{1}{2}}
}
where here $F(u)$ denotes the nonlinearity in either~\eqref{u eq},~\eqref{u eq wms}, or~\eqref{u eq wmh}.
\end{lem} 

\subsection{Reduction to Rigidity Theorems}

The proofs of Theorem~\ref{thm: main} and Theorem~\ref{thm: wm} are now reduced to  showing that a nonzero critical element, $\vec u^*(t)$ as in Proposition~\ref{prop: ce} cannot exist. We prove the following rigidity statement. 
\begin{thm}[Rigidity Theorem]\label{thm: rig} Let $\vec u(t)$ be a solution to either~\eqref{u eq},~\eqref{u eq wms}, or~\eqref{u eq wmh}. Suppose that there exits a continuous function $N: [0, T_+( \vec u)) \to (0, \infty)$ so that the trajectory 
\EQ{
K:=   \left\{ \left(\frac{1}{N(t)} u^*\left( t, \frac{\cdot}{N(t)} \right), \, \frac{1}{N^2(t)} u_t^*\left( t, \frac{\cdot}{N(t)} \right) \right) \mid t \in I_{\max}\right\}
} 
is pre-compact in $\dot{H}^\frac{3}{2} \times \dot H^{\frac{1}{2}} (\R^3)$ and we have 
\EQ{ \label{inf N}
\inf_{t \in [0, T_+(\vec u))} N(t) >0.
}
Then $\vec u  \equiv ( 0, 0)$. 
\end{thm} 

Following the work of Kenig, Merle~\cite[Sections $5$ and $6$]{KM11a} and Duyckaerts, Kenig, and Merle~\cite[Section $2$]{DKM5}, the proof of Theorem~\ref{thm: rig} reduces to proving the following two propositions, the first where the scaling parameter $N(t)$ is bounded away from $0$ on the entire interval $I_{\max}$, as opposed to just the half-open interval $[0, T_+)$, and the second which assumes that the solution experiences self-similar finite time blow up in forward time. 
\begin{prop} \label{prop: r1}
Let $\vec u(t)$ be a solution to either~\eqref{u eq},~\eqref{u eq wms}, or~\eqref{u eq wmh} defined on an interval $I_{\max}(\vec u) = (T_-, T_+)$. Suppose that there exits a continuous function $N: I_{\max}(\vec u) \to (0, \infty)$ so that the trajectory 
\EQ{
K:=   \left\{ \left(\frac{1}{N(t)} u^*\left( t, \frac{\cdot}{N(t)} \right), \, \frac{1}{N^2(t)} u_t^*\left( t, \frac{\cdot}{N(t)} \right) \right) \mid t \in I_{\max}(\vec u))\right\}
} 
is pre-compact in $\dot{H}^\frac{3}{2} \times \dot H^{\frac{1}{2}} (\R^3)$ and we have 
\EQ{ \label{inf N1}
\inf_{t \in I_{\max}(\vec u)} N(t) >0.
}
Then $\vec u  \equiv ( 0, 0)$. 
\end{prop}


\begin{prop} \label{prop: r2} There is \textbf{no} radial solution to either~\eqref{u eq},~\eqref{u eq wms}, or~\eqref{u eq wmh} with the compactness property on $I_{\max}(\vec u)$ as in Theorem~\ref{thm: rig},  with  $T_+(\vec u)<\infty$ and with the scaling parameter given by $$N(t)  = (T_+ - t)^{-1}$$ on the half interval $[0, T_+)$. Note that in this case the trajectory 
\ant{
K_+ :=   \left\{ \left((T_+ - t) u^*\left( t, \,  (T_+ - t) \cdot  \, \right), \, (T_+ - t)^2 u_t^*\left( t, \, (T_+ - t) \cdot \,  \right) \right) \mid t \in [0, T_+(\vec u))\right\}
} 
is pre-compact in $\dot{H}^\frac{3}{2} \times \dot H^{\frac{1}{2}} (\R^3)$.
\end{prop}

\begin{rem} \label{rem: N}
We make note of the following  reductions which we will use later. \begin{enumerate}
\item Suppose that the scaling parameter $N(t)$ as in~\eqref{inf N} satisfies $N(t) \le C< \infty$ on $[0, T_+)$ (resp. $(T_-, 0])$. 
Then, a standard argument, see for example~\cite{KM11a} or~\cite[proof of Proposition~$4.4$]{L13}, shows that in fact we must have $T_+ = \infty$ (resp. $T_- = - \infty$). In fact, one can then  modify  the profile so that $N(t)  = 1$ for all $t \ge 0$, (resp. $t \le 0$). 
\item Let $\vec u(t)$ have the compactness property as in Theorem~\ref{thm: rig}, or Proposition~\ref{prop: r1},~\ref{prop: r2}. If, say, $T_+<\infty$, then without loss of generality we can assume that $\supp u(t), u_t(t) \in B(0, T_+-t)$; see for example~\cite[Lemma $4.15$]{KM11a} for more details.  
\end{enumerate}
\end{rem} 

\begin{rem}The reduction of Theorem~\ref{thm: rig} to Propositions~\ref{prop: r1},~\ref{prop: r2} is a standard argument in the field and does not depend on the dimension or the precise structure of the nonlinearity. In particular, the  argument in~\cite[Section $2$]{DKM5}, which deals with the focusing radial, supercritical semilinear wave equation in $d=3$, applies here as well. Such reductions hold as well for different equations such as the nonlinear Schr\"odinger equation, see for example \cite{KTV}, or for the gKdV equation, see for example~\cite{KKSV, D13}. 

Rather than repeat the argument that reduces Theorem~\ref{thm: rig} to the two propositions, we focus the rest of the paper on the proofs of Proposition~\ref{prop: r1}, and Proposition~\ref{prop: r2}, where several aspects of proof differ significantly from the arguments in~\cite{DKM5}. 
\end{rem}

\section{Additional regularity and decay for solutions with the Compactness Property} \label{sec: dd}
In this section we prove additional regularity and additional spacial decay for solutions to~\eqref{u eq},~\eqref{u eq wms}, and~\eqref{u eq wmh} with the compactness property on an open interval $I_{\max} \ni 0$ using the so-called Double Duhamel trick. The methods used in this section are similar to the techniques used in~\cite{DL14a} for the cubic wave equation in $\R^{1+3}$. In the first two subsections we carry out the arguments in the case of the focusing cubic equation~\eqref{u eq} as this keeps the technical difficulties at a minimum. In the last two subsections we extend the main results to solutions to~\eqref{u eq wms} and~\eqref{u eq wmh}. 

\subsection{Higher regularity for compact solutions to~\eqref{u eq}}
We begin by showing that  a solution to~\eqref{u eq} with that compactness property in $\dot{H}^{3/2} \times \dot H^{1/2}$ is in fact, more regular. In particular we prove the following estimates for the $\dot{H}^{5/2} \times \dot H^{3/2}$ norm of $\vec u(t)$. 

\begin{prop}\label{prop: u 52}  Suppose $\vec u(t)$ is a solution to~\eqref{u eq} with the compactness property on $I_{\max}(\vec u)$ as in Theorem~\ref{thm: rig}, i.e., assume that there exists a function $N: I_{\max} \to (0, \infty)$ so that the set 
\EQ{\label{1.1}
K:= \left\{ \left(\frac{1}{N(t)} u\left( t, \frac{\cdot}{N(t)} \right), \, \frac{1}{N^2(t)} u_t\left( t, \frac{\cdot}{N(t)} \right) \right) \mid t \in I\right\}
}
is pre-compact in $\dot{H}^{3/2} \times \dot H^{1/2}$. 
%
Then for all $t \in I_{\max}$,
\ant{
\| \vec u(t) \|_{\dot{H}^{5/2} \times \dot{H}^{3/2}(\R^{5})} \lesssim N(t).
}
with a constant that is uniform in $t \in I_{\max}$. 
\end{prop}

\begin{rem}
We note that all implicit constants in this section in the symbol $\lesssim$ will be allowed to depend on the $L^{\infty}_t (I_{\max}; \dot{H}^{3/2} \times \dot H^{1/2})$ norm of $\vec u$, which is bounded by a fixed constant. 
\end{rem}

Before beginning the proof of Proposition~\ref{prop: u 52} we establish a few preliminary definitions and facts. Define 
\ant{
v = u + \frac{i}{\sqrt{-\Delta}} u_{t}.
}
Note that if $u$ solves
\EQ{ \label{u F eq}
u_{tt} - \Delta u = F(u),
}
then  $v$ solves
\ant{
v_{t} = u_{t} + \frac{i}{\sqrt{-\Delta}} (\Delta u + F(u)) = -i \sqrt{-\Delta} v + \frac{i}{\sqrt{-\Delta}} F(u).
}
and we have 
\EQ{
\| \vec u(t) \|_{\dot{H}^s \times \dot{H}^{s-1}} \simeq \| v(t) \|_{\dot{H}^s}
}
By Duhamel's formula
\ant{
v(t) = e^{-it \sqrt{-\Delta}} v(0) + \int_{0}^{t} \frac{e^{-i(t - \tau) \sqrt{-\Delta}}}{\sqrt{-\Delta}}  F(u) d\tau.
}
If $\vec u(t)$ has the compactness property on $I_{\max}$ then using Lemma~\ref{lem: weak} we note that for any $t_{0} \in I_{\max}$, as 
\EQ{ \label{v weak}
\int_{t_{0}}^{T} e^{-i(t_{0} - \tau) \sqrt{-\Delta}} F(u(\tau)) d\tau \rightharpoonup v(t_{0}), \mas T \to T_-, T_+
}
weakly in $\dot{H}^{3/2}(\R^{5})$.

 We next prove a refined local estimate on the scattering norm of $\vec u(t)$. 
\begin{lem}\label{l1.1} Let $\vec u(t)$ satisfy the assumptions in Proposition~\ref{prop: u 52}. Then for any $\eta > 0$ there exists $\delta(\eta) > 0$ such that for any $t_0 \in I_{\max}$,
\ant{
\| u \|_{L_{t}^{2} L_{x}^{10}([t_{0} - \frac{\delta}{N(t_{0})}, t_{0} + \frac{\delta}{N(t_{0})}] \times \R^{5})} \lesssim \eta.
}
\end{lem}

\begin{proof} We can assume without loss of generality that $t_0 = 0$ and define the interval $J:=[- \frac{\delta}{N(t_{0})},  + \frac{\delta}{N(t_{0})}]$. By Duhamel's formula
\EQ{ \label{str 1}
 \|u(t) \|_{L^2(J; L^{10})} \le  \| S(t) \vec u(0) \|_{L^2(J; L^{10})} +   \left\| \int_0^t S(t-s)(0, u^3) \, ds \right\|_{L^2(J; L^{10})}
}
We begin by estimating the first term on the right-hand-side of~\eqref{str 1}.  Choose $C(\eta)$ as in Remark~\ref{Ceta}, so that 
\EQ{ \label{P C}
\| P_{\geq C(\eta) N(0)} \vec u(0) \|_{\dot{H}^{3/2} \times \dot{H}^{1/2}(\R^{5})}  \le \eta.
}
By compactness $C(\eta)$ above can be chosen uniformly in $t \in I_{\max}$, which is the reason why it suffices to just consider $t_0=0$ in this argument. Next, we have
\begin{multline*}
\|S(t) \vec u(0) \|_{L^2_t(J; L^{10}_x)} \lesssim \\ \|S(t) P_{\ge C(\eta) N(0)}\vec u(0) \|_{L^2_t(J; L^{10}_x)} + \|S(t) P_{\le C(\eta) N(0)}\vec u(0) \|_{L^2_t(J; L^{10}_x)}
\end{multline*}
We use~\eqref{P C} together with Strichartz estimates to handle the first term on the right-hand-side above: 
\ant{
\|S(t) P_{\ge C(\eta) N(0)}\vec u(0) \|_{L^2_t(J; L^{10}_x)} \lesssim \| P_{\geq C(\eta) N(0)} \vec u(0) \|_{\dot{H}^{3/2} \times \dot{H}^{1/2}}  \lesssim \eta.
}
To control the second term we use Bernstein's inequalities,~\eqref{bern} and Sobolev embedding, 
\begin{equation}\label{bern1}
\| P_{\leq C(\eta) N(0)} S(t) \vec u(0) \|_{L_{x}^{10}(\R^{5})} \lesssim C(\eta)^{1/2} N(0)^{1/2}
\end{equation}
Taking the $L^2_t(J)$ norm above yields 
\ant{
\|S(t) P_{\le C(\eta) N(0)}\vec u(0) \|_{L^2_t(J; L^{10}_x)} \lesssim C(\eta)^{\frac{1}{2}} \de^{\frac{1}{2}} 
}
Next, we use Strichartz estimates on the second term on the right-hand-side of~\eqref{str 1}. 
\ant{
 \left\|  \int_0^t S(t-s) \, (0, \pm u^3) \, ds \right\|_{L^2_t(J; L^{10}_x)} \lesssim \| (D^{\frac{1}{2}}u) u^2\|_{L^1_t(J;  L^2)} \lesssim \| u\|^2_{L^2_t(J; L^{10}_x)}
 }
Combining all of the above we obtain, 
\begin{equation}\label{4.49}
\| u \|_{L^2_t(J; L^{10}_x)} \lesssim \eta + C(\eta)^{\frac{1}{2}}\de^{\frac{1}{2}} + \| u\|^2_{L^2_t(J; L^{10}_x)}
\end{equation}
The proof is concluded using the usual continuity argument after taking $\de$ small enough. 
\end{proof}

We can now begin the proof of Proposition~\ref{prop: u 52}. 

\begin{proof}[Proof of Proposition~\ref{prop: u 52}] Using compactness, we can assume, without loss of generality that $t_0= 0$. We prove Proposition~\ref{prop: u 52} by finding a frequency envelope $\al_k(0)$ so that 
\EQ{ \label{f e}
&\| P_k  \vec u(0)\|_{\dot{H}^{5/2} \times \dot H^{3/2}} \lesssim 2^{k}  \al_k(0)\\
&\| \{2^k \al_k(0)\}_{k \in \Z} \|_{\ell^2} \lesssim N(0)
}
We note that finding $\al_{k}(0)$ as above proves Proposition~\ref{prop: u 52} in light of Defninition~\ref{def: freq en}. 
\begin{claim}\label{l1.2} Let $\eta>0$ be a small number and let $J:= [ -  \de/ N(0), + \de/ N(0)]$ where $\de  = \de (\eta)$ is as in Lemma~\ref{l1.1}. Define 
\ant{
&a_k : = 2^{\frac{3k}{2}} \| P_k u\|_{L^{\infty}(J; L^2)} + 2^{\frac{k}{2}} \| P_k u_t\|_{L^{\infty}(J; L^2)} \\
& a_{k}(0) := 2^{\frac{3k}{2}} \| P_k u(0)\|_{ L^2} + 2^{\frac{k}{2}} \| P_k u_t(0)\|_{ L^2} 
}
Define the frequency envelopes
\ant{
 \al_k: =  \sum_{j} 2^{-\frac{5}{4} |j - k|} a_j, \quad \al_k(0):= \sum_{j} 2^{-\frac{5}{4} |j - k|} a_j(0)
 }
 Then,  
 \EQ{ \label{a k 0}
 a_k \lesssim a_k(0) + \eta^2 \sum_{j \ge k-3} 2^{3(k-j)/2} a_j
 }
 and $\eta>0$ can be chosen small enough so that 
\begin{equation}\label{1.14}
\alpha_{k} \lesssim \alpha_{k}(0).
\end{equation}
\end{claim}
\begin{proof}

First we observe that after localizing to frequency $k$, Strichartz estimates along with Lemma $\ref{l1.1}$ give  
\EQ{\label{a k}
a_k  
& \lesssim 2^{3k/2} \| P_{k} u(0) \|_{L_{x}^{2}}  + 2^{k/2} \| P_{k} u_{t}(0) \|_{L_{x}^{2}} +2^{\frac{k}{2}} \| P_{k}(u^{3}) \|_{L_{t}^{1} L_{x}^{2}(J)} \\
& \lesssim a_{k}(0) + \eta^{2} \sum_{j \geq k - 3} 2^{3(k - j)/2} a_{j}.
}
Indeed, to prove the last line above we observe that it suffices to show that 
\EQ{ \label{cubic lp}
2^{\frac{k}{2}}  \| P_{k}(u^{3}) \|_{ L_{x}^{2}} \lesssim 2^{\frac{3k}{2}} \| u\|_{L^{10}}^2 \sum_{j \ge k-3} \| P_{j} u\|_{L^2}
}
First, noting that since $P_{k}((P_{\leq k - 4} u)^{3}) = 0$,  we have 
\ant{
\|P_k u^3\|_{L^{2}_x} &\lesssim  \| P_k [(P_{\le k-4} u)^2 P_{\ge k-3} u]\|_{L^2} +\| P_k [(P_{\le k-4} u (P_{\ge k-3} u)^2]\|_{L^2} \\
& \quad +\| P_k [ P_{\ge k-3} u]^3\|_{L^2}
}
We estimate the terms on the right-hand side above as follows: First by Young's inequality and Bernstein's inequality, 
\ant{
 \| P_k [(P_{\le k-4} u)^2 P_{\ge k-3} u]\|_{L^2} &\lesssim \| (P_{\le k-4} u)^2 P_{\ge k-3} u \|_{L^2} \\
 & \lesssim  \| P_{\le k-4} u\|_{L^{\infty}}^2   \sum_{j \ge k-3} \| P_j u\|_{L^2}   \\
& \lesssim 2^k  \| u\|_{L^{10}}^2 \sum_{j \ge k-3} \| P_{j} u\|_{L^2}
 }
 Next, again using Young's inequality and the fact that $P_k$ is given by convolution with  $ \check \phi_k( \cdot):= 2^{5k}  \check  \phi(  2^k \cdot)$ where $\check \phi \in \cS$, we have 
  \ant{
  \| P_k [(P_{\le k-4} u (P_{\ge k-3} u)^2]\|_{L^2}& \lesssim   \| P_{\le k-4} u \|_{L^{\infty}} \| \check \phi_k\|_{L^{\frac{10}{9}}} \| (P_{\ge k-3} u)^2\|_{L^{\frac{5}{3}}} \\
  & \lesssim 2^k \| u\|_{L^{10}}^2 \sum_{j \ge k-3} \| P_{j} u\|_{L^2}
  }
  Finally, arguing similarly for the last term we have 
  \ant{
  \| P_k [ P_{\ge k-3} u]^3\|_{L^2}& \lesssim  \| \check \phi_k\|_{L^{\frac{5}{4}}} \| (P_{\ge k-3} u)^2 P_{\ge k-3} u\|_{L^{\frac{10}{7}}} \\
  & \lesssim 2^k  \| u\|_{L^{10}}^2 \sum_{j \ge k-3} \| P_{j} u\|_{L^2}
  }
  This proves~\eqref{cubic lp} as desired, which implies~\eqref{a k 0}.   To establish~\eqref{1.14} we use~\eqref{a k 0} to obtain
\EQ{\label{al k sum}
\sum_j 2^{-\frac{5}{4}\abs{j - k}} a_{j} \lesssim \sum_j a_{j}(0) 2^{-\frac{5}{4} |j - k|} + \eta^{2} \sum_{j} 2^{-\frac{5}{4} |j - k|} \sum_{j_{1} \ge j - 3} 2^{3(j - j_{1})/2} a_{j_{1}}.
}
Reversing the order of summation in the second term above we have
\EQ{ \label{al k 0}
&\sum_{j_{1} \leq k} \sum_{j \le j_{1} + 3} 2^{\frac{3}{2}(j - j_{1})} 2^{\frac{5}{4} (j - k)} a_{j_{1}} \lesssim \sum_{j_{1} \leq k} 2^{\frac{5}{4} (j_{1} - k)} a_{j_{1}} \lesssim \alpha_{k},\\
&\sum_{j_{1} > k} \sum_{j \leq j_{1} + 3} 2^{\frac{3}{2}(j - j_{1})} 2^{-\frac{5}{4} \abs{j - k}} a_{j_{1}} \lesssim \sum_{j_{1} > k} (2^{-\frac{3}{2} ( j_{1} - k)} + 2^{-\frac{5}{4} ( j_{1}-k)}) a_{j_{1}} \lesssim \alpha_{k}.
}
Therefore,~\eqref{al k sum} implies that
\ant{
\alpha_{k} \lesssim \alpha_{k}(0) + \eta^{2} \alpha_{k},
}
which in turn yields~\eqref{1.14} once $\eta>0$ is chosen small enough. 
\end{proof}

 Now we are ready to use the double Duhamel trick to prove Proposition~\ref{prop: u 52}. 
 Let $\langle \cdot, \cdot \rangle_{\dot{H}^{3/2}}$ denote the inner product with 
\ant{
\langle v, v \rangle_{\dot{H}^{3/2}} = \| v \|_{\dot{H}^{3/2}}^{2}.
}
Again, recall that we can assume  without loss of generality that $t_{0} = 0$. For any $k$, and for any $T_1<T_+$ we have
\begin{multline*}
\ang{P_{k} v(0), P_{k} v(0)}_{\dot{H}^{3/2}} \\=  \Bigg\langle P_{k}\Big(e^{-i T_{1} \sqrt{-\Delta}} v(T_{1}) + \int_{0}^{T_{1}} \frac{e^{-i\tau \sqrt{-\Delta}}}{\sqrt{-\Delta}} F(u(\tau)) d\tau\Big), P_{k} v(0) \Bigg\rangle_{\dot{H}^{3/2}} \\
= \lim_{T_{1} \rightarrow T_+} \left\langle P_{k}\left(\int_{0}^{T_{1}} \frac{e^{-i \tau \sqrt{-\Delta}}}{\sqrt{-\Delta}} F(u(\tau)) d\tau\right), P_{k} v(0) \right\rangle_{\dot{H}^{3/2}},
\end{multline*}
where $F(u)  = u^3$. Then for any $T_- < T_{2} < 0$, this further reduces as 
\begin{multline*}
= \left\langle P_{k} v(0), P_{k} e^{-iT_{2} \sqrt{-\Delta}} v(T_{2}) \right\rangle_{\dot{H}^{3/2}} \\ 
+ \lim_{T_{1} \rightarrow T_+} \left\langle P_{k}\left(\int_{0}^{T_{1}} \frac{e^{-i \tau \sqrt{-\Delta}}}{\sqrt{-\Delta}} F(u(\tau)) d\tau\right), P_{k}\left(\int_{T_{2}}^{0} \frac{e^{-it \sqrt{-\Delta}}}{\sqrt{-\Delta}} F(u(t)) dt\right) \right\rangle_{\dot{H}^{3/2}}
\end{multline*}
Taking the limit $T_{2} \rightarrow T_-$, for any integer $k$, yields
\begin{multline*}
\left\langle P_{k} v(0), P_{k} v(0)\right\rangle_{\dot{H}^{3/2}} = \\ =    \left\langle P_{k}\left(\int_{0}^{T_+}\frac{e^{-i\tau \sqrt{-\Delta}}}{\sqrt{-\Delta}} F(u(\tau))d \tau\right) , \,  P_{k} \left(\int_{T_-}^{0}\frac{e^{-it \sqrt{-\Delta}}}{\sqrt{-\Delta}} F(u(t))dt \right) \right\rangle_{\dot{H}^{3/2}}.
\end{multline*}

Now let $\chi$ be a smooth, radial, non-increasing function, $\chi(x) = 1$ when $|x| \leq 1$, $\chi(x) = 0$ when $|x| \geq 2$. Let $c > 0$ be a small fixed constant, say $c = \frac{1}{4}$. 
We rewrite the inner product above as 
\EQ{ \label{A B} 
\langle A + B, A' + B' \rangle_{\dot{H}^{3/2}} &= \langle A, A' + B' \rangle_{\dot{H}^{3/2}} + \langle A + B, A' \rangle_{\dot{H}^{3/2}} \\
& \quad - \langle A, A' \rangle_{\dot{H}^{3/2}} + \langle B, B' \rangle_{\dot{H}^{3/2}},
}
 where
\ant{
&A := P_{k}\left( \int_{0}^{\frac{\delta}{N(0)}} \frac{e^{-it \sqrt{-\Delta}}}{\sqrt{-\Delta}} u^{3} dt\right) + P_{k}\left( \int_{\frac{\delta}{N(0)}}^{T_+} \frac{e^{-it \sqrt{-\Delta}}}{\sqrt{-\Delta}} (1 - \chi)(\frac{x}{ct}) u^{3} dt \right), \\
&B := P_{k}\left( \int_{\frac{\delta}{N(0)}}^{T_+} \frac{e^{-it \sqrt{-\Delta}}}{\sqrt{-\Delta}} \chi (\frac{x}{ct}) u^{3} dt\right).
}
and  $A'$, $B'$ are the corresponding integrals in the negative time direction.

First we estimate the term $ \ang{A, A'}_{\dot{H}^{3/2}} \le \| A \|_{\dot{H}^{3/2}} \|A'\|_{\dot{H}^{3/2}}$. To control the terms in the right-hand side in the inequality in the preceding line we observe first that by the argument in Claim~\ref{l1.2} we have 
\EQ{ \label{0 to N}
 \left\| P_{k}\left( \int_{0}^{\frac{\delta}{N(0)}} \frac{e^{-it \sqrt{-\Delta}}}{\sqrt{-\Delta}} u^{3} dt\right)  \right \|_{\dot{H}^{3/2}} \lesssim \eta^2 \sum_{j \ge 3} 2^{\frac{3}{2}(k-j)} a_j.
 }
 Next, we prove the following claim: 
 \begin{claim} \label{claim: N to inf} There exists a sequence $b_{k} \in \ell^2$, so that
 \begin{equation}\label{1.39}
\left\| P_{k} \left( \int_{\frac{\delta}{N(0)}}^{T_+} \frac{e^{-i t \sqrt{-\Delta}}}{\sqrt{-\Delta}} (1 - \chi)(\frac{x}{ct}) u^{3} \, dt  \right)\right\|_{\dot{H}^{3/2}} \lesssim 2^{-k} b_{k}.
\end{equation}
and 
\begin{equation}\label{bj}
\| b_{k} \|_{l^{2}} \lesssim \frac{N(0)}{c^{2} \delta} \| u \|_{L_{t}^{\infty} \dot{H}^{3/2}}^{3},
\end{equation}
 \end{claim} 
\begin{proof}[Proof of Claim~\ref{claim: N to inf}]
Note that it suffices to show that 
\EQ{ \label{bk suf}
\left\|  \int_{\frac{\delta}{N(0)}}^{T_+} \frac{e^{-i t \sqrt{-\Delta}}}{\sqrt{-\Delta}} (1 - \chi)(\frac{x}{ct}) u^{3} \, dt  \right\|_{\dot{H}^{5/2}} \lesssim \frac{N(0)}{c^{2} \delta} \| u \|_{L_{t}^{\infty} \dot{H}^{3/2}}^{3}.
}
To see this we begin by observing that 
\EQ{ \label{N to T+}
\Bigg\|   \int_{\frac{\delta}{N(0)}}^{T_+} &\frac{e^{-i t \sqrt{-\Delta}}}{\sqrt{-\Delta}} (1 - \chi)(\frac{x}{ct}) u^{3} \, dt  \Bigg\|_{\dot{H}^{5/2}}  
\lesssim \int_{\frac{\delta}{N(0)}}^{\infty}  \| (1-  \chi)( \frac{x}{ ct}) u^3 \|_{\dot{H}^{\frac{3}{2}}} \, dt \\
& \lesssim  \int_{\frac{\delta}{N(0)}}^{\infty}  \bigg[ \frac{1}{c t} \| \chi'( \frac{x}{ct}) u^3 \|_{\dot{H}^{\frac{1}{2}}} +  \| (1- \chi)( \frac{x}{ct}) \na u u^2 \|_{\dot{H}^{\frac{1}{2}}} \bigg] \, dt
}
Next,  using Sobolev embedding we have 
\ant{
\frac{1}{c t} \| \chi'( \frac{x}{ct}) u^3 \|_{\dot{H}^{\frac{1}{2}}} &\lesssim   \frac{1}{ct} \| \chi'( \frac{x}{ct}) u^3 \|_{\dot{W}^{1, \frac{5}{3}}} \\
& \lesssim  \frac{1}{(ct)^2} \| u \|_{L^5}^3  +  \frac{1}{ct} \| \chi'( \frac{x}{ct})  \na u u^2 \|_{L^{ \frac{5}{3}}}   \\
& \lesssim \frac{1}{(ct)^2}   \| u \|_{\dot{H}^{\frac{3}{2}}}^3  +  \frac{1}{ct} \| \na u\|_{L^{\frac{5}{2}}}  \| u \|_{L^{5}} \| u \|_{L^{\infty}( \abs{x} \ge ct)}\\
& \lesssim \frac{1}{(ct)^2}   \| u \|_{\dot{H}^{\frac{3}{2}}}^3
}
where in the last line above we have used   Lemma~\ref{lem: rad se}, i.e.,  
\begin{equation}\label{1.32}
\| u \|_{L_{x}^{\infty}(|x| \geq R)} \lesssim R^{-1} \| u \|_{\dot{H}^{3/2}(\R^{5})}.
\end{equation}
By the fractional product rule 
\EQ{ \label{f p r}
\| (1 - \chi)(\frac{x}{t c}) (\nabla u) u^{2} \|_{\dot{H}^{1/2}}& \lesssim \| \nabla u \|_{L_{x}^{5/2}} \| (1 - \chi)(\frac{x}{tc}) u^{2} \|_{\dot{W}^{1/2, 10}}  \\
& \quad \, \, + \| u \|_{\dot{H}^{3/2}} \| (1 - \chi)(\frac{x}{tc}) u^{2} \|_{L_{x}^{\infty}}.
}
Again, by Sobolev embedding  and Lemma~\ref{lem: rad se} we have 
\ant{
\| (1 - \chi)(\frac{x}{ct}) u^{2} \|_{\dot{W}^{1/2, 10}} &\lesssim \| (1 - \chi)(\frac{x}{ct}) u^{2} \|_{\dot{W}^{1, 5}} \\
& \lesssim  \frac{1}{ct} \| u\|_{L^{\infty}( \abs{x \ge ct})} \| u \|_{L^5} +  \| (1 - \chi)(\frac{x}{ct}) u  \na u\|_{L^5}\\
& \lesssim \frac{1}{(ct)^2} \| u \|_{\dot{H}^{\frac{3}{2}}} +  \| (1 - \chi)(\frac{x}{ct}) u \|_{L^{10}} \| (1 - \chi)(\frac{x}{ct})  \na u\|_{L^{10}} \\
& \lesssim \frac{1}{(ct)^2} \| u \|_{\dot{H}^{\frac{3}{2}}}
}
We also can use Lemma~\ref{1.32} to deduce that 
\ant{
 \| (1 - \chi)(\frac{x}{tc}) u^{2} \|_{L_{x}^{\infty}} \lesssim \frac{1}{(ct)^2} \| u \|_{\dot{H}^{\frac{3}{2}}}^2
 }
Therefore we can estimate~\eqref{f p r} by 
\EQ{
\| (1 - \chi)(\frac{x}{t c}) (\nabla u) u^{2} \|_{\dot{H}^{1/2}} \lesssim  \frac{1}{(ct)^2} \| u \|_{\dot{H}^{\frac{3}{2}}}^3
}
Plugging the preceding estimates into~\eqref{N to T+} gives 
\EQ{ \label{N to T+2}
\Bigg\|   \int_{\frac{\delta}{N(0)}}^{T_+} &\frac{e^{-i t \sqrt{-\Delta}}}{\sqrt{-\Delta}} (1 - \chi)(\frac{x}{ct}) u^{3} \, dt  \Bigg\|_{\dot{H}^{5/2}}   \lesssim  \int_{\frac{ \de}{N(0)}}^{\infty} \frac{1}{(ct)^2}\| u \|_{\dot{H}^{\frac{3}{2}}}^3 \, dt 
} 
from which~\eqref{bk suf} is immediate. 
\end{proof} 
Combining~\eqref{0 to N} and~\eqref{1.39} and recalling the definition of $A$ yields 
\ant{
\| A \|_{\dot{H}^{3/2}} \lesssim 2^{-k} b_{k} + \eta^{2} \sum_{j \geq k - 3} a_{j} 2^{3(k - j)/2}.
}
Estimating $A'$ in an identical manner then yields
\EQ{\label{A A'}
\langle A, A' \rangle_{\dot{H}^{\frac{3}{2}}} \lesssim \left(2^{-k} b_{k} + \eta^{2} \sum_{j \geq k - 3} a_{j} 2^{3(k - j)/2}\right)^{2}.
}
Next using again the fact that $e^{-it \sqrt{-\Delta}} v(t) \rightharpoonup 0$ weakly in $\dot{H}^{3/2}$ as $t \nearrow T_+$, $t \searrow T_-$, as well as~\eqref{0 to N} and~\eqref{1.39} we have
\EQ{ \label{AB}
\langle A + B, A' \rangle + \langle A, A' + B' \rangle \lesssim a_{k}(0) \left(2^{-k} b_{k} + \eta^{2} \sum_{j \geq k - 3} a_{j} 2^{3(k - j)/2}\right).
}
Finally it remains to estimate  $\langle B, B' \rangle$ which is given by 

\begin{align}
&\langle B, B' \rangle_{\dot{H}^{3/2}} =   \notag\\
&= \int_{T_-}^{\frac{-\delta}{ N(0)}} \int_{\frac{\delta}{N(0)}}^{T_+} \ang{ P_{k}\bigg(\frac{e^{-it \sqrt{-\Delta}}}{\sqrt{-\Delta}} \chi(\frac{\cdot}{ct}) u^{3}(t)\bigg), P_{k} \bigg(\frac{e^{-i \tau \sqrt{-\Delta}}}{\sqrt{-\Delta}} \chi(\frac{\cdot}{c\tau}) u^{3}(\tau) \bigg) }_{\dot{H}^{\frac{3}{2}}} dt d\tau \notag \\
&= \int_{-T_-}^{\frac{-\delta}{N(0)}} \int_{\frac{\delta}{N(0)}}^{T_+} \ang{  \chi(\frac{x}{ct}) u^{3}(t), P_{k}^2 \bigg(e^{i(t - \tau) \sqrt{-\Delta}}\sqrt{-\Delta} \chi(\frac{\cdot}{c\tau}) u^{3}(\tau) \bigg) }_{L^2} dt d\tau. \label{b b'}
\end{align}
Here we use an argument based on the sharp Huygens principle, see for example~\cite[Section 4]{DL14a}.  The kernel $K_k( \cdot)$ of the operator $ P_k e^{i(t-\tau) \sqrt{- \De}}  \sqrt{- \De}$ is given by 
\EQ{
K_k(x) = K_k( \abs{x}) := c 2^k \int_0^{\pi} \int_0^{\infty} e^{i \abs{x} \rho \cos \te} e^{i  \rho (t- \tau)} \phi( \frac{\rho}{2^k}) \rho^5 \, d \rho \sin^3 \te \, d \te
}
where the integrand is written in polar coordinates on $\R^5$ where $ \rho = \abs{ \xi}$,  $\te$ is the azimuthal angle on $\Sp^4$, and $\phi( \cdot/ 2^k)$ is the Fourier multiplier for  the $k$th Littlewood-Paley projection $P_k$. Integration by parts  $L \in \N$ times in $\rho$ yields the estimate 
\ant{
\abs{K_k( \abs{x-y})} \lesssim_L  \frac{2^{6k}}{\ang{2^k \abs{ ( \tau- t) - \abs{x-y}}}^L}
}
where here recall that $ \tau >0$ and $t <0$. Note that in~\eqref{b b'} we have $\abs{x} \le \frac{1}{4} \abs{t}$  and $ \abs{y} \le \frac{1}{4} \abs{\tau}$ which means that $\abs{x-y} \le \frac{1}{4}\abs{t- \tau}$. Therefore we have 
\ant{
(\tau- t) - \abs{x-y} \ge \frac{1}{2} \abs{\tau- t}
}
which yields the estimate 
\EQ{ \label{L}
\abs{K_k( \abs{x-y})} \lesssim_L  \frac{2^{6k}}{\ang{2^k \abs{ \tau- t }}^L}
}
in the relevant region of integration. 
%
%
%
First, if $N(0) \ll 2^k$, we use~\eqref{L} with $L = 9$ which yields, 
\EQ{ \label{N<2k}
& \ang{  \chi(\frac{x}{ct}) u^{3}(t), P_{k}^2 \bigg(e^{i(t - \tau) \sqrt{-\Delta}}\sqrt{-\Delta} \chi(\frac{\cdot}{c\tau}) u^{3}(\tau) \bigg) }_{L^2}  \\
& \lesssim  \int_0^{c \abs{t}} u^3(t, x) \int_0^{c \abs{\tau}} \abs{K_k(x - y)} u^3(\tau, y) \, dy \, dx \\
& \lesssim 2^{-3k} \frac{ \abs{\tau}^2 \abs{t}^2}{ \abs{ \tau-t}^9} \| u\|_{L^{\infty}_tL^5_x}^6 \lesssim  \frac{ 2^{-3k}}{ \abs{ \tau- t}^5} \| u\|_{L^{\infty}_t \dot{H}^{\frac{3}{2}}_x}^6
}
Integrating~\eqref{N<2k} in $\tau$ and in $t $ as in~\eqref{b b'} gives 
\EQ{ \label{le 2k}
 \ang{B, B'}_{\dot{H}^{\frac{3}{2}}} &\lesssim \int_{-T_-}^{\frac{-\delta}{N(0)}} \int_{\frac{\delta}{N(0)}}^{T_+}\frac{ 2^{-3k}}{ \abs{ \tau- t}^5}  \, d \tau \, dt \\
 & \lesssim 2^{-3k} N(0)^3 \lesssim 2^{-\frac{9}{4} k} N(0)^{\frac{9}{4}} \mif N(0) \ll 2^k
 }
where in the last inequality above we have used that $N(0) \ll 2^{k}$. Next suppose that $2^k \lesssim N(0)$. In the region $\abs{ \tau- t} \le 2^{-k}$ we used the crude estimate $ \abs{K_k( x-y)} \lesssim 2^{6k}$ and in the region $\abs{\tau-t} \ge 2^{-k}$ we use~\eqref{L} with $L = 7$ to obtain, via the same argument as above, that 
\EQ{ \label{ge 2k}
\ang{B, B'}_{\dot{H}^{\frac{3}{2}}} \lesssim 1 \mif N(0) \gtrsim 2^k
}
Putting~\eqref{le 2k} and~\eqref{ge 2k} together gives 
\EQ{ \label{b final}
\ang{B, B'}_{\dot{H}^{\frac{3}{2}}} \lesssim  \min(2^{-\frac{9}{4}k} N(0)^{\frac{9}{4}} \, , \, 1)
}

Combining~\eqref{A A'}, \eqref{AB}, and~\eqref{b final} gives 
\ant{
a_k^2(0) &\lesssim  \bigg(2^{-k} b_{k} + \eta^{2} \sum_{j \geq k - 3}  2^{3(k - j)/2}a_j\bigg)^{2} + a_k(0) \bigg(2^{-k} b_{k} + \eta^{2} \sum_{j \geq k - 3} a_{j} 2^{3(k - j)/2}\bigg) \\
& \quad  \min(2^{-\frac{9}{4}k} N(0)^{\frac{9}{4}} \, , \, 1)
}
which implies that 
\EQ{
a_k(0) \lesssim 2^{-k} b_{k} + \eta^{2} \sum_{j \geq k - 3}  2^{3(k - j)/2} a_j+ \min(2^{-\frac{9}{8}k} N(0)^{\frac{9}{8}} \, , \, 1)
}
Recalling the definitions of the envelopes $\al_k(0)$ and $\al_k$ we have 
\EQ{
 \al_k(0) \lesssim  \eta^2 \al_k +  \sum_{j} 2^{-\frac{5}{4} \abs{j-k}} 2^{-j} b_j + \sum_j  2^{-\frac{5}{4} \abs{j-k}} 2^{-j}\min(2^{-\frac{1}{8}j} N(0)^{\frac{9}{8}} \, , \, 2^j)
}
%
%
%
Next,  using~\eqref{1.14} to control $\al_k  \lesssim \al_k(0)$, and taking $\eta>0$ small enough we obtain 
\EQ{
 \al_k(0) \lesssim \sum_{j} 2^{-\frac{5}{4} \abs{j-k}} 2^{-j} b_j + \sum_j  2^{-\frac{5}{4} \abs{j-k}} 2^{-j}c_j
}
where $c_j:= \min(2^{-\frac{1}{8}j} N(0)^{\frac{9}{8}} \, , \, 2^j)$ satisfies 
\EQ{ \label{cj}
 \| \{c_j\} \|_{\ell^2} \lesssim N(0)
}
Finally, we use Schur's test along with~\eqref{bj},~\eqref{cj} to deduce that 
\EQ{
\| \{2^k \al_k\} \|_{\ell^2} \lesssim N(0)
}
as desired. This completes the proof of Proposition~\ref{prop: u 52}
\end{proof}
%
%
%
%
%
%
%

\subsection{Improved uniform spacial decay for compact solutions to~\eqref{u eq}}
\noindent In this section we prove that in the case that $\inf_{t \in I}N(t)>0$ then a solution $\vec u(t)$ to~\eqref{u eq} with the compactness property on $I$ has uniform spacial decay that ``breaks the scaling." In particular, we show that $\vec u(t)$ must be uniformly bounded in $\dot{H}^{\frac{3}{4}} \times \dot{H}^{-\frac{1}{4}}({\R}^{5})$, which in turn implies by Lemma~\ref{lem: rad se}  that $\abs{ r^{7/4}u(t, r)} \lesssim 1$. We will not use this pointwise estimate in particular, but rather the fact that boundedness in $\dot{H}^{\frac{3}{4}} \times \dot{H}^{-\frac{1}{4}}({\R}^{5})$ and compactness in $\dot{H}^{\frac{3}{2}} \times \dot{H}^{\frac{1}{2}}({\R}^{5})$ will yield compactness in the energy space $\dot{H}^1 \times L^2$. This last point will be crucial in Section~\ref{sec: rig}.  
\begin{prop}\label{prop: u 34}
 Let $\vec u(t)$ be a solution to~\eqref{u eq} with the Compactness Property as in Proposition~\ref{prop: r1}, i.e, suppose that the scaling parameter $N(t)$ satisfies 
\ant{
 \inf_{t \in I_{\max}(\vec u)} N(t) >0
 }
 then for all $t \in I_{\max}$,
\begin{equation}\label{u 34}
\| \vec u(t) \|_{\dot{H}^{\frac{3}{4}} \times \dot{H}^{-\frac{1}{4}}({\R}^{5})} \lesssim 1.
\end{equation}
with a constant that is uniform in $t \in I_{\max}$. 
\end{prop}

\begin{proof}  Without loss of generality,  it will again suffice to consider only the case $t=0$.  We will prove Proposition~\ref{prop: u 34} by finding a frequency envelope $\al_k = \al_k(0)$ so that 
\EQ{
&\| (P_k u(0), P_k u_t(0))\|_{\dot{H}^{\frac{3}{4}} \times \dot{H}^{-\frac{1}{4}}} \lesssim 2^{-\frac{3k}{4}} \al_k\\
& \| \{2^{-\frac{3k}{4}} \al_k\}_{k \in \Z} \|_{\ell^2} \lesssim 1
}
Note that for $k \ge 0$ it suffices, by the uniform boundedness of the $\dot{H}^{\frac{3}{2}} \times \dot H^{\frac{1}{2}}$ norm of $\vec u(t)$ to take 
\ant{
\al_k := 1 \mfor k  \ge 0
}
Now, for each $j$ define 
\EQ{
a_j:= 2^{3j/2} \|P_j u\|_{L^{\infty}_t L^2_x}  + 2^{j/2}\|P_j u_t\|_{L^{\infty}_t L^2_x}
}
and for $k < 0$ set 
\EQ{
\al_k:= \sum_j 2^{-\abs{j-k}} a_j \mfor k <0
}
 As in the previous subsection  let $v = u + \frac{i}{\sqrt{-\Delta}} u_{t}$. Then $v$ solves the equation
\ant{
v_{t} = u_{t} + \frac{i}{\sqrt{-\Delta}} u_{tt} = i u_{t} + \frac{i}{\sqrt{-\Delta}} (\Delta u + u^{3}) = -i \sqrt{-\Delta} v + \frac{i}{\sqrt{-\Delta}} u^{3}.
}
and we have 
\EQ{
\|P_j v\|_{L^{\infty}_t \dot{H}^{\frac{3}{2}}}  \simeq a_j
}
Next, note that we can also assume without loss of generality that $N(t) \ge 1$ for all $t \in I_{\max}$. By Lemma~\ref{lem: weak} we have the weak limits
\ant{
\lim_{t \nearrow T_+, t \searrow T_-} e^{i t \sqrt{-\Delta}} v(t) \rightharpoonup 0,
}
weakly in $\dot{H}^{3/2}$. As in the previous subsection we thus obtain the reduction 
\begin{multline*}
\left\langle P_{M} v(0), P_{M} v(0)\right\rangle_{\dot{H}^{3/2}} = \\ =    \left\langle P_{M}\left(\int_{0}^{T_+}\frac{e^{-i\tau \sqrt{-\Delta}}}{\sqrt{-\Delta}} u^3(\tau)d \tau\right) , \,  P_{M} \left(\int_{T_-}^{0}\frac{e^{-it \sqrt{-\Delta}}}{\sqrt{-\Delta}} u^3(t)dt \right) \right\rangle_{\dot{H}^{3/2}}.
\end{multline*}
for any fixed frequency $M$. As we are only concerned with low frequencies in what follows $M <1$ and later we will write $M = 2^k$ for $k <0$.  

Again as in the previous subsection we let $\chi$ be a smooth, radial, non-increasing function, $\chi(x) = 1$ when $|x| \leq 1$, $\chi(x) = 0$ when $|x| \geq 2$. Let $c > 0$ be a small fixed constant, say $c = \frac{1}{4}$. 
We rewrite the inner product above as 
\EQ{ \label{A B 2} 
\langle A + B, A' + B' \rangle_{\dot{H}^{3/2}} &= \langle A, A' + B' \rangle_{\dot{H}^{3/2}} + \langle A + B, A' \rangle_{\dot{H}^{3/2}} \\
& \quad - \langle A, A' \rangle_{\dot{H}^{3/2}} + \langle B, B' \rangle_{\dot{H}^{3/2}},
}
 where
\ant{
&A :=  \int_{0}^{CM^{-1}} \frac{e^{-it \sqrt{-\Delta}}}{\sqrt{-\Delta}} P_M(u^{3}(t)) dt + \int_{CM^{-1}}^{T_+} \frac{e^{-it \sqrt{-\Delta}}}{\sqrt{-\Delta}} (1 - \chi)(\frac{x}{ct}) P_M(u^{3}(t)) dt , \\
&B :=  \int_{CM^{-1}}^{T_+} \frac{e^{-it \sqrt{-\Delta}}}{\sqrt{-\Delta}} \chi (\frac{x}{ct}) P_M (u^{3}(t)) dt.
}
and  $C$ is will be a fixed constant to be determined below. Again,  $A'$, $B'$ are the corresponding integrals in the negative time direction.

First we estimate the term $ \ang{A, A'}_{\dot{H}^{3/2}} \le \| A \|_{\dot{H}^{3/2}} \|A'\|_{\dot{H}^{3/2}}$.  First,  we have 
\begin{align} \notag
 \left\| P_M \left(\int_{0}^{CM^{-1}}\frac{e^{-i\tau \sqrt{-\Delta}}}{\sqrt{-\Delta}} u^3(\tau)d \tau\right) \right\|_{\dot{H}^{\frac{3}{2}}} &\lesssim M^{\frac{1}{2}}\left\| \int_0^{CM^{-1}}  e^{-i\tau \sqrt{-\Delta}}P_M u^3(\tau) \, d \tau\right\|_{L^2}  \\
 & \lesssim CM^{-\frac{1}{2}} \| P_M u^3 \|_{L^{\infty}_t L^2_x} \label{A est}
 \end{align}
 We also note  that by the compactness of the set $K$ as in Proposition~\ref{prop: r1} and using the fact that $N(t) \ge 1$ we can find a small number $N_0 = N_0(\eta)$ so that 
\EQ{
\| P_{\le N_0} u\|_{L^5} \lesssim \| P_{\le N_0} u\|_{\dot{H}^{\frac{3}{2}}} \lesssim \eta
}
where we have used Sobolev embedding and Remark~\ref{Ceta} above. We now estimate the right-hand-side of~\eqref{A est}. We claim that  
 \EQ{\label{A est 1}
 CM^{-\frac{1}{2}} \| P_M u^3 \|_{L^{\infty}_t L^2_x} \lesssim C \eta^2 M^{\frac{3}{2}} \|P_{>M/4} u \|_{L^{\infty}_tL^2} + C N_0^{-\frac{3}{2}} M^{\frac{3}{2}}
 }
 where the implicit constants above are allowed, as always, to depend on $\| v\|_{L^{\infty}_t \dot{H}^{\frac{3}{2}}}$ and the constant $C$ is yet to be chosen. To prove~\eqref{A est 1} we write 
 \ant{
  \|P_M(u^3) \|_{L^2} &= \left\|P_M \left[(P_{\le M/4} u + P_{>M/4} u)(P_{\le N_0} u + P_{>N_0} u)^2\right] \right\|_{L^2} \\
  & \lesssim \|P_M( P_{\le M/4} u(P_{\le N_0} u)^2) \|_{L^2} + \|P_M( P_{\le M/4} u P_{\le N_0} u P_{>N_0} u) \|_{L^2} \\
  &  \quad + \|P_M( P_{\le M/4} u ( P_{>N_0} u)^2) \|_{L^2} + \|P_M( P_{> M/4} u(P_{\le N_0} u)^2) \|_{L^2} \\
  & \quad  \|P_M( P_{> M/4} u P_{\le N_0} u P_{>N_0} u) \|_{L^2} + \|P_M( P_{> M/4} u ( P_{>N_0} u)^2) \|_{L^2}
 }
We next estimate each term on the right-hand-side above. We begin with the term $\|P_M( P_{\le M/4} u(P_{\le N_0} u)^2) \|_{L^2}$. If $N_0 \le M/4$ then this term is identically zero. In the case that $M/4 < N_0$ we write $P_{\le N_0} u = P_{ \le M/4} u + P_{M/4 <  \cdot \le N_0}u$ and we have 
\ant{
\|P_M( P_{\le M/4} u(P_{\le N_0} u)^2) \|_{L^2} &\lesssim \|P_M( (P_{\le M/4}  u)^2 \, P_{M/4 <  \cdot \le N_0}u) \|_{L^2} \\
& \quad + \|P_M( P_{\le M/4} u  \, (P_{M/4 <  \cdot \le N_0}u)^2) \|_{L^2}
}
recalling that $P_M$ is given by convolution with $\check \phi_M = M^5  \check \phi( M \cdot)$, $ \check \phi \in \cS$, we use Young's inequality followed by H\"older to obtain 
\ant{
\|P_M( (P_{\le M/4}  u)^2 \, P_{M/4 <  \cdot \le N_0}u) \|_{L^2} &\lesssim \| \check \phi_M\|_{L^{\frac{5}{3}}} \|P_{ \le M/4} u\|_{L^5}^2 \|P_{M/4 <  \cdot \le N_0}u \|_{L^2} \\
 & \lesssim   M^2  \eta^2 \|P_{> M/4} u\|_{L^2}
}
Similarly, we have 
\begin{multline*}
\|P_M( P_{\le M/4} u  \, (P_{M/4 <  \cdot \le N_0}u)^2) \|_{L^2} \lesssim \\
\lesssim  \| \check \phi_M\|_{L^{\frac{5}{3}}} \|P_{ \le M/4} u\|_{L^5} \|P_{M/4 <  \cdot \le N_0}u \|_{L^5} \|P_{M/4 <  \cdot \le N_0}u \|_{L^2} \\
 \lesssim M^2 \eta^2 \|P_{> M/4} u\|_{L^2}
\end{multline*}
where  we have been using throughout the fact that $M/4 < N_0$. Next we show how to control the term $\|P_M( P_{> M/4} u ( P_{>N_0} u)^2) \|_{L^2}$. Using the same argument as above, together with Sobolev embedding we have 
\ant{
\|P_M( P_{> M/4} u ( P_{>N_0} u)^2) \|_{L^2} &\lesssim  \| \check \phi_M\|_{L^{\frac{5}{3}}} \|P_{ > M/4} u\|_{L^5} \|P_{> N_0}u \|_{L^5}   \|P_{> N_0}u \|_{L^2}\\
 & \lesssim   M^2  \|u\|_{ \dot{H}^{\frac{3}{2}}}^2 \|P_{> N_0} u\|_{L^2} \lesssim M^2 N_0^{-\frac{3}{2}}
 }
The remaining terms are all handled similarly as the above two are representative and we thus obtain~\eqref{A est 1}. Combining~\eqref{A est 1} with~\eqref{A est} gives 
\EQ{ \label{0 to CM}
 \left\| P_M \left(\int_{0}^{CM^{-1}}\frac{e^{-i\tau \sqrt{-\Delta}}}{\sqrt{-\Delta}} u^3(\tau)d \tau\right) \right\|_{\dot{H}^{\frac{3}{2}}} \lesssim C \eta^2 M^{\frac{3}{2}} \|P_{>M/4} u \|_{L^{\infty}_tL^2} + C N_0^{-\frac{3}{2}} M^{\frac{3}{2}} 
 }  
 We next control the second term in $A$ using the Lemma~\ref{lem: rad se} to obtain sufficient time decay.  Indeed, we have 
 \begin{multline} \label{A2}
 \left\| \int_{CM^{-1}}^{T_+} \frac{e^{-it \sqrt{-\Delta}}}{\sqrt{-\Delta}} (1 - \chi)(\frac{x}{ct}) P_M(u^{3}(t)) dt  \right\|_{\dot{H}^{\frac{3}{2}}} \\ \lesssim  \int_{CM^{-1}}^{\infty} \left \| (1- \chi)(\frac{x}{ct}) P_M u^3 \right\|_{\dot{H}^{\frac{1}{2}}} \, dt  
  \lesssim \int_{CM^{-1}}^{\infty} \left \| (1- \chi)(\frac{x}{ct}) P_M u^3 \right\|_{ \dot{W}^{1, \frac{5}{3}}} \, dt 
 \end{multline}
 We estimate the integrand on the right-hand-side above. 
 \EQ{\label{A2 integrand}
 \left\| (1- \chi)(\frac{x}{ct}) P_M u^3 \right\|_{ \dot{W}^{1, \frac{5}{3}}}  \lesssim  \frac{1}{ct}\| \chi'( \frac{x}{ct}) P_M u^3 \|_{L^{\frac{5}{3}}} +  \left\| (1- \chi)(\frac{x}{ct}) \na P_M u^3 \right\|_{ L^{\frac{5}{3}}} 
 }
 Using Lemma~\ref{lem: rad se} on the first term on the right in~\eqref{A2 integrand} gives 
 \EQ{ \label{A2 est 1}
 \frac{1}{ct}\| \chi'( \frac{x}{ct}) P_M u^3 \|_{L^{\frac{5}{3}}} &\lesssim \frac{1}{(ct)^2}\| r P_M u^3 \|_{L^{\frac{5}{3}}}  \lesssim \frac{1}{(ct)^2}\|  \abs{\na}^{\frac{1}{2}} P_M u^3 \|_{L^{\frac{10}{9}}} \\
 & \lesssim \frac{1}{(ct)^2} M^{\frac{1}{2}} \|P_M u^3 \|_{L^{\frac{10}{9}}} \lesssim  \frac{1}{(ct)^2} M^{\frac{1}{2}} \|P_{>M/4} u\|_{L^2} \|u \|_{L^5}^2
 }
 where in the very last line we used Young's inequality and the decomposition $u = P_{\le M/4} u + P_{>M/4}$. We again use Lemma~\ref{lem: rad se} to estimate the second term in~\eqref{A2 integrand}. 
 \begin{align} \notag
 \left\| (1- \chi)(\frac{x}{ct}) \na P_M u^3 \right\|_{ L^{\frac{5}{3}}}  &\lesssim (ct)^{-\frac{6}{5}} \| r^{\frac{6}{5}} \na P_M u^3\|_{L^{\frac{5}{3}}} \\ \notag
 & \lesssim (ct)^{-\frac{6}{5}}  \| \abs{\na}^{\frac{3}{10}} \na P_M u^3\|_{L^{\frac{10}{9}}}  \lesssim (ct)^{-\frac{6}{5}}M^{\frac{13}{10}}  \|  \na P_M u^3\|_{L^{\frac{10}{9}}}  \\
 & \lesssim (ct)^{-\frac{6}{5}}M^{\frac{13}{10}}  \|P_{>M/4} u\|_{L^2} \|u \|_{L^5}^2 \label{A2 est 2}
 \end{align}
 Plugging the estimates~\eqref{A2 est 1} and \eqref{A2 est 2} into the integrand in~\eqref{A2} and integrating in $t$ from $CM^{-1}$ to $\infty$ gives 
 \begin{multline} \label{CM to inf}
  \left\| \int_{CM^{-1}}^{T_+} \frac{e^{-it \sqrt{-\Delta}}}{\sqrt{-\Delta}} (1 - \chi)(\frac{x}{ct}) P_M(u^{3}(t)) dt  \right\|_{\dot{H}^{\frac{3}{2}}} \lesssim \\
\lesssim ( \frac{1}{C}+ \frac{1}{C^{\frac{1}{5}}})  M^{\frac{3}{2}} \|P_{>M/4} u\|_{L^{\infty}_t L^2} \| u\|_{L^{\infty}_t \dot H^{\frac{3}{2}}}  
  \end{multline}
 Combining~\eqref{0 to CM} and~\eqref{CM to inf} and choosing the constant $C:= \frac{1}{\eta}>1$ (where $\eta$ remains to be fixed below) we obtain, 
 \EQ{ \label{A M}
 \| A \|_{\dot{H}^{\frac{3}{2}}} &\lesssim  ( C\eta^2 + C^{-\frac{1}{5}} + C^{-1}) M^{\frac{3}{2}}  \|P_{>M/4} u\|_{L^{\infty}_t L^2}  + C N_0^{-\frac{3}{2}} M^{\frac{3}{2}} \\
 & \lesssim  \eta^{\frac{1}{5}} M^{\frac{3}{2}}  \|P_{>M/4} u\|_{L^{\infty}_t L^2}  + \eta^{-1} N_0^{-\frac{3}{2}} M^{\frac{3}{2}}
 }
 The same estimate holds for $\|A'\|_{\dot H^{\frac{3}{2}}}$. 
Using again the fact that  $e^{-it \sqrt{-\Delta}} v(t) \rightharpoonup 0$ weakly in $\dot{H}^{3/2}$ as $t \rightarrow T_+$, $T_-$, we estimate $\ang{A, A' + B'}_{\dot{H}^{\frac{3}{2}}}$ and $\ang{A + B, A'}_{\dot{H}^{\frac{3}{2}}}$ by 
\EQ{ \label{A B A} 
\abs{\ang{A, A' + B'}_{\dot{H}^{\frac{3}{2}}} } &\lesssim \| A\|_{\dot{H}^{\frac{3}{2}}} M^{\frac{3}{2}}\| P_M v\|_{L^{\infty}_t L^2}  \\
&  \lesssim  \left( \eta^{\frac{1}{5}} M^{\frac{3}{2}}  \|P_{>M/4} u\|_{L^{\infty}_t L^2}  + \frac{ M^{\frac{3}{2}}}{ \eta N_0^{\frac{3}{2}}} \right) M^{\frac{3}{2}}\| P_M v\|_{L^{\infty}_t L^2} \\
\abs{\ang{A+B, A'}_{\dot{H}^{\frac{3}{2}}} } &
\lesssim \left( \eta^{\frac{1}{5}} M^{\frac{3}{2}}  \|P_{>M/4} u\|_{L^{\infty}_t L^2}  + \frac{ M^{\frac{3}{2}}}{ \eta N_0^{\frac{3}{2}}} \right) M^{\frac{3}{2}}\| P_M v\|_{L^{\infty}_t L^2}
}
Finally, we use the sharp Huygens principle to deduce that $\ang{B, B'}_{\dot{H}^{\frac{3}{2}}} = 0$. Indeed, we have 
\begin{align} \notag
\ang{B, B'}_{\dot{H}^{\frac{3}{2}}}  &= \int_{T_-}^{-\frac{C}{M}} \int_{\frac{C}{M}}^{T_+}\ang{  \chi(\frac{x}{c \abs{t}}) P_M u^3 (t) , \, \na e^{- i(t- \tau) \sqrt{- \De}} \chi( \frac{\cdot}{ c\abs{ \tau}}) P_M u^3 (\tau) }_{L^2} \, dt\, d \tau\\
& = 0 \label{B 0}
\end{align}
With say $c = \frac{1}{4}$, it follows  from the sharp Huygens principle that the two terms inside the $L^2$ bracket above have disjoint spacial supports -- the term on the left part of the bracket  is supported in $\abs{x} \le \frac{1}{4} \abs{t}$ and the term on the right side of the bracket is supported on $|x| > \frac{3}{4} |t - \tau|> \frac{3}{4} \abs{t}$. 
 
Now, we let $M = 2^k$ for $k <0$. By~\eqref{A M},~\eqref{A B A}, and~\eqref{B 0}, we have 
\EQ{
a_k^2 &\lesssim  \left( \eta^{\frac{1}{5}}  \sum_{j \ge k-3} 2^{\frac{3}{2}(k-j)}a_j   + \eta^{-1} N_0^{-\frac{3}{2}} 2^{\frac{3}{2}k}\right)^2  \\
& \quad + a_k \left( \eta^{\frac{1}{5}}  \sum_{j \ge k-3} 2^{\frac{3}{2}(k-j)}a_j   + \eta^{-1} N_0^{-\frac{3}{2}} 2^{\frac{3}{2}k}\right)
}
 which implies that 
 \EQ{
 a_k \lesssim \eta^{\frac{1}{5}}  \sum_{j \ge k-3} 2^{\frac{3}{2}(k-j)}a_j   + \eta^{-1} N_0^{-\frac{3}{2}} 2^{\frac{3}{2}k}
 }
 for all $k <0$. For $k >0$ we simply use the estimate $a_k \lesssim 1$. Recalling the definition of $\al_k$ we then have 
 \EQ{
 \al_k \lesssim  \eta^{\frac{1}{5}} \al_k +  \sum_j  2^{-\abs{j-k}} \inf( \eta^{-1} N_0^{-1}2^{\frac{3}{2}j}, \, \, 1)
 }
 Now we fix $\eta>0$ small enough in order to absorb the first term on the right into to  the left-hand-side giving 
 \EQ{
 \al_k \lesssim  \sum_j  2^{-\abs{j-k}} \inf( \eta^{-1} N_0^{-1}2^{\frac{3}{2}j}, \, \, 1)
 }
This implies that 
\EQ{
\al_k \lesssim 2^k \mfor k<0.
}
Since we have also set $\al_k:=1$ for $k \ge0$ we have  now proved that $\{2^{-\frac{3}{4}k} \al_k \}\in \ell^2$. This completes the proof of Proposition~\ref{prop: u 34}. 
\end{proof}

\subsection{Higher regularity and additional decay for wave maps with the Compactness Property}
In this section we extend the conclusions of Proposition~\ref{prop: u 52}  and Proposition~\ref{prop: u 34} to the case when $\vec u$ is a solution to either~\eqref{u eq wms} or~\eqref{u eq wmh} with the Compactness Property on an interval $I$. In particular we prove the following results.


\begin{prop}\label{prop: wm 52}
Suppose $\vec u(t)$ is a solution to either~\eqref{u eq wms} or~\eqref{u eq wmh}  with the Compactness Property on $I_{\max}(\vec u)$ as in Theorem~\ref{thm: rig}, i.e., assume that there exists a function $N: I_{\max} \to (0, \infty)$ so that the set 
\EQ{
K:= \left\{ \left(\frac{1}{N(t)} u\left( t, \frac{\cdot}{N(t)} \right), \, \frac{1}{N^2(t)} u_t\left( t, \frac{\cdot}{N(t)} \right) \right) \mid t \in I\right\}
}
is pre-compact in $\dot{H}^{3/2} \times \dot H^{1/2}$. 
Then for all $t \in I_{\max}$,
\ant{
\| \vec u(t) \|_{\dot{H}^{5/2} \times \dot{H}^{3/2}(\R^{5})} \lesssim N(t).
}
with a constant that is uniform in $t \in I_{\max}$. 
\end{prop}

\begin{prop} \label{prop: wm 34}
 Let $\vec u(t)$ be a solution to either~\eqref{u eq wms} or~\eqref{u eq wmh} with the Compactness Property as in Proposition~\ref{prop: r1}, i.e, suppose that the scaling parameter $N(t)$ satisfies 
\ant{
 \inf_{t \in I_{\max}(\vec u)} N(t) >0
 }
 then for all $t \in I_{\max}$,
\begin{equation}\label{wm 34}
\| \vec u(t) \|_{\dot{H}^{\frac{3}{4}} \times \dot{H}^{-\frac{1}{4}}({\R}^{5})} \lesssim 1.
\end{equation}
with a constant that is uniform in $t \in I_{\max}$. 
\end{prop}

The proof of Proposition~\ref{prop: wm 52} (respectively Proposition~\ref{prop: wm 34}) is nearly identical to the proof of Proposition~\ref{prop: u 52} (respectively Proposition~\ref{prop: u 34}) and we will thus omit many of the details. In fact, because of the assumption that $\vec u(t) \in \dot{H}^{\frac{3}{2}} \times \dot{H}^{\frac{1}{2}}$ with a uniform in $t \in I$ bound, we also have a uniform $L^{\infty}_x$ estimate for $ r u$, i.e., by Lemma~\ref{lem: rad se} we have 
\ant{
 \| r u\|_{L^{\infty}_t L^\infty_x} \lesssim \| u \|_{L^{\infty}_t \dot{H}^{\frac{3}{2}}} \lesssim 1
 }
 This means that for both $F_{\Sp^3}$ and $F_{\Hp^3}$ we have the estimate
 \EQ{
 \abs{ F_{\Sp^3}(r, u)} , \abs{ F_{\Hp^3}(r, u)}  =  \abs{ u^3 Z_{\Sp^3}(ru)}, \abs{ u^3 Z_{\Hp^3}(ru)} \lesssim \abs{u}^3
 }
 as was previously mentioned in~\eqref{F S} and~\eqref{F H bound}. Therefore, we will only highlight below the instances in the proof of Proposition~\ref{prop: wm 52} in which the structure, rather than just the size of the nonlinearities comes into play. For the proof of Proposition~\ref{prop: wm 34} we provide even fewer details since the necessary additional techniques will already have been introduced in the proof of Proposition~\ref{prop: wm 52}. 



\begin{proof}[Sketch of the proof of Proposition~\ref{prop: wm 52}] 

As we mentioned above, the proof follows the exact same argument as in the proof of Propsition~\ref{prop: u 52} and therefore below we will only give a few details regarding the estimates where the structure, rather than just the size, of the nonlinearity enters into the argument. In particular, we will need replacements for the estimates in Claim~\ref{l1.2}, the related estimate~\eqref{0 to N}, and Claim~\ref{claim: N to inf}.







We begin by providing the estimates necessary to prove~\eqref{N to T+2} which then implies~\eqref{N to T+} and thus the analog of Claim~\ref{claim: N to inf}.  Indeed here we can easily prove that 
\ant{
\| (1 - \chi)(\frac{x}{ct}) F(r, u) \|_{\dot{H}^{3/2}} \lesssim_{\| u\|_{L^{\infty}_t\dot{H}^{\frac{3}{2}}}}  \frac{1}{(ct)^{2}}.
}
where here $F(r,u)$ is either $F(u) = F_{\Sp^3}(r, u)$ or $F(r, u) = F_{\Hp^3}(r, u)$. We have 
\ant{
\| (1 - \chi)(\frac{r}{ct}) F(r, u) \|_{\dot{H}^{3/2}} \lesssim \|  \na ((1 - \chi)(\frac{r}{ct}) F(r, u)) \|_{\dot{H}^{\frac{1}{2}}} 
}
Since everything is radial above we compute, writing $F(r, u) = Z(ru) u^3$ with $Z = Z_{\Sp^3}$ or $Z = Z_{\Hp^3}$, 
\begin{multline*}
\p_r  \left( (1 - \chi)(\frac{r}{ct}) Z(ru) u^3 \right)  =   - \frac{1}{ct} \chi'( \frac{r}{ct}) Z(ru) u^3  + \\
+ (1- \chi( \frac{r}{ct})) (ru) Z'(ru)\left( \frac{u^3}{r}  +  u_r u^2\right)  
+ 2(1- \chi( \frac{r}{ct}))Z(ru) u^2 u_r
\end{multline*}
We estimate the right-hand-side above in $\dot{H}^{\frac{1}{2}}$ proceeding exactly as in the proof of Claim~\ref{claim: N to inf} using now the  boundedness of $ru$  and the structure of $Z$ to note that $Z(ru)$ and $ru Z'( ru)$ are uniformly bounded. The only minor difference in the argument is that we note that Lemma~\ref{lem: rad se} allows us to treat $\frac{u}{r}$ exactly as we would treat $\na u$ and indeed we have 
\ant{
\| r^{-1} u\|_{L^{\frac{5}{2}}} \lesssim \| u\|_{\dot{H}^{\frac{3}{2}}}, \mand \| r^{-1} u\|_{\dot{H}^{\frac{1}{2}}} \lesssim \| u\|_{\dot{H}^{\frac{3}{2}}}
}
which are terms that arise after an application of the fractional product rule as in~\eqref{f p r}.

Next, we show how to proceed in the proofs of the analog  Claim~\ref{l1.2} and the related estimate~\eqref{0 to N}.  Setting $J:=[- \de/ N(0), \de/N(0)]$ and $\eta$ as in Lemma~\ref{l1.1} (of course now with a solution $u$ to \eqref{u eq wms} or~\eqref{u eq wmh}) we know that 
\EQ{ \label{2 10 J}
\| u\|_{L^{2}_t (J; L^{10}_x)} \lesssim \eta
}
Examining the proof of Claim~\ref{l1.2} we note that  we have 
\EQ{
a_k &\lesssim a_k(0) + 2^{\frac{k}{2}}  \| P_k F(r, u) \|_{L^1_t(J;  L^2_x)} \\
}
where $a_k$ and $a_k(0)$ are as in the statement of Claim~\ref{l1.2} and $\vec u$ is as in  Proposition~\ref{prop: wm 52}.  As in the proof of Claim~\ref{l1.2} we need to estimate $\| P_k F(r, u) \|_{L^1_t(J;  L^2_x)}$ and we claim that 
\EQ{  \label{Pk F}
2^{\frac{k}{2}}\| P_k F(r, u) \|_{L^1_t(J;  L^2_x)} \lesssim_{\|u\|_{L^{\infty}_t \dot H^{\frac{3}{2}}}}   \eta^2 \sum_{j} 2^{-\frac{3}{2}\abs{k-j}} a_j 
}
To prove~\eqref{Pk F} we simply consider the power series expansion of $F(r, u)$. Indeed we have 
\ant{
F(r, u) =  \sum_{n \ge 1} \frac{ \iota^{2n+1}}{ (2n+1)!} 2^{2n+1} r^{2n-2} u^{2n+1}
}
where $\iota = -1$ in the case that $F = F_{\Sp^3}$ and $\iota = 1$ in the case that $F = F_{\Hp^3}$. It follows that 
\EQ{
2^{\frac{k}{2}}\| P_k F(r, u) \|_{L^1_t(J;  L^2_x)} &\lesssim 2^{\frac{k}{2}} \|P_k u^3\|_{L^1_t(J;  L^2_x)}  \\
& \quad +  \sum_{n \ge2} \frac{2^{2n+1}}{ (2n+1)!} 2^{\frac{k}{2}} \| P_k ( r^{2n-2} u^{2n+1})\|_{L^1_t(J;  L^2_x)}
}
From the proof of Claim~\ref{l1.2} we know that 
\EQ{ \label{eta 2}
2^{\frac{k}{2}} \|P_k u^3\|_{L^1_t(J;  L^2_x)}  \lesssim \eta^2 \sum_{j>k-4} 2^{\frac{3}{2}(k-j)} a_j 
}
Hence it suffices to show that 
\EQ{ \label{n sum}
\sum_{n \ge2} \frac{2^{2n+1}}{ (2n+1)!} 2^{\frac{k}{2}} \| P_k ( r^{2n-2} u^{2n+1})\|_{L^1_t(J;  L^2_x)} \lesssim_{\|u\|_{L^{\infty}_t \dot H^{\frac{3}{2}}}}   \eta^2 \sum_{j} 2^{\frac{3}{2}\abs{k-j}} a_j 
}
To start, we note that  
\ant{
2^{\frac{k}{2}} \| P_k ( r^{2n-2} u^{2n+1})\|_{L^1_t(J;  L^2_x)} &\lesssim 2^{\frac{k}{2}} \| P_k ( r^{2n-2} u^{2n+1} - r^{2n-2} (P_{\le k-4}u)^{2n+1})\|_{L^1_t(J;  L^2_x)} \\
& \quad + 2^{\frac{k}{2}} \| P_k ( r^{2n-2} (P_{\le k-4}u)^{2n+1})\|_{L^1_t(J;  L^2_x)}
}
We estimate the first term on the right-hand-side above. We claim that 
\EQ{ \label{>M}
2^{\frac{k}{2}} \| P_k ( r^{2n-2} u^{2n+1} - r^{2n-2} (P_{\le k-4}u)^{2n+1})\|_{L^1_t(J;  L^2_x)}  \lesssim_{\|u\|_{L^{\infty}_t \dot H^{\frac{3}{2}}}} \eta^2 \sum_{j >k-4} 2^{-\frac{3}{2}(k-j)} a_j
}
Indeed, writing $M = 2^k$, the left-hand-side can be broken into terms of the form 
\EQ{
M^{\frac{1}{2}} \| P_M (r^{2n-2} u_{>M/4}^{\ell} u_{ \le M/4}^m )\|_{L^2}
}
with $\ell + m = 2n+1$ and $\ell  \ge 1$ and we have introduced the notation $u_{\le K} := P_{\le K} u$. We then have, using Young's inequality and Lemma~\ref{lem: rad se},
\ant{
 M^{\frac{1}{2}} \| P_M (r^{2n-2} u_{>M/4}^{\ell} u_{ \le M/4}^m )\|_{L^2} &\lesssim  M^{\frac{1}{2}} \| \check \phi_M\|_{L^{\frac{5}{4}}} \| ru\|_{L^{\infty}}^{2n-2} \| u_{>M/4}^{\ell_2} u_{\le M/4}^{r_2}\|_{L^{\frac{10}{7}}} \\
 & \lesssim M^{\frac{3}{2}} \| u\|_{\dot{H}^{\frac{3}{2}}}^{2n-2} \| u_{>M/4}\|_{L^2} \|u\|_{L^{10}}^2
 }
 where $m_2 +r^2 = 3$ above and $m_2 \ge 1$. Integrating in time over the interval $J$  proves~\eqref{>M}.  To finish the proof of~\eqref{n sum} we show that 
 \EQ{ \label{j le k}
 2^{\frac{k}{2}} \| P_k ( r^{2n-2} (P_{\le k-4}u)^{2n+1})\|_{L^1_t(J;  L^2_x)} \lesssim_{\|u\|_{L^{\infty}_t \dot H^{\frac{3}{2}}}} \eta^2 \sum_{j \le k-4} 2^{\frac{3}{2}(j-k)} a_j
 }
Again writing $M = 2^k$ we use Lemma~\ref{lem bern} to deduce that 
\EQ{ \label{na 2}
M^{\frac{1}{2}} &\| P_k ( r^{2n-2} (P_{\le k-4}u)^{2n+1})\|_{L^2} \lesssim M^{-\frac{3}{2}}\| P_k \na^2( r^{2n-2} (P_{\le k-4}u)^{2n+1})\|_{L^2}  \\
& \lesssim M^{-\frac{3}{2}} \bigg( (2n-2)(2n-3)\|P_k(r^{2n-4} u_{\le M/4}^{2n+1})\|_{L^2} + \\
& \quad + 2(2n-2)(2n+1) \|P_k(r^{2n-3} u_{\le M/4}^{2n} \na u_{\le M/4})\|_{L^2} +\\
& \quad + (2n+1)(2n)  \|P_k(r^{2n-2} u_{\le M/4}^{2n-1} (\na u_{\le M/4})^2)\|_{L^2} +\\
& \quad + (2n+1)(2n)  \|P_k(r^{2n-2} u_{\le M/4}^{2n-1} \na^2 u_{\le M/4}^2)\|_{L^2}  \bigg)
}
We estimate the first term inside the parentheses on the right-hand-side above as follows 
\ant{
\|P_k(r^{2n-4} u_{\le M/4}^{2n+1})\|_{L^2} \lesssim \| r u\|_{L^{\infty}}^{2n-4} \| u_{\le M/4}^5\|_{L^2}
}
Then note that replacing again $M = 2^k$ we have using Lemma~\ref{lem bern} that 
\ant{
\| (P_{\le k-4}u)^5\|_{L^2} &\lesssim \sum_{j_1 \le  \dots \le j_5 \le k-4}  \| P_{j_1} u\|_{L^{10}} \| P_{j_2} u\|_{L^{10}}\|P_{j_3} u\|_{L^{\infty}}\|P_{j_4} u\|_{L^{\infty}} \| P_{j_5} u\|_{L^{\frac{10}{3}}} \\
& \lesssim \| u\|_{L^{10}}^2  \sum_{ j_3 \le j_4 \le j_5 \le k-4} 2^{-\frac{1}{2}j_5} 2^{j_4} 2^{j_3} \|P_{j_3} u\|_{L^5} \| P_{j_4} u\|_{L^5} \| P_{j_5}u \|_{ \dot{H}^{\frac{3}{2}}} \\
& \lesssim \| u\|_{L^{10}}^2  \| u \|_{\dot{H}^{\frac{3}{2}}}^2 \sum_{j \le k-4} 2^{\frac{3}{2}j } a_j
}
The three remaining terms inside the parentheses on the right-hand-side of~\eqref{na 2} are handled in an almost identical manner and we omit the details. Integrating in time over $J$ and using~\eqref{2 10 J} completes the proof of~\eqref{j le k} and hence of~\eqref{n sum}. 

We have now completed the proof of~\eqref{Pk F}. By an analogous argument we can establish the analog of~\eqref{0 to N} for $u$ as in Proposition~\ref{prop: wm 52}. To complete  the proof of Proposition~\ref{prop: wm 52} simply run the same argument as the proof of Proposition~\ref{prop: u 52} with the estimates proved above inserted in the appropriate places.  
\end{proof}

\begin{proof}[Brief Sketch of the proof of Proposition~\ref{prop: wm 34}]
For the proof of Proposition~\ref{prop: wm 34} simply follow the exact same outline as the proof of Propostion~\ref{prop: u 34} inserting arguments based on the power series expansions of $F_{\Sp^3}$ and $F_{\Hp^3}$ where necessary as in the sketch of Propostion~\ref{prop: wm 52} above. 
 We omit the details. 
\end{proof} 

\section{Rigidity Part I: Nonexistence of compact self-similar blow-up} \label{sec: ss} 
In this section we prove Proposition~\ref{prop: r2}. We divide the augment into the following two subsections. The first deals with the case that the compact solution $\vec u(t)$ solves the focusing cubic  equation~\eqref{u eq}. The second subsection deals with the case of self-similar wave maps with the compactness property. The arguments in both subsections are similar to and were inspired by the argument presented in~\cite[Section $4$]{DKM5}. However, here the extra regularity gained in Section~\ref{sec: dd} is required to make the proof go through. In fact the cubic type nonlinearity would be a limiting case of the argument  in~\cite{DKM5} adapted to the present $5d$ setting and falls just out of the scope of the techniques used there. Thankfully, the extra regularity gained for critical elements in Section~\ref{sec: dd} is strong enough to allow us to adapt the argument from~\cite{DKM5} to our situation, and in fact allows for a few simplifications. 
\subsection{Proof of Proposition~\ref{prop: r2} for solutions to the cubic equation~\eqref{u eq}} \label{sec: u 3 ss}

Suppose that $ \vec u$ is a self-similar solution to~\eqref{u eq} with the compactness property as in Proposition~\ref{prop: r2}. Without loss of generality we can assume, 
 $T_{+}(\vec u) = 1$. 
 \begin{lem}\label{lem: H2} Let $\vec u(t)$ be as in Proposition~\ref{prop: r2}. Then 
 the set 
 \ant{
K_+:= \{ ((1- t) u(t, (1-t) r), \, (1-t)^2 u_t(t, (1-t) r) \mid t \in [0, 1)\}
 }
 is pre-compact in $(\dot{H}^2 \times \dot H^1) \cap (\dot{H}^{3/2} \times \dot H^{1/2})$. 
 \end{lem}
 \begin{proof} This is immediate consequence of Proposition~\ref{prop: u 52} as we can interpolate between the $\dot{H}^{5/2} \times \dot{H}^{3/2}$ bound and the  $\dot{H}^{3/2} \times \dot H^{1/2}$ bound to control the $\dot{H}^{2} \times \dot H^1$ norm of $((1- t) u(t, (1-t)  \cdot), \, (1-t)^2 u_t(t, (1-t)  \cdot)$. Compactness in $\dot{H}^{3/2} \times \dot H^{1/2}$  together with uniform  boundedness in $\dot{H}^{5/2} \times \dot{H}^{3/2}$ then implies compactness in $\dot{H}^{2} \times \dot H^1$. 
 \end{proof}

 
 Now we introduce new  coordinates. Let $s = -\log(1 - t)$, $e^{-s} = 1 - t$, and $t = 1 - e^{-s}$. Then for $s \in [0, \infty)$
\begin{align}\label{3.2}
&w(s, r) = e^{-s} u(1 - e^{-s}, e^{-s} r) \\
&\partial_{s} w(s, r) = -w(s, r) - e^{-2s} r u_{r}(1 - e^{-s}, e^{-s} r) + e^{-2s} u_{t}(1 - e^{-s}, e^{-s} r). \label{3.3}
\end{align}
 Now, using~\cite[Lemma $4.15$]{KM11a} we know that $T_+(\vec u) = 1$ implies that  $\vec u(t)$ is supported in $B(0, 1-t)$, which means that we only need to consider $ r \in [0, 1)$. The compact support then implies that the set $K_+$ in Lemma~\ref{lem: H2} is pre-compact in the inhomogeneous space $H^2 \times H^1(\R^5)$. 
Rephrasing this in terms of $w$, and using the compact support of $\vec u(t) \in B(0, 1-t)$ we see that the set 
\EQ{
\ti K_+:= \{ w(s) \mid s \in [0, \infty)\}
}
is pre-compact in  $H^{2}$. Next, we derive an equation for $w(s)$.  
%
Given that $\vec u(t)$ solves~\eqref{u eq} we have 
\begin{equation}\label{w eq ss} 
\partial_{s}^{2} w - \frac{(1 - r^{2})}{r^{4}} \partial_{r} (r^{4} \partial_{r} w) + 2r \partial_{r} \partial_{s} w + 3 \partial_{s} w + 2w - w^{3} = 0.
\end{equation}
First we prove the following estimates using a monotonicity formula for an energy that we associate to~\eqref{w eq ss}.  

\begin{lem}\label{lem: w mf} 
Under the preceding assumptions 
\ant{
\int_{0}^{\infty} \int_{0}^{1} (1 - r^{2})^{-2} (\partial_{s} w(s,r))^{2} r^{4} dr < \infty.
}
\end{lem}

\begin{proof} By \eqref{w eq ss}, we have 

\EQ{\label{3.10}
&0 = \int r^{4} (1 - r^{2})^{-1}  w_s \bigg( w_{ss} - \frac{(1 - r^{2})}{r^{4}} \partial_{r} (r^{4}  w_r) + 2r  w_{sr} + 3 \partial_{s} w + 2w - w^{3}\bigg) dr \\
&= \frac{1}{2} \partial_{s} \int (\partial_{s} w)^{2} (1 - r^{2})^{-1} r^{4} dr + \frac{1}{2} \partial_{s} \int (\partial_{r} w)^{2} r^{4} dr + \int r^{5}(1 - r^{2})^{-1} \partial_{r} (\partial_{s} w)^{2} dr \\
&+ \partial_{s} \int w^{2} (1 - r^{2})^{-1} r^{4} dr - \frac{1}{4} \partial_{s} \int w^{4}  (1 - r^{2})^{-1} r^{4}dr + 3 \int   (\partial_{s} w)^{2} (1 - r^{2})^{-1}r^{4}dr.
}
Defining
\EQ{\label{3.11}
E(s) = &\frac{1}{2} \int_0^1  w_s^{2}(s) (1 - r^{2})^{-1} r^{4} dr + \frac{1}{2} \int_0^1  w_r^{2}(s) r^{4} dr \\
& \quad + \int_0^1 w^{2}(s) (1 - r^{2})^{-1} r^{4} dr - \frac{1}{4} \int_0^1 (1 - r^{2})^{-1} w^{4}(s) r^{4} dr, 
}
and integrating $(\ref{3.10})$ by parts, we have that 
\begin{equation}\label{3.12}
\frac{d}{ds} E(s) = 2 \int_{0}^{1}  (1 - r^{2})^{-2} w_s^{2}(s, r) r^{4}dr.
\end{equation}
 Next, we observe that  $E(0) \geq -C_{0}$ for some constant $C_{0}$. Indeed, by inspection

\begin{equation}\label{3.13}
E(0) \geq -\frac{1}{4} \int_{0}^{1} (1 - r^{2})^{-1} w^{4} r^{4} dr.
\end{equation}

\noindent Now $\| w \|_{H^{2}} \lesssim 1$, so by the radial Sobolev embedding this implies that $\| w_{r} \|_{L^{\infty}} \lesssim 1$ for $r \in [\frac{1}{2}, 1]$. Therefore, since $w(s)$ is supported on $B(0,1)$, by the fundamental theorem of calculus $\| w(s,r) \|_{L^{\infty}} \lesssim (1 - r)$ for $r \in [\frac{1}{2}, 1]$. Therefore, since $w(s)$ lies in a pre-compact subset of $H^{2}$,
\EQ{\label{3.14}
\int_0^1 w^{4}(s)(1 - r^{2})^{-1} r^{4} dr \leq C.
}
with a uniform in $s$ constant.  It remains to show that $E(s)$ is uniformly bounded above for all $s \in [0, \infty)$. Since $\| w(s,r) \|_{L^{\infty}} \lesssim 1 - r$ for all $r \in [\frac{1}{2}, 1]$ and $w(s)$ lies in a  pre-compact subset of $H^{2}$, we can find a constant  $C$ independent of $s \in [0, \infty)$ so that 
\ant{
\int_0^1 w^{2}(s) (1 - r^{2})^{-1} r^{4} dr \leq C.
}
Also since $\| w(s) \|_{\dot{H}^{1}}$ is uniformly bounded in $s \in [0, \infty)$,
\begin{equation}\label{3.16}
\int_{0}^{1}  w_r^{2}(s) r^{4}dr \leq C.
\end{equation}
Also, by H{\"o}lder's inequality, the fundamental theorem of calculus, the support of $w$ lying in $B(0,1)$, and pre-compactness of $\ti K_+$ in $H^{2} \times H^{1}$ for all $t \in [0, 1)$ yields 
\ant{
|e^{-2s} r u_{r}(e^{-s} y, 1 - e^{-s})| + |e^{-2s} u_{t}(e^{-s} y, 1 - e^{-s})| \lesssim (1 - r)^{1/2},
}
 which combined with $(\ref{3.3})$ and $(\ref{3.16})$ proves that

\begin{equation}\label{3.18}
\int_{0}^{1} (1 - r^{2})^{-1}  w_s^{2}(s, r)\,  r^{4} dr \leq C.
\end{equation}

\noindent Then by the fundmantal theorem of calculus in $s$ and $(\ref{3.12})$, the proof of Lemma~\ref{lem: w mf} is complete. 
\end{proof}

Next, we show that in fact $w(s)$ must converge to a \emph{stationary solution} to~\eqref{w eq ss} along a sequence $s_n \to \infty$.
First note that 
 since $(1 - r^{2})^{-2} \geq 1$, we can conclude from Lemma~\ref{lem: w mf} that 
\begin{equation}\label{3.19}
\int_{0}^{\infty} \int_{0}^{1}   w_s^{2}(s, r) r^{4}dr \leq C_{0}.
\end{equation}
Now, let $s_n \to \infty$ be any sequence. Using the compactness of $\ti K_+$ we can find $w^* \in H^2$ with $\supp w^* \in B(0, 1)$ so that after passing to a subsequence we have
\EQ{ \label{w star}
w(s_n) \to w^* \in H^2
}
Next we claim  that for any $T \ge 0$ we have 
\EQ{ \label{w star h2}
 \|w(s_n +T) - w^*\|_{H^2} \to 0 \mas n \to \infty
 }
Indeed, let $T \ge0$. Then by the Fundamental Theorem of Calculus and~\eqref{3.19} we deduce that 
\EQ{
 \| w(s_n+T) - w(s_n) \|_{L^2}^2  &= \int \abs{ w(s_n +T, r) - w(s_n, r)}^2 \, r^4 \, dr \\
 & =  \int \abs{  \int_{s_n}^{s_n +T} w_s(s, r) \, ds}^2 \, r^4 \, dr \\
 & \le T  \int_{s_n}^{s_n +T} \int  w_s^2(s, r) \, r^4 \, dr \, ds \to 0 \mas n \to \infty
 }
This implies that for all $T  \ge 0$ we have 
\EQ{
 \| w(s_n+T) - w^* \|_{L^2} \to 0  \mas n \to \infty
 }
 Another application of the compactness in $H^2$ of $ \ti K_+$ allows us to upgrade the above to obtain~\eqref{w star h2}. 
 
 Next, define $\vec v_n(0):=(v_{n, 0}(r), v_{n, 1}(r))$ by 
 \EQ{
   (v_{n, 0}(r), v_{n, 1}(r)):=   \left( e^{-s_n} u( 1- e^{-s_n}, \, e^{-s_n} r), \, e^{-2s_n} u_t( 1-e^{-s_n}, \, e^{-s_n} r) \right)
 }
 By the pre-compactness of $K_+$ in $H^2 \times H^1$ we can find a strong limit $ \vec v(0) = (v_0, v_1)$ so that 
 \ant{
  (v_{n, 0}, v_{n, 1}) \to  \vec v(0) \in H^2 \times H^1
  }
  Let $\vec v_n(t)$ be the solution to~\eqref{u eq} with initial data $ \vec v_n(0)$ and let $\vec v(t)$ be the solution to~\eqref{u eq} with data $\vec v(0)$. By the Perturbation Lemma~\ref{lem: pert} we know that for each $s \in [0, T_+( \vec v))$ we have
 \EQ{
 \| \vec v_n(s) -  \vec v(s)\|_{\dot H^{3/2} \times \dot H^{1/2}} \to 0 \mas n \to \infty
 }
 Moreover, for $s<\min \{1, T_+(\vec v)\}$ and $T$ chosen so that $s= 1-e^{-T}$ we have 
 \EQ{
 v_{n}(s, r)  = e^{-s_n} u( 1-e^{-s_n} + e^{-s_n} s, \, e^{-s_n} r)
 }
 Next observe that for $T$ as above
 \EQ{
 w(s_n +T, r) = e^{-s_n} e^{-T}u( 1-e^{-s_n -T}, \, e^{-s_n-T} r) = e^{-T} v_n(1-e^{-T}, e^{-T} r)
 }
 The left hand side above tends to $w^*$ in $H^2$ so letting $n \to \infty$ above gives that for all $T$ as above we have 
 \EQ{ \label{w=v}
 w^*(r)  = e^{-T} v(1-e^{-T}, e^{-T} r)
 }
 Since $\vec v$ is a solution to~\eqref{u eq} we have that $w^*$ is a solution to~\eqref{w eq ss} which is \emph{independent of $s$}. Therefore $w^*$ solves  
\begin{equation}\label{3.24}
(1 - r^{2})(w^*_{rr} + \frac{4}{r} w^*_{r}) + 2w^* - (w^*)^{3} = 0.
\end{equation}

\begin{lem}
$w^*\equiv 0$.
\end{lem}

\begin{proof} Since $w^*\in H^{2}$, and $\supp w^* \in B(0, 1)$ we have $\| w^*_{rr} \|_{L^{2}([\frac{1}{2}, 1])} \lesssim 1$. Since $w^*$ is supported on $B(0,1)$ and by the Fundamental Theorem of Calculus and H{\"o}lder's inequality,
\begin{equation}\label{3.25}
|w^*_{r}(r)| \lesssim (1 - r)^{1/2},  \mfor r \in [1/2, 1]
\end{equation}
and from this we 
\begin{equation}\label{3.26}
|w^*(r)| \lesssim (1 - r)^{3/2} \mfor  \, r \in [1/2, 1]
\end{equation}
 Plugging~\eqref{3.25} and~\eqref{3.26} into~\eqref{3.24}, for $\frac{1}{2} \leq r \leq 1$, gives 
\begin{equation}\label{3.27}
|w^*_{rr}(r)| \leq \frac{4}{r} \abs{w^*_{r}} + \frac{2}{1 - r} \abs{w^*} + \frac{\abs{w^*}^{3}}{1 - r} \leq C_{0} (1 - r)^{1/2}.
\end{equation}
Now also by~\eqref{3.24}, if $|w^*_{rr}| \le C(1 - r)^{1/2}$ for all $r \in [1 - \delta, 1]$, and $C \leq C_{0}$, then by the Fundamental Theorem of Calculus,~\eqref{3.25}, and~\eqref{3.26} (which implies that $w^*(1) = w^*_{r}(1) = 0$), we have 
\begin{equation}\label{3.28}
| ((1 - r^{2}) (w^*_{rr} + \frac{4}{r} w^*_{r}))| \leq 2 |w^*(r)| + |w^*(r)|^{3} \leq 2C(1 - r)^{5/2} + C^{3} (1 - r)^{15/2}.
\end{equation}
Then since $1 + r \geq 1$,
\begin{equation}\label{3.29}
|w^*_{rr} + \frac{4}{r} w^*_{r}| \leq 2C (1 - r)^{3/2} + C^{3} (1 - r)^{13/2}.
\end{equation}
For $r \in [\frac{1}{2}, 1]$, $|\frac{4}{r} w^*_{r}| \leq 8C(1 - r)^{3/2}$. Therefore,
\begin{equation}\label{3.30}
|w^*_{rr}(r)| \leq 10C (1 - r)^{3/2} + 3 C^{3} (1 - r)^{13/2}.
\end{equation}
Then for $r \in [1 - \frac{1}{100 (1 + C_{0}^{3})}, 1]$,~\eqref{3.30} implies that $|w^*_{rr}(r)| \leq \frac{C}{2} (1 - r)^{1/2}$. Then by induction starting with~\eqref{3.27}, one can easily prove that $w^*_{rr}(r) = 0$ for $r \in [1 - \frac{1}{100 (1 + C^{3})}, 1]$. Therefore, $w^*$ is the solution to the elliptic partial differential equation on $B(0, 1 - \delta)$, $\delta > 0$,
\begin{equation}\label{3.31}
(1 - |x|^{2}) \Delta w^* + 2w^* - (w^*)^{3} = 0, \, \, \,  w^*|_{\partial B(0, 1 - \delta)} = 0, \, \, \,  \partial_{n} w^*|_{\partial B(0, 1 - \delta)} = 0.
\end{equation}
Note that \eqref{3.31} is a \emph{nondegenerate} elliptic equation with Dirichlet and Neumann boundary conditions. Therefore $w^* \equiv 0$. 
\end{proof}

At this point it is easy to conclude the proof. Since $w^*  \equiv 0$ we can deduce from~\eqref{w=v} that $ \vec v \equiv 0$. But this means that $\vec v_n(0)$ converges to $(0, 0)$ in $H^2 \times H^1$. But this is impossible since we have assumed that $\vec u(t)$ is a blow-up solution. 
We have arrived at a contradiction and hence we have proved Proposition~\ref{prop: r2} for solutions to~\eqref{u eq}.

\subsection{Nonexistence of self-similar wave maps with the compactness property}
Next we exclude the possibility of compact, self-similar blow-up for solutions to~\eqref{u eq wms} and~\eqref{u eq wmh}. We again can assume that $T_+ =1$ and that by \cite[Lemma~$4.15$]{KM11a} we have $\supp u (t, r), u_t(t, r) \in B(0, 1-t)$. Again we use this support property along with Proposition~\ref{prop: wm 52} to deduce that  the set 
\EQ{
K_+ := \{ ((1- t) u(t, (1-t) r), \, (1-t)^2 u_t(t, (1-t) r) \mid t \in [0, 1)\}
 }
is in fact pre-compact in the inhomogeneous space $H^2 \times H^1$. A simple argument allows us to exclude this type of solution to the defocusing-type equation~\eqref{u eq wmh}. Solutions to~\eqref{u eq wmh} have a positive definite conserved energy given by 
\ant{
\E_{\Hp^3} (\vec u):=  \frac{1}{2} \int_{0}^{\infty}  \left( u_r^2  + u_t^2 + 2\frac{\sinh^2(ru) -(ru)^2}{r^4} \right) \, r^4 \, dr
}
Setting 
\EQ{
\vec v(t, r) = (v(t, r), v_t(t, r)):= ((1- t) u(t, (1-t) r), \, (1-t)^2 u_t(t, (1-t) r)
}
we have that $\vec v(t)$ is pre-compact in $H^2 \times H^1$ for $t \in [0, 1)$ and  thus
\EQ{
 \E_{\Hp^3} (\vec u(t)) \lesssim  (1-t)  \left(  \| v_t(t)\|_{L^2}^2 + \|v_r(t)\|_{L^2}^2 + \| v(t)\|_{L^4}^4\right) \to 0 \mas t \to 1
 }
 Therefore $\E_{\Hp^3} (\vec u) \equiv 0$ and it follows that $\vec u \equiv (0, 0)$, which is contradiction since we assuming that $\vec u(t)$ blows up at $t =1$. 
 
 In the case that $\vec u(t)$ solves the focusing-type equation~\eqref{u eq wms} we follow the exact same argument as in the Section~\ref{sec: u 3 ss}, defining $s = -\log(1 - t)$, $e^{-s} = 1 - t$, and $t = 1 - e^{-s}$ and  for $s \in [0, \infty)$, setting 
\begin{align*}
&w(s, r) = e^{-s} u(1 - e^{-s}, e^{-s} r) \\
&\partial_{s} w(s, r) = -w(s, r) - e^{-2s} r u_{r}(1 - e^{-s}, e^{-s} r) + e^{-2s} u_{t}(1 - e^{-s}, e^{-s} r). 
\end{align*}
Of course in the present situation the nonlinearity $F_{\Sp^3}(r, u)$ is different than $u^3$, however an inspection of the proof in Section~\ref{sec: u 3 ss} reveals that only the rough size of the nonlinearity 
 \EQ{
\abs{ F(r, u)} \lesssim \abs{u}^3
}
is needed along with the uniform estimate 
\EQ{
 (rw)^2 - \sin^2 (rw)   = ((rw) - \sin rw)(rw + \sin rw) \lesssim (rw)^4
 }
 which arises when providing a lower bound for the analog of $E(0)$ as in~\eqref{3.13}.  With these observations the proof proceeds exactly as in Section~\ref{sec: u 3 ss}. We omit the details. This completes the proof of Proposition~\ref{prop: r2}.

\section{Rigidity Part II: Proof of~Proposition~\ref{prop: r1}} \label{sec: rig}

In this section we prove Proposition~\ref{prop: r1}. We use an argument developed in~\cite{KLS} that is based on exterior energy estimates for the free wave equation in $\R^{1+5}$. This method is a refinement of the channels of energy method developed by Duyckaerts, Kenig, and Merle in~\cite{DKM4, DKM5}. In the current implementation, the argument closely resembles the one in~\cite[Section $5$]{L13}, however, we provide a detailed argument here for completeness. We begin with the following consequence of Proposition~\ref{prop: u 34} and Proposition~\ref{prop: wm 34}. 

\begin{lem} \label{lem: compact}
Let $\vec u(t)$ be a solution to either~\eqref{u eq},~\eqref{u eq wms},~\eqref{u eq wmh} with the Compactness Property on its maximal interval of existence $I_{\max}(\vec u)$ as in Propsition~\ref{prop: r1}, i.e., assume that the scaling parameter $N(t)$ satisfies 
\ant{
\inf_{t \in I_{\max}} N(t) >0. 
}
Then, we have $\vec u(t) \in \dot{H}^1 \times L^2 (\R^5)$ for all $t \in I_{\max}$, and   in fact the trajectory 
\EQ{\label{K1}
K_1:= \{ \vec u(t) \mid t\in I_{\max}\}
}
is pre-compact in $\dot{H}^1 \times L^2(\R^5)$. 
As a consequence,  we have the following vanishing: Let $R>0$ be arbitrary. Then 
\EQ{ \label{u h1 R}
\limsup_{t \to T_-}\| \vec u(t) \|_{\dot{H}^1 \times L^2 (r \ge R + \abs{t})}   = \limsup_{t \to T_+}\| \vec u(t) \|_{\dot{H}^1 \times L^2 (r \ge R + \abs{t})}  = 0
}
where
\ant{
\| \vec u(t) \|_{\dot H^1 \times L^2( r \ge R+\abs{t})}^2 := \int_{R + \abs{t}}^{\infty}( u_t^2(t, r) + u_r^2(t, r) )\, r^4 \, dr
}
\end{lem}

\begin{proof} 
We begin by proving that $K_1$ is pre-compact in the energy space. We show that for every sequence $t_n \to T_+,$ or $t_n \to  T_-$ the sequence $ \vec u(t_n)$ has a convergent subsequence in $\dot{H}^1 \times L^2$. Let $t_n \to T_+$. First assume that $N(t_n)$ remains bounded for all $n$. In this case we can, without loss of generality assume $N(t_n)  = 1$ for all $n$ and we have using Proposition~\ref{prop: u 34} and interpolation 
\EQ{
  \| \vec u(t_n) - \vec u(t_m)\|_{\dot H^1 \times L^2} &\le  \|\vec u(t_n) - \vec u(t_m)\|^{\frac{2}{3}}_{\dot H^{\frac{3}{4}} \times \dot{H}^{-\frac{1}{4}}}  \|\vec u(t_n) - \vec u(t_m)\|^{\frac{1}{3}}_{\dot H^{\frac{3}{2}} \times \dot{H}^{\frac{1}{2}}}  \\
  & \le C \|\vec u(t_n) - \vec u(t_m)\|^{\frac{1}{3}}_{\dot H^{\frac{3}{2}} \times \dot{H}^{\frac{1}{2}}}
}
The claim then follows from the compactness in $\dot H^{\frac{3}{2}} \times \dot{H}^{\frac{1}{2}}$ since we know that $\vec u(t_n)$ has a convergent subsequence there. Next assume that $t_n$ has a subsequence, still denoted by $t_n$ so that $N(t_n) \to \infty$. In this case we claim that $\vec u(t_n) \to 0$ in $\dot{H}^1 \times L^2$. To see this let $\eta$ be a small number and find  $c(\eta)$ as in Remark~\ref{Ceta}  so that 
\EQ{
 \int_{\abs{\xi} \le c(\eta)N(t)} \abs{\xi}^3 \abs{\hat{u}(t,\xi)}^{2}\,  d\xi  \le \eta_0,
}
Then 
\ant{
 \| u(t_n)\|_{\dot{H}^1}^2 &= \int_{\abs{\xi} \le c(\eta)N(t_n)} \abs{\xi}^2 \abs{\hat{u}(t_n,\xi)}^{2}\,  d\xi  + \int_{\abs{\xi} \ge c(\eta)N(t_n)} \abs{\xi}^2 \abs{\hat{u}(t_n,\xi)}^{2}\,  d\xi   \\
 & \le \int_{\abs{\xi} \le c(\eta)N(t_n)} \abs{\xi}^2 \abs{\hat{u}(t_n,\xi)}^{2}\,  d\xi  + N(t_n)^{-1} c(\eta)^{-1} \|u(t_n)\|_{\dot{H}^{\frac{3}{2}}}^2
 }
The second term on the right-hand-side above tends to zero as $n \to \infty$ since we have assumed $N(t_n) \to \infty$  as $n \to \infty$. For the first term, we interpolate 
\ant{
&\int_{\abs{\xi} \le c(\eta)N(t_n)} \abs{\xi}^2 \abs{\hat{u}(t_n,\xi)}^{2}\,  d\xi \le  \\
 & \le \left(\int_{\abs{\xi} \le c(\eta)N(t_n)} \abs{\xi}^{\frac{3}{2}} \abs{\hat{u}(t_n,\xi)}^{2}\,  d\xi \right)^{2/3} \left(\int_{\abs{\xi} \le c(\eta)N(t_n)} \abs{\xi}^3 \abs{\hat{u}(t_n,\xi)}^{2}\,  d\xi \right)^{1/3} \\
 &  \le  \eta \| u(t_n)\|_{\dot{H}^{3/4}} \lesssim \eta
}
where the last line follows from Proposition~\ref{prop: u 34} and Remark~\ref{Ceta}. The same argument works for the time derivatives $u_t(t_n)$. As $\eta$ can be chosen arbitrarily small we have proved that $\vec u(t_n) \to 0$ in $\dot{H}^1 \times L^2$. Hence $K_1$ is pre-compact in the energy space.  

Next we prove~\eqref{u h1 R}. First assume that $T_+ = \infty$. Then since $K_1$ is pre-compact in $\dot{H}^1 \times L^2$ is follows that 
for every $\e >0$ there exists an $R(\e)>0$ so that for all $t \in [0, \infty)$, we have
\ant{
 \| \vec u(t) \|_{\dot{H}^1 \times L^2 (r \ge R(\e))} \le \e, 
}
and~\eqref{u h1 R} is a direct consequence of the above. 

Next, assume that $T_+<\infty$. A standard argument using compactness,  see Remark~\ref{rem: N}$(2)$,  gives that in this case we must have 
\EQ{
\supp \vec u(t, \cdot) \subset B(0, T_+ - t)
}
where $B(0, r)$ denotes the ball of radius $r>0$ in $\R^5$. Then for any $R>0$, we can find a $t_0>0$, $t_0 = t_0(R)$ so that 
\ant{
 R+ t > T_+ - t, \quad \forall \, \, t_0  \le t<T_+
 }
 Hence, 
 \ant{
  \| \vec u(t) \|_{\dot{H}^1 \times L^2 (r \ge R+ \abs{t})} = 0  \quad \,  \forall \,t_0  \le t<T_+. 
}
The statements in~\eqref{u h1 R} in the negative time directions follow from identical arguments. 
\end{proof}

\subsection{Singular stationary solutions}
Next, we construct one parameter families of stationary solutions to~\eqref{u eq},~\eqref{u eq wms} and~\eqref{u eq wmh} that do not lie in the critical space $\dot{H}^{\frac{3}{2}}(\R^5)$. We will later show that any nonzero pre-compact trajectory as in Proposition~\ref{prop: r1} has to be equal to one of these singular stationary solutions, which will yield a contradiction. We begin with the construction of a family of infinite energy stationary wave maps to $\Sp^3$ and to $\Hp^3$.  

\subsubsection{Singular stationary wave maps}
 We prove the following result via a simple phase portrait analysis after a reduction to an autonomous ODE. 

\begin{lem} \label{lem: hm} For any $\ell    \in \R$, there exists a unique radial $C^2$ solution $ \fy^1_{\ell}$ of 
\EQ{ \label{hm 5 s}
-\fy_{rr}  - \frac{4}{r} \fy_r  = Z_{\Sp^3}(r \fy)\fy^3\,  \quad r >0
}
as well as a unique radial $C^2$ solution $\fy^2_{\ell}$ of 
\EQ{ \label{hm 5 h}
-\fy_{rr}  - \frac{4}{r} \fy_r  = Z_{\Hp^3}(r \fy)\fy^3\,  \quad r >0
}
where in both cases the solution has the asymptotic behavior 
\EQ{ \label{hm ell s}
r^3\fy_{\ell}^j(r)  =  \ell + O( r^{-4}) \mas r  \to \infty, \quad j = 1, 2
}
The $O( \cdot)$ is determined by $\ell$ and vanishes for $\ell =0$. Moreover, in each case $j  = 1, 2$  for $\ell \neq 0$ we have $ \fy_{\ell}^j \notin L^5( \R^5)$, which means that $\fy_{\ell}^j\notin \dot{H}^{\frac{3}{2}}(\R^5)$ by Sobolev embedding. 

\end{lem}

\begin{proof}  We first prove the lemma for the $\Sp^3$ target, which concerns solutions to~\eqref{hm 5 s}. Recall  then if we define $Q(r) := r \fy(r)$
then $\fy$ solves~\eqref{hm 5 s}, if and only if $Q$ solves 
\EQ{
Q_{rr} +  \frac{2}{r} Q_r  =  \sin 2 Q
}
Moreover, if we would like $\fy$ to satisfy~\eqref{hm ell s} with $\ell  \in \R$ then we need to impose the boundary condition 
\EQ{
\lim_{r  \to \infty}  Q(r)  = 0
}
A standard trick is  to introduce new variables $s= \log(r)$ and $\phi(s) = Q(r)$, and  we obtain an autonomous differential equation for $\phi$,
\EQ{\label{eq: phi}
\ddot \phi + \dot\phi= \sin(2 \phi)
 }
Note that the above is the equation for a damped pendulum. The proof thus reduces to an analysis  of  the phase portrait associated to \eqref{eq: phi}. 
Setting $x(s) = \phi(s)$, $y(s) =  \dot\phi(s)$ we can rewrite~\eqref{eq: phi} as the system
\EQ{\label{xy ode}
 \pmat{ \dot{x} \\  \dot y}  = \pmat{ y \\ - y+ \sin(2x)} =: X(x, y)
}
and we denote by $\Phi_s$ the flow associated to the vector filed $X$. The equilibria of \eqref{xy ode}  occur at points $z_{k/2} = (\frac{ k \pi }{2} , 0)$ where $k \in \Z$. For each $\frac{k}{2} = n \in  \Z$  the flow has a saddle structure with eigenvalues $\la_+ = 1$, $\la_-= -2$, and the corresponding unstable and stable invariant subspaces for the linearized flow are given by the spans of $(1, \la_+)=(1, 1)$, respectively $(1, \la_-) = (1, -2)$. In a neighborhood $V \ni v_n = (n \pi, 0)$ we then have the  one-dimensional invariant unstable manifold 
\ant{
W^u_n = \{ (x, y) \in V \mid \Phi_s(x, y) \to v_n\, \,  \textrm{as} \, \, s \to - \infty\}
}
and the one-dimensional invariant stable manifold 
\ant{
W^{s}_n = \{ (x, y) \in V \mid \Phi_s(x, y) \to v_n\, \,  \textrm{as} \, \, s \to  \infty\}
}
which are tangent at $z_n$ to the invariant subspaces of the linearized flow. In particular, for  $n = 0$ we can parameterize the stable manifold $W^{s}_0$ by 
\ant{
\phi_{0, \ell}(s) = \ell e^{-2s} + O(e^{-6s})
}
with $\ell \in \R$ determining all the coefficients of higher order and the $O( \cdot)$ vanishing for $\ell  = 0$. It is straightforward  to show that if $\ell>0$ then $\phi_{0, \ell}(s)$ lies on the branch of the stable manifold that stays above $0$ for all $s \in \R$, i.e., $\phi_{0, \ell}(s) > 0$ for all $s \in \R$. If $\ell=0$ then $\phi_{0, \ell}(s) = 0$ for all $s \in \R$. Lastly, if $\ell<0$ then $\phi_{0, \ell}(s) < 0$ for all $s \in \R$.  Different choices of $\ell$ correspond to translations  in $s$ along the respective branches of the stable manifold, which is what uniqueness means in the statement of Lemma~\ref{lem: hm}.  

Multiplying the equation~\eqref{eq: phi} by $\dot{\phi}$ and integrating from $s$ to $\infty$, one obtains the energy identity 
\EQ{
- \dot{\phi}_{0, \ell}^2(s) +  2\int_{s}^{\infty}  \dot{\phi}^2_{0, \ell}( \rho) \, d \rho  =  -\sin^2( \phi_{0, \ell}(s))
}
from which one deduces that if $\ell \neq 0$ it is impossible for $(\phi_{0, \ell}(s), \dot{\phi}_{0, \ell}(s))$ to ever equal $(0, 0)$ for any $s \in [- \infty,  \infty)$. 
Passing back to the original variables, $r, Q(r)$ we have three trajectories
\EQ{ \label{Q pm}
Q_{\ell_{\pm}}(r) :=  \ell_{\pm} r^{-2} + O(r^{-6}), \quad 
Q_{0}(r) := 0
}
To pass back to the $5d$ setting we set $\phi_{\ell}(r) := Q_{\ell}(r)/ r$ proving~\eqref{hm ell s} for our solutions to~\eqref{hm 5 s}. When $\ell \neq 0$ we note that the fact that $\lim_{r \to 0}Q_{\ell}(r)  \neq 0$  means that $\phi_{\ell} \notin L^5(\R^5)$ for $\ell \neq 0$. 

Next we prove the lemma for solutions to~\eqref{hm 5 h}. The proof is nearly identical so we only briefly summarize the argument. Again, first defining $Q(r):= r \fy(r)$ and then setting $s = \log(r)$ and $\phi(s) = Q(r)$, we can reduce matters to a phase portrait analysis for the autonomous ODE, 
\EQ{\label{eq: phi h}
\ddot \phi + \dot\phi= \sinh(2 \phi) 
 }
 We note that here there is only one fixed point for the vector field $X(x, y) =  (y, -y + \sinh 2x)$ at $(x, y) = (0, 0)$. The rest of the proof is identical to the $\Sp^3$ target case. 
 
\end{proof}

\subsubsection{Singular stationary solutions to~\eqref{u eq}} 
We prove the an analogous result for the cubic equation,~\eqref{u eq}. Again below we reduce matters to a phase portrait analysis after a reduction to an autonomous ODE. For a alternative, but also simple approach to the analogous result for the focusing supercritical semilinear equation in $3d$ we refer the reader to \cite[Proposition $3.2$]{DKM5}. 
\begin{lem}\label{lem: ell} For any $\ell  \in \R$, there exists a  radial $C^2$ solution $ \fy_{\ell}$ of 
\EQ{ \label{eq e}
-\fy_{rr}  - \frac{4}{r} \fy_r  = \fy^3\,  \quad r >0
}
with the asymptotic behavior 
\EQ{ \label{fy ell}
r^3\fy_{\ell}(r)  =  \ell + O( r^{-4}) \mas r  \to \infty
}
The $O( \cdot)$ is determined by $\ell$ and vanishes for $\ell =0$. Moreover, for $\ell  \neq 0$ we have $ \fy_{\ell} \notin L^5( \R^5)$, which means that $\fy_{\ell} \notin \dot{H}^{\frac{3}{2}}(\R^5)$ by Sobolev embedding.

\end{lem} 

\begin{proof} Motivated by the reduction to an autonomous system in the plane in the case of the wave maps equations, we seek a similar reduction for solutions to~\eqref{eq e}. Observe that $\fy$ solves~\eqref{eq e} if and only if $w(r):= r \fy(r)$ solves 
\EQ{
- w_{rr}  - \frac{2}{r} w_r + \frac{2 w - w^3}{r^2}  = 0
}
In order for a solution $\fy$ to~\eqref{eq e} to satisfy~\eqref{fy ell} we also require 
\EQ{
\lim_{r \to \infty} w(r) = 0
}
To find the reduction to an autonomous equation, we set $s:= \log(r)$ and $\phi(s) = w(r)$ and obtain the following equation for $\phi$: 
\EQ{ \label{aut ode}
 \ddot{\phi} +  \dot{\phi} - 2 \phi + \phi^3 = 0,  \quad \lim_{s \to \infty} \phi(s)  = 0
 }
 Again, we set $x(s) = \phi(s)$, $y(s) =  \dot\phi(s)$ and rewrite~\eqref{eq: phi} as the system
\EQ{\label{xy ode1}
 \pmat{ \dot{x} \\  \dot y}  = \pmat{ y \\ - y+ 2x - x^3} =: X(x, y)
}
and we denote by $\Phi_s$ the flow associated to the vector filed $X$. The equilibria of \eqref{xy ode1}  occur at points $z_0 = (0 , 0)$ and $z_{\pm}  = ( \pm 1, 0)$.  At $z_0 = (0, 0)$ the flow has a \emph{saddle} structure with eigenvalues $\la_+ = 1$, $\la_-= -2$, and the corresponding unstable and stable invariant subspaces for the linearized flow are given by the spans of $(1, \la_+)=(1, 1)$, respectively $(1, \la_-) = (1, -2)$. The flow has a \emph{sink} at each of the other two equilibrium points $z_{\pm}$.  In a neighborhood $V \ni z_0 = (0, 0)$ we then have the  one-dimensional invariant unstable manifold 
\ant{
W^u = \{ (x, y) \in V \mid \Phi_s(x, y) \to z_0\, \,  \textrm{as} \, \, s \to - \infty\}
}
and the one-dimensional invariant stable manifold 
\ant{
W^{s} = \{ (x, y) \in V \mid \Phi_s(x, y) \to z_0\, \,  \textrm{as} \, \, s \to  \infty\}
}
which are tangent at $z_0 =(0, 0)$ to the invariant subspaces of the linearized flow. In particular, we can parameterize the stable manifold $W^{s}$ by 
\ant{
\phi_{\ell}(s) = \ell e^{-2s} + O(e^{-6s})
}
with $\ell \in \R$ determining all the coefficients of higher order and $O(\cdot)$ vanishing for $\ell =0$. Next, multiply the equation~\eqref{aut ode} by $\phi$ and integrate from $s$ to $+ \infty$ to obtain the energy identity
\EQ{
 \int_s^{\infty} \dot{\phi}_{\ell}( \rho) \, d \rho =  \frac{1}{2} \dot{\phi}^2_{\ell} (s)  - \phi^2_{\ell}(s) + \frac{1}{4} \phi^4_{\ell}(s)
 }
 From the above it is clear that if $\ell \neq 0$ then it is impossible to have $(\phi_{\ell}(s),  \dot{\phi}_{\ell}(s)) = (0, 0)$ for any $s \in [- \infty, \infty)$ since the right-hand-side would vanish while the left-hand-side would be strictly positive. Changing back to the original variables we have $w_{\ell}(r)  = \phi_{\ell} (s)$ and 
 \EQ{
 w_{\ell}(r)  =  \ell r^{-2} + O( r^{-6}) \mas r \to \infty
 }
 and for $\ell  \neq 0$ we have $\lim_{r \to 0} w(r) \neq 0$. Hence, setting $\fy_{\ell}(r)  = w(r)/r$ we see that $\fy_{\ell} \notin L^5(\R^5)$. This finishes the proof. 
\end{proof}

\subsection{Proof of Proposition~\ref{prop: r1}} 
We are now ready to begin the proof of Proposition~\ref{prop: r1} in earnest. The proof proceeds in several steps and is similar to the arguments presented in~\cite[Section $5$]{KLS} as well as~\cite[Section $5$]{L13}. The method is inspired by the ``channels of energy" technique introduced in the seminal papers~\cite{DKM4, DKM5}. The key ingredient in the proof are the following exterior energy estimates for the free radial wave equation in $\R^{1+5}$ that were proved in~\cite[Section $4$]{KLS}. 

\begin{prop}\label{prop: ext en}\cite[Proposition $4.1$]{KLS}
Let $\Box V=0$ in $\R^{1+5}_{t,x}$ with radial data $(V_0,V_1)\in \dot H^1\times L^2(\R^5)$.
Then with some absolute constant $c_0>0$ one has for every $R>0$
\EQ{
\label{ext}
\max_{\pm}\, \limsup_{t\to\pm\infty} \int_{r>R+|t|}^\I(V_t^2 + V_r^2)(t,r) r^4\, dr \ge c_0 \| \pi_a^\perp (V_0,V_1)\|_{\dot H^1\times L^2(r>R)}^2
}
where $\pi_R=\Id-\pi_R^\perp$ is the orthogonal projection onto the plane $$P(R):=\{( c_1 r^{-3}, c_2 r^{-3})\:|\: c_1, c_2\in\R\}$$
in the space $\dot H^1\times L^2(r>R)$.  The left-hand side of~\eqref{ext} is zero for all data in this plane. 
\end{prop}
\begin{rem} The presence of the projections $\pi^{\perp}_R$ on the right-hand-side of~\eqref{ext} can be attributed to the fact that $r^{-3}$ is the Newton potential in $\R^5$. To see this,  consider initial data $(V_0, 0) \in \dot H^1 \times L^2 (r \ge R)$ which satisfies $(V_0, 0) = (r^{-3}, 0)$ for all $r \in \{r \ge R>0\}$, with $V_0(r) \equiv 0$ on $\{r \le R/2\}$. Then, using the finite speed of propagation,  the corresponding free evolution $V(t, r)$ is given by $V(t, r) = r^{-3}$ on the region $\{r \ge R+\abs{t}\}$. It is clear that the left-hand-side of~\eqref{ext} vanishes for this solution as $t \to \pm \infty$, and therefore an estimate without the projection is false. The other family of counterexamples to an estimate without the projection is generated by taking initial data $(0, V_1) = (0, r^{-3})$ on the exterior region $\{r\ge R>0\}$ which has solution $V(t, r) = tr^{-3}$ on $\{r \ge R+ \abs{t}\}$. 
\end{rem}
\begin{rem}
The orthogonal projections $\pi_R$, $\pi_R^{\perp}$ are precisely  
\begin{align*}
&\pi_R(V_0, 0) = R^3r^{-3} V(R), \quad 
\pi_R(0, V_1) = R\,  r^{-3} \int_R^{\infty} V_1( \rho)  \rho \, d \rho,\\
&\pi_R^{\perp}(V_0, 0) = V_0(r) - R^3r^{-3} V_0(a), \quad \pi_R^{\perp}(0, V_1) = V_1(r) - R\,  r^{-3} \int_R^{\infty} V_1( \rho)  \rho \, d \rho,
\end{align*}
and thus we have 
\begin{align*}
\|\pi_R(f,g)\|_{\dot{H}^1 \times L^2(r>R)}^2 &= 3R^3 f^2(R) + R \, \left(\int_R^{\infty} rg(r) \, dr\right)^2 \\
\|\pi_R^{\perp}(f, g)\|_{\dot{H}^1 \times L^2(r>R)}^2&= \int^{\infty}_R f_r^2(r) \, r^4 \, dr - 3R^3f^2(R)\\
& \quad + \int_R^{\infty} g^2(r) \, r^4 \, dr - R\left(\int_R^{\infty} rg(r) \, dr\right)^2.
\end{align*}
\end{rem}

The idea behind the proof is that the exterior energy decay~\eqref{u h1 R} together with the exterior energy estimates for the  free equation, i.e.,~\eqref{ext}, can be combined to obtain exact asymptotics for the initial data of our pre-compact trajectory, $u_0(r)= u(0, r)$  and $u_1(r) = u_t(0, r)$ as $r \to \infty$, viz., 
\EQ{\label{u0 asymp}
&r^3 u_0(r)  \to \ell_0 \mas r \to \infty\\
&r \int_r^{\infty} u_1(s) s \, d s  \to 0 \mas r \to \infty 
}
We conclude by showing via a contradiction argument  that $\vec u(t) =(0, 0)$ is the only solution to either~\eqref{u eq},~\eqref{u eq wms}, or~\eqref{u eq wmh} with both the compactness property as in Proposition~\ref{prop: r1} and initial data with the above asymptotics.  For this final portion of the proof, the exact structure of the nonlinearity in the underlying equation only enters at the level of the underlying elliptic theory , i.e., when we use the results of Lemma~\ref{lem: ell} for solutions to~\eqref{u eq} and Lemma~\ref{lem: hm} when dealing with solutions to either~\eqref{u eq wms} or~\eqref{u eq wmh}.  Given that we have already proved Lemma~\ref{lem: ell} and Lemma~\ref{lem: hm}, and the fact that the conclusions of these results are identical, we carry out the rest of the argument under the general assumption that~$\vec u(t)$ has the compactness property as in Proposition~\ref{prop: r1} and satisfies an equation of the form 
\EQ{\label{eq u F}
&u_{tt} - u_{rr} - \frac{4}{r} u_r = F(r, u)\\
&\vec u(0) = (u_0, u_1)
}
where
\EQ{ \label{F u3}
\abs{F(r, u)} \lesssim  \abs{u}^3
}
We note that the estimate~\eqref{F u3} is clear in the case of the nonlinearities in both~\eqref{u eq} and~\eqref{u eq wms}. Moreover, since we will only be considering solutions $\vec u(t)$ to~\eqref{u eq wmh} with the compactness property, the bound on the $\dot{H}^{3/2}\times \dot{H}^{1/2}$ norm of $\vec u(t)$ implies  $L^{\infty}$ control over $r u(t, r)$, which means that in our situation the  nonlinearity $F_{\Hp^3}(r, u) = Z_{\Hp^3}(ru) u^3$ also satisfies the estimate~\eqref{F u3}. 

The proof now proceeds in serval steps. \\
\textbf{Step 1:} First,  we use the exterior energy estimates for the free radial wave equation in Proposition~\ref{prop: ext en} together with~\eqref{u h1 R} to deduce an inequality for a solutions $\vec u(t)$ with the compactness property as in Proposition~\ref{prop: r1}. We remark that by using compactness this result holds uniformly in time.  

\begin{lem} \label{lem: R ineq} Let $\vec u(t)$ be as in Lemma~\ref{lem: compact}. There exists a number $R_0>0$ such that for all $R>R_0$ and for all $t \in I_{\max}(\vec u)$ we have 
\EQ{ \label{R ineq}
\left\| \pi^{\perp}_R \vec u(t) \right\|_{\dot{H}^1 \times L^2(r \ge R)} \lesssim R^{-1} \| \pi_R \vec u(t)\|_{\dot{H}^1 \times L^2(r \ge R)}^3 
}
where $\pi_R$ and $\pi_{R}^{\perp}$ are as in Proposition~\ref{prop: ext en}. We remark that the constant in~\eqref{R ineq} is uniform in $t \in I_{\max}(\vec u)$. 
\end{lem}

Before beginning the proof of Lemma~\ref{lem: R ineq} we  need a preliminary result regarding a modified Cauchy problem with a small data theory in the energy space, which also captures the dynamics of the solution with the compactness property on exterior cones $\Cc_R:=\{(t, r) \mid r \ge R + \abs{t}\}$. In order to compare a nonlinear wave and the underlying free evolution with the same initial data in the energy space, we need to be in a small data setting  where the Duhamel formula and Strichartz estimates can be combined effectively. That we only will consider the evolution on the exterior cone $\Cc_R$ allows us to truncate the initial data as well as the nonlinearity in a way that renders the initial value problem \emph{subcritical} relative to the energy, but still preserves the flow on $\Cc_R$. 

 We now fix a smooth radial function $\chi \in C^{\infty}(\R^5)$ where $\chi(\abs{x}) = \chi(r) = 1$ for $\{r \ge 1\}$ and $\chi(r) = 0$ on $\{r \le 1/2\}$. Then rescaling we set $\chi_R(r):= \chi(r/R)$ and for every $R>0$ we consider the modified Cauchy problem: 
\EQ{\label{h eq}
&h_{tt} - h_{rr} - \frac{4}{r} h_r = F_R(r, h), \quad F_{R}(r, h) := \chi_R(r)F(r, u)\\
&\vec h(0) = (h_0, h_1)
}
The point of this modification is that by forcing the nonlinearity to have support outside $B(0, R)$ we remove the super-critical nature of the problem. This allows for a small-data theory in the energy space $\dot{H}^1 \times L^2$ by way of Strichartz estimates and the usual contraction mapping based argument. We define a norm $Z(I)$ where $0 \in I \subset \R$ is a time interval by 
\EQ{ \label{def: Z}
\|\vec h\|_{Z(I)} = \|h\|_{L^{2}_t(I; L^{5}_x(\R^5))} + \|\vec h(t) \|_{L^{\infty}_t(I; \dot H^1 \times L^2(\R^5))}
}

\begin{lem} \label{lem: h sd}
There exists an $\e_0$ small enough, so that for all  $R>0$ and all initial data $\vec h(0)=(h_0, h_1) \in \dot{H}^1 \times L^2 (\R^5)$ with 
\ant{
\| \vec h(0) \|_{ \dot H^1 \times L^2(\R^5)} < \e_0
}
there exists a unique global-in-time solution $\vec h(t) \in \dot{H}^1 \times L^2$ to \eqref{h eq}. Moreover,  $\vec h(t)$ satisfies 
\EQ{\label{h Z small}
\|\vec h\|_{Z(\R)} \lesssim \|\vec h(0)\|_{\dot{H}^1 \times L^2} \lesssim  \e_0
}
If we denote the free evolution of the same initial data by $\vec h_L(t):= S(t) \vec h(0)$, then we also have 
\EQ{\label{h hL}
\sup_{t \in \R} \|  \vec h(t) - \vec h_L(t) \|_{ \dot{H}^1 \times L^2 } \lesssim R^{-1} \| \vec h(0) \|_{\dot H^1 \times L^2}^3  
}
\end{lem}
\begin{proof}
As we mentioned above, the proof follows the usual argument based on Strichartz estimates and the Duhamel formula. Using Proposition~\ref{prop stric} with $s=1$,  $(p, q, \ga) = (2, 5, 1)$ and $(a, b, \rho) = (\infty, 2, 0)$, it suffices to control 
\EQ{
\| \chi_R F(r, h)\|_{L^1_tL^2_x} &\lesssim \| \chi_R  \abs{h}\|_{L^{\infty}_t L^{10}_x} \| h\|_{L^{2}_tL^{5}_x}^2 \\
& \lesssim R^{-1} \| h\|_{L^{\infty}_t  \dot{H}^1_x}\| h\|_{L^{2}_tL^{5}_x}^2 \lesssim R^{-1} \| h\|_{Z}^3
}
where we have used~\eqref{F u3} as well as radial Sobolev embedding, i.e., Lemma~\ref{lem: rad se}. 
\end{proof}

\begin{rem}\label{u uR}
We remark that for every $t \in I_{\max}(u)$ the nonlinearity $F_R$ in~\eqref{h eq} satisfies
\ant{
F_R(r, u) = F(r, u), \quad \forall r \ge R + \abs{t}.
}
 By the finite speed of propagation we can deduce that solutions to~\eqref{h eq} and~\eqref{eq u F} agree on the exterior cone $\Cc_R:=\{(t, r) \mid r \ge R + \abs{t}\}$.
\end{rem}
With the above remark in mind, we can now prove Lemma~\ref{lem: R ineq}. 
\begin{proof}[Proof of Lemma~\ref{lem: R ineq}] As will always be the case in this section, $\vec u(t)$ is a solution with the compactness property as in Lemma~\ref{lem: compact}. We will prove the inequality~\eqref{R ineq} first for time $t=0$. The proof for all times $t \in I_{\max}$ for   $R>R_0$ independent of $t$ will follow immediately from the compactness property of $\vec u(t)$. 
First,  define truncated  data, $\vec u_R(0) = (u_{0, R}, u_{1, R})$ by
\ant{
&u_{0, R}(r):= \begin{cases} u_0(r) \mfor r \ge R\\ u_0(R) \mfor 0 \le r \le R \end{cases}\\
&u_{1, R}(r):= \begin{cases} u_1(r) \mfor r \ge R\\ 0 \mfor 0 \le r \le R \end{cases}
}
This new data is small in $\dot H^1 \times L^2$  for large $R>0$. In fact, 
\ant{
\| \vec u_R(0)\|_{\dot{H}^1 \times L^2(\R^5)} = \|\vec u(0) \|_{\dot{H}^1 \times L^2(r \ge R)}
}
which means $R_0>0$ can be chosen large enough such that for all $R \ge R_0$ we have 
\ant{
\| \vec u_R(0)\|_{\dot{H}^1 \times L^2} \le \de \le \min(\e_0, 1)
}
where $\e_0$ is chosen as in Lemma~\eqref{lem: h sd}. By Lemma~\eqref{lem: h sd} we can find the unique global solution $\vec u_R(t)$ to~\eqref{h eq} with initial data $\vec u_R(0)$, which satisfies~\eqref{h Z small} and \eqref{h hL}. 

Let $\vec u_{R, L}(t)$ be the free evolution of the initial data $\vec u_{R}(0)$. Note that 
\begin{multline*}
\| \vec u_R(t) \|_{\dot{H}^1\times L^2(r \ge R+\abs{t})} = \| \vec u_R(t) \|_{\dot{H}^1\times L^2(r \ge R+\abs{t})} \\
\ge \|u_{R, L}(t) \|_{\dot{H}^1\times L^2(r \ge R+\abs{t})} - \|\vec u_R(t)- u_{R, L}(t) \|_{\dot{H}^1\times L^2(r \ge R+\abs{t})} 
\end{multline*}
Now apply~\eqref{h hL} to $\vec u_{R}(t)$ and take $R>R_0$ large enough so that  
\begin{align*}
\|\vec u_R(t)- u_{R, L}(t) \|_{\dot{H}^1\times L^2(r \ge R+\abs{t})}  &\le  \|\vec u_R(t)- u_{R, L}(t) \|_{\dot{H}^1\times L^2} \\
&\lesssim R^{-1} \| \vec u_{R}(0)\|_{\dot{H}^1 \times L^2}^3 \\
& = R^{-1} \| \vec u(0)\|_{\dot{H}^1 \times L^2(r \ge R)}^3 
\end{align*}
Next, we put together the two preceding inequalities to obtain
\EQ{ \label{uR ineq}
\| \vec u_R(t) \|_{\dot{H}^1\times L^2(r \ge R+\abs{t})} \ge  \|u_{R, L}(t) \|_{\dot{H}^1\times L^2(r \ge R+\abs{t})} - CR^{-1}\| \vec u(0)\|_{\dot{H}^1 \times L^2(r \ge R)}^3.
}
We note that by Remark~\ref{u uR} we have 
\begin{align}\label{u uR eq}
\vec u_R(t) = \vec u(t), \quad \forall (t, r) \in \Cc_R.
\end{align}
This means that we can use Lemma~\ref{lem: compact} to deduce that 
\EQ{ \label{uR to 0}
\lim_{\abs{t} \to \infty}  \| \vec u_R(t)\|_{\dot{H}^1 \times L^2(r \ge R+ \abs{t})} = 0
}
Next we will let $t \to \pm \infty$ in~\eqref{uR ineq} (the choice here is determined by the maximum in Proposition~\ref{prop: ext en}) and  use Proposition~\ref{prop: ext en} to provide a lower bound for the right-hand side of~\eqref{uR ineq}. Taking the limit as $t \to \pm \infty$, using~\eqref{uR to 0}, Proposition~\ref{prop: ext en} and the fact that $\vec u_R(0)= \vec u(0)$ on $\{r \ge R\}$ we obtain that 
\ant{
 \| \pi^{\perp}_R  \vec u(0) \|_{ \dot{H}^1 \times L^2(r \ge R)} \lesssim R^{-1}\| \vec u(0)\|_{\dot{H}^1 \times L^2(r \ge R)}^3.
 }
Now,  use the orthogonality of the projection $\pi_{R}$ to expand out the right-hand side  above
 \ant{
 \| \pi^{\perp}_R  \vec u(0) \|_{ \dot H ^1 \times L^2(r \ge R)} \lesssim R^{-1}\left(\| \pi_R\vec u(0)\|_{\dot H ^1 \times L^2(r \ge R)}^2  + \| \pi_R^{\perp}\vec u(0)\|_{\dot H^1 \times L^2(r \ge R)}^2 \right)^{\frac{3}{2}}.
 }
 To complete the proof, we let $R_0$ be large enough so that we can absorb the $\pi_R^{\perp}$ term above on the right-hand-side into the left-hand-side which gives 
  \begin{align*}
  \| \pi^{\perp}_R  \vec u(0) \|_{ \dot H^1 \times L^2(r \ge R)} \lesssim R^{-1}\| \pi_R\vec u(0)\|_{\dot H^1 \times L^2(r \ge R)}^3,
  \end{align*}
 proving the lemma for $t=0$. To see that~\eqref{R ineq} holds for all $t \in \R$ we note that by the pre-compactness of $K$ we can choose $R_0 = R_0( \e_0)$ so that  for all $R \ge R_0$ we have 
 \ant{
 \| \vec u(t) \|_{ \dot{H}^1 \times L^2( r \ge R)} \le \min(\e_0, 1)
 } 
 uniformly in $t \in I_{\max}$. Now simply repeat the above argument with truncated initial data at time $t=t_0$ and $R \ge R_0$ given by 
 \ant{
&u_{0, R, t_0}(r):= \begin{cases} u(t_0, r) \mfor r \ge R\\ u_0(t_0, R) \mfor 0 \le r \le R \end{cases}\\
&u_{1, R, t_0}(r):= \begin{cases} u_t(t_0, r) \mfor r \ge R\\ 0 \mfor 0 \le r \le R \end{cases}
}
This ends the proof. 
\end{proof}

\textbf{Step 2:}
Next, we will use the estimates in Lemma~\ref{lem: R ineq} to deduce the asymptotic behavior of $\vec u(0, r)$ as $r \to \infty$ that was described in~\eqref{u0 asymp}. To be precise, we establish the following lemma. 

\begin{lem}\label{lem: space decay}
Let $\vec u(t)$ be as in Proposition~\ref{prop: r1}. Then we can find an $\ell_0 \in \R$ such that 
\begin{align}
r^3 u_0(r) \to \ell_0 \mas r \to \infty \label{u0 decay}\\
r \int_r^\I u_1( \rho) \rho \, d \rho \to 0 \mas r \to \infty \label{u1 decay}
\end{align}
at the rates 
\begin{align}
&\abs{r^3 u_0(r) - \ell_0}  = O(r^{-4}) \mas r \to \infty \label{u0 rate}\\
&\abs{r \int_r^\I u_1( \rho) \rho \, d \rho } = O(r^{-2}) \mas r \to \infty\label{u1 rate}
\end{align}
\end{lem}
For the proof of Lemma~\ref{lem: space decay} we make the following substitutions, which simplify the notation. Set
\EQ{\label{v def}
&v_0(t, r):= r^3 u(t, r), \\
&v_1(t, r):= r \int_r^\I u_t( t, \rho) \rho \, d \rho.
}
We will also often write $v_0(r) :=v_0(0, r)$,  $v_1(r):=v_1(0, r)$. Writing the projections in terms of $v_0, v_1$ gives
\EQ{\label{v u project} 
&\| \pi_R \vec u(t)\|^2_{\dot{H}^1 \times L^2(r \ge R)} = 3R^{-3} v_0^2(t, R) + R^{-1}v_1^2(t, R)\\
&\| \pi_R^{\perp} \vec u(t)\|^2_{\dot{H}^1 \times L^2(r \ge R)} = \int_R^\I \left( \frac{1}{r} \p_r v_0(t, r) \right)^2 \, dr + \int_R^\I \left( \p_r v_1(t ,r) \right)^2 \, dr,
}
and we can rephrase the conclusions of Lemma~\ref{lem: R ineq} in terms of the newly defined functions $(v_0, v_1)$. 
\begin{lem}\label{lem: v pro}
Let $(v_0, v_1)$ be defined as in~\eqref{v def}. Then there is an $R_0>0$ so that for all $R \ge R_0$, 
\ant{
\left( \int_R^\I \left( \frac{1}{r} \p_r v_0(t, r) \right)^2 + \left( \p_r v_1(t ,r) \right)^2 \, dr \right)^{\frac{1}{2}} \lesssim R^{-1} \left(3R^{-3} v_0^2(t, R) + R^{-1}v_1^2(t, R)\right)^{\frac{3}{2}}
}
where the implicit constant above is uniform in $t \in I_{\max}(\vec u)$. 
\end{lem}

Lemma~\ref{lem: v pro} is now used to establish difference estimates. We let $\de_1>0$ be small (to be determined precisely below) so that $\de_1 \le \e_0^2$ where $\e_0$ is the small constant in Lemma~\ref{lem: h sd}. We also choose $R_1 = R_1(\de _1)$ large enough such that for all $R \ge R_1$, 
\EQ{\label{R1 de1}
\| \vec u(t) \|_{\dot H^1 \times L^2(r \ge R)}^2 \le  \de_1 \le \e_0^2, \quad R_1^{-1} \le  \de_1
} 
Such an $R_1$ exists by the assumption that $\vec u$ satisfies the compactness property on $I_{\max}$.  

\begin{cor}\label{cor: diff est}
Let $R_1$ be as above. Then for all $R_1 \le r \le r' \le 2r$ and for all $t \in \R$  we have 
\begin{align}\label{v0 diff est}
&\abs{ v_0(t, r) - v_0(t, r') } \lesssim r^{-4} \abs{ v_0(t, r)}^3 + r^{-1} \abs{ v_1(t, r)}^3\\
\label{v1 diff est}
&\abs{ v_1(t, r) - v_1(t, r') } \lesssim r^{-5} \abs{ v_0(t, r)}^3 + r^{-2} \abs{ v_1(t, r)}^3
\end{align}
with the above estimates holding uniformly in $t \in \R$. 
\end{cor}

For convenience we  also record a  rewording of Corollary~\ref{cor: diff est} that  is a direct consequence of~\eqref{v u project}. 

\begin{cor} \label{cor: diff est 2}
Let  $R_1, \de_1$ be defined as in~\eqref{R1 de1}. Then for all $r, r'$ with $R_1 \le r \le r' \le 2r$ and for all $t \in \R$ we have 
\begin{align}\label{v0 diff est2}
&\abs{ v_0(t, r) - v_0(t, r') } \lesssim r^{-1} \de_1 \abs{ v_0(t, r)}+ \de_1 \abs{ v_1(t, r)}\\
\label{v1 diff est2}
&\abs{ v_1(t, r) - v_1(t, r') } \lesssim r^{-2} \de_1 \abs{ v_0(t, r)} + r^{-1} \de_1 \abs{ v_1(t, r)}
\end{align} 
where all of the above  hold uniformly in $t \in I_{\max}( \vec u)$. 
\end{cor}

\begin{proof}[Proof of Corollary~\ref{cor: diff est}]
This is an immediate consequence of Lemma~\ref{lem: v pro}. If $r \ge R_1$ and $r \le r' \le 2r$, then using Lemma~\ref{lem: v pro} gives 
\ant{
\abs{ v_0(t, r) - v_0(t, r') } &\le \left( \int_r^{r'} \abs{\p_r v_0(t,  \rho) } \, d \rho\right)\\
& \le  \left( \int_r^{r'} \abs{\frac{1}{\rho}\p_r v_0(t,  \rho) }^2 \, d \rho\right)^{\frac{1}{2}} \left( \int_r^{r'} \rho^2 \, d \rho\right)^{\frac{1}{2}}\\
&\lesssim r^{\frac{3}{2}} \left[ r^{-1} \left(3r^{-3} v_0^2(t, r) + r^{-1}v_1^2(t, r)\right)^{\frac{3}{2}}\right]\\
&\lesssim r^{-4} \abs{ v_0(t, r)}^3 + r^{-1} \abs{ v_1(t, r)}^3.
}
In the same fashion
\ant{
\abs{ v_1(t, r) - v_1(t, r') } 
& \le  \left( \int_r^{r'} \abs{\p_r v_1(t,  \rho) }^2 \, d \rho\right)^{\frac{1}{2}} \left( \int_r^{r'}  \, d \rho\right)^{\frac{1}{2}}\\
&\lesssim r^{-5} \abs{ v_0(t, r)}^3 + r^{-2} \abs{ v_1(t, r)}^3,
}
which proves the difference estimates. 
\end{proof}

Now, we use the difference estimates to establish an upper bound on the growth rates of $v_0(t, r)$ and $v_1(t, r)$.  
\begin{claim}
Let $v_0(t, r)$ and $v_1(t, r)$ be as in~\eqref{v def}. Then 
\begin{align}
 &\abs{v_0(t, r)} \lesssim r^{\frac{1}{6}} \label{v0 16}\\
 &\abs{v_1(t, r)} \lesssim r^{\frac{1}{6}}\label{v1 16}
\end{align}
with  constants that are uniform in $t \in I_{\max}$. 
\end{claim}
\begin{proof}
 Again, it  suffices to deduce the claim when $t=0$ because  the argument uses only the estimates in this section that hold uniformly in $t \in I_{\max}$. 

Let $r_0 \ge R_1$ be fixed and note that by setting $r= 2^nr_0$, $r'=2^{n+1}r_0$ in the difference estimates~\eqref{v0 diff est2},~\eqref{v1 diff est2} we have for all positive integers $n \in \N$, 
\begin{align}
&\abs{v_0(2^{n+1} r_0)} \le (1+ C_1(2^n r_0)^{-1} \de_1) \abs{v_0(2^nr_0)} + C_1 \de_1 \abs{v_1(2^nr_0)}\label{v0 trianlge}\\
&\abs{v_1(2^{n+1} r_0)} \le (1+ C_1(2^n r_0)^{-1} \de_1) \abs{v_1(2^nr_0)} + C_1 \de_1(2^n r_0)^{-2} \abs{v_0(2^nr_0)}\label{v1 triangle}
\end{align}
We introduce the notation 
\ant{
a_n:= \abs{v_0(2^nr_0)}\\
b_n:=\abs{v_1(2^nr_0)}
}
Adding~\eqref{v0 trianlge} to \eqref{v1 triangle} gives
\begin{align*}
a_{n+1} + b_{n+1} &\le (1+ C_1\de_1((2^n r_0)^{-1} +(2^n r_0)^{-2}))a_n + (1+C_1 \de_1(1+(2^n r_0)^{-1}))b_n\\
&\le (1+ 2C_1 \de_1)(a_n + b_n)
\end{align*}
Arguing inductively we conclude that for all $n \in \N$, 
\begin{align*}
(a_n+b_n) \le (1+2C_1 \de_1)^{n}(a_0+b_0).
\end{align*}
Now let $\de_1$ be small enough so that $(1+2C_1 \de_1) \le 2^{\frac{1}{6}}$, giving 
\EQ{\label{an bn}
a_n \le C(2^n r_0)^{\frac{1}{6}},\\
b_n \le C(2^n r_0)^{\frac{1}{6}}.
}
Note that  $C=C(r_0)$ but this is irrelevant for our purposes since we have fixed $r_0$.  Observe that~\eqref{an bn} proves~\eqref{v0 16} and~\eqref{v1 16} for $r=2^nr_0$. The estimates~\eqref{v0 16} and~\eqref{v1 16} for arbitrary $r$ now follow by combining~\eqref{an bn} together with the difference estimates~\eqref{v0 diff est},~\eqref{v1 diff est}. 
\end{proof}

We are now ready to begin  extracting a limits. We first show that  $v_1(t, r)$ tends to zero in  order to get the correct decay rate for $v_0(t, r)$. This requires several steps. 
\begin{claim} \label{claim: ell1}For every $t \in I_{\max}$ there is a number $\ell_1(t) \in \R$ so that 
\EQ{
\abs{ v_1(t, r) -  \ell_1(t)}  = O(r^{-2}) \mas r \to \infty
}
where the implicit constant is uniform in $t \in I_{\max}$. 
\end{claim}
\begin{proof}
Again it suffices to prove the claim for  $t=0$. Let $r_0 \ge R_1$ with $R_1 $ as in~\eqref{R1 de1}. Inserting~\eqref{v0 16},~\eqref{v1 16} into the difference estimate~\eqref{v1 diff est} yields
\EQ{\label{v1 lim}
\abs{ v_1(2^{n+1}r_0) - v_1(2^n r_0) } &\lesssim (2^nr_0)^{-5} (2^nr_0)^{\frac{1}{2}} + (2^nr_0)^{-2} (2^nr_0)^{\frac{1}{2}}\\
& \lesssim  (2^nr_0)^{-\frac{3}{2}}
} 
Therefore the infinite series 
\begin{align*}
 \sum_n \abs{ v_1(2^{n+1}r_0) - v_1(2^n r_0) } < \infty,
 \end{align*}
 is bounded, 
 which then implies that there  exists  $\ell_1 \in \R$ so that 
 \ant{
  \lim_{n \to \infty} v_1(2^nr_0) = \ell_1.
  }
Another application of  the difference estimates along with the growth estimates~\eqref{v0 16},~\eqref{v1 16} gives, 
\begin{align*}
  \lim_{r \to \infty} v_1(r) = \ell_1.
  \end{align*}
To establish the correct rate of convergence, we remark  that the  limit above implies that $\abs{v_1(r)}$ is bounded, and therefore  the same logic that give~\eqref{v1 lim} can be used to get 
\ant{
\abs{ v_1(2^{n+1}r) - v_1(2^n r) } &\lesssim(2^nr)^{-2}
}
for all $r$ large enough.
Finally, 
\ant{
\abs{v_1(r) - \ell_1} = \abs{ \sum_{n \ge 0}(v_1(2^{n+1} r) - v_1(2^n r)) } \lesssim r^{-2} \sum_{n \ge 0} 2^{-2n} \lesssim r^{-2}
}
which finishes the argument.
\end{proof}

Next, we show that  $\ell_1(t) = \ell_1$ is independent of $t$. 
\begin{claim}\label{claim: ell1 constant} The function $\ell_1(t)$ in Claim~\ref{claim: ell1} is independent of $t \in I_{\max}$, that is there is a fixed number $\ell_1 \in \R$ so that  $\ell_1(t) =  \ell_1$ for all $t \in I_{\max}$.
\end{claim}

\begin{proof} Recall that
\begin{align*}
v_1(t, r) := r\int_r^{\I} u_t(t, \rho) \, \rho \, d\rho
\end{align*}
By Claim~\eqref{claim: ell1}, we see that  
\ant{
\ell_1(t) =  r\int_r^{\I} u_t(t, \rho) \, \rho \, d\rho + O(r^{-2}) \mas r \to \infty
}
Let $t_1, t_2 \in I_{\max}$ be arbitrary with, say  $t_1 \neq t_2$. We prove that 
\EQ{
\ell_1(t_2)- \ell_1(t_1) =0
}
Indeed,  averaging in $R \ge R_1$ yields 
\ant{
\ell_1(t_2) - \ell_1(t_1) 
&= \frac{1}{R} \int_R^{2R}\left(s \int_s^{\infty} (u_t(t_2, r) - u_t(t_1, r)) r \, dr \right) \, ds + O(R^{-2})\\
&=\frac{1}{R} \int_R^{2R}\left(s \int_s^{\infty}  \int_{t_1}^{t_2}u_{tt}(t, r) \, dt\,  r \, dr \right) \, ds + O(R^{-2})
}
Using the fact that $\vec u(t)$ is a solution to \eqref{eq u F}, we can rewrite the above integral as
\begin{align*}
&=  \int_{t_1}^{t_2}\frac{1}{R} \int_R^{2R}\left(s \int_s^{\infty}  (ru_{rr}(t, r) + 4u_r(t, r) ) \, dr \right) \, ds\, dt + \\
&\, +  \int_{t_1}^{t_2} \frac{1}{R} \int_R^{2R}\left(s \int_s^{\infty}F(r, u)   \, dr \right) \,ds\, dt  + O(R^{-2})\\
& = I + II +  O(R^{-2})
\end{align*}
To estimate $I$ we integrate by parts twice: 
\EQ{ \label{I}
I&= \int_{t_1}^{t_2}\frac{1}{R} \int_R^{2R}\left(s \int_s^{\infty}  \frac{1}{r^3} \p_r(r^4 u_r(t, r)) \, dr \right) \, ds\, dt\\
& = 3\int_{t_1}^{t_2}\frac{1}{R} \int_R^{2R}\left(s \int_s^{\infty}  u_r(t, r) \, dr \right) \, ds\, dt -\int_{t_1}^{t_2}\frac{1}{R} \int_R^{2R}s^2  u_r(t, s) \, ds\, dt \\
& = -3\int_{t_1}^{t_2}\frac{1}{R} \int_R^{2R}r \, u(t, r) \, dr\, dt -\int_{t_1}^{t_2}\frac{1}{R} \int_R^{2R}r^2\, u_r(t, r) \, dr\, dt \\
&= -\int_{t_1}^{t_2}\frac{1}{R} \int_R^{2R}r \, u(t, r) \, dr\, dt + \int_{t_1}^{t_2} (Ru(t,R)-2Ru(t, 2R))\, dt
}
To bound the above we recall that by the definition of $v_0$ and~\eqref{v0 16} we have 
\EQ{\label{u bound1}
r^3\abs{u(t, r)} := \abs{v_0(t, r))} \lesssim r^{\frac{1}{6}}
}
uniformly in $t \in I_{\max}$, and hence $\abs{u(t, r)} \lesssim r^{-\frac{17}{6}}$ uniformly in $t \in I_{\max}$. Inserting~\eqref{u bound1} into the last line in~\eqref{I} gives, 
\begin{align*}
I = \abs{t_2-t_1}O(R^{-\frac{11}{6}})
\end{align*}
Next we estimate $II$. We again use~\eqref{u bound1} to see that for $r>R_1$ we have 
\ant{
\abs{F(r, u(t, r))} \lesssim \abs{u(t, r)}^3  \lesssim r^{-\frac{17}{2}}  \lesssim r^{-8}
}
Therefore, 
\ant{
II \lesssim \int_{t_1}^{t_2} \frac{1}{R} \int_R^{2R}\left(s \int_s^{\infty}r^{-8}   \, dr \right) \,ds\, dt = \abs{t_2- t_1} O(r^{-6})
}
Combining all of the above we have
\ant{
\abs{\ell_1(t_2)- \ell_1(t_1)}  = O(R^{-\frac{11}{6}}) \mas R \to \infty
}
which means that $\ell_1(t_1) = \ell_1(t_2)$.
\end{proof}

Next, we prove  that $\ell_1 \equiv 0$. 
\begin{claim}\label{claim: ell1=0}  $\ell_1 =0$. 
\end{claim} 
\begin{proof} Suppose that $\ell_1 \neq 0$. We have shown that for $R \ge R_1$  and for every $t \in I_{\max}$,
\ant{
R\int_R^{\infty}u_t(t, r) \, r \, dr = \ell_1 + O(R^{-2}),
}
Hence, by taking $R$ large, the left-hand side will have the same sign as $\ell_1$, for all $t \in I_{\max}$.
Therefore, we can choose $R \ge R_1$ large enough such that for all $t \in I_{\max}$, 
\EQ{
\label{ge ell1}
\abs{ R\int_{R}^{\infty}u_t(t, r) \, r \, dr }\ge \frac{ |\ell_1|}{2}.
}
We now consider two cases. First suppose that $T_+< \infty$. Then, since both $$\supp(u), \supp u_t \subset B(0, T_+-t)$$ we can find $t$ close enough to $T_+$ so that the left-hand-side of~\eqref{ge ell1} is identically zero, which is a contradiction if $\ell_1 \neq 0$. 

Next, consider the case that $T_+ = \infty$.  By integrating~\eqref{ge ell1} from $t=0$ to $t=T$ we obtain
\ant{
\Big| \int_0^TR\int_{R}^{\infty}u_t(t, r) \, r \, dr\, dt\Big|  \ge T\frac{|\ell_1|}{2} .
}
However, integrating in $t$ on the left-hand side and using \eqref{u bound1} we also have 
\ant{
\abs{R\int_{R}^{\infty}\int_0^Tu_t(t, r) \, r \, dt \, dr }&= \abs{ R\int_R^{\infty}(u(T, r) - u(0,r) )\, r \, dr}\\
& \lesssim R\int_R^{\infty} r^{-\frac{11}{6}}\, dr
\lesssim R^{\frac{1}{6}}.
}
Thus for a large  $R$ fixed we have 
\ant{
T\frac{|\ell_1|}{2}  \lesssim R^{\frac{1}{6}},
}
which is a contradiction once $T$ is taken large enough. 
\end{proof}

We are ready to complete the proof of Lemma~\ref{lem: space decay}. 
\begin{proof}[Proof of Lemma~\ref{lem: space decay}]
First we remark that  by putting together  Claims~\ref{claim: ell1},~\ref{claim: ell1 constant}, and~\ref{claim: ell1=0}, we have proved~\eqref{u1 decay} as well as~\eqref{u1 rate}, i.e.,  
\EQ{\label{v1  0}
\abs{v_1(r)} = O(r^{-2}) \mas r \to 0.
}
 It thus remains to prove that there exists $\ell_0 \in \R$ such that 
\ant{
\abs{v_0(r) - \ell_0} = O(r^{-4}) \mas r \to \infty.
}
To show the above, we plug~\eqref{v1 0} along with~\eqref{v0 16} into the difference estimate~\eqref{v0 diff est}. We see  that for fixed $r_0 \ge R_1$ and all $n \in \N$,  
\begin{align*}
\abs{v_0(2^{n+1}r_0) - v_0(2^nr_0)} &\lesssim (2^nr_0)^{-4} (2^nr_0)^{\frac{1}{2}} + (2^nr_0)^{-1} (2^nr_0)^{-6}\\
&\lesssim (2^nr_0)^{-\frac{7}{2}}.
\end{align*}
Thus the infinite series 
\ant{
 \sum_n \abs{ v_0(2^{n+1}r_0) - v_0(2^n r_0) } < \infty,
 }
 which then implies that there exists $\ell_0 \in \R$ so that 
 \ant{
  \lim_{n \to \infty} v_0(2^nr_0) = \ell_0.
  }
Another application of the  difference estimates~\eqref{v0 diff est} and the fact that $\abs{v_1(r)} \to 0$ gives that, 
 \ant{
  \lim_{r \to \infty} v_0(r) = \ell_0.
  }
To prove the  rate of convergence, we observe that $\abs{v_0(r)}$ is bounded since it has a limit, and therefore the difference estimate can be upgraded show that   
\ant{
\abs{ v_0(2^{n+1}r) - v_0(2^n r) } &\lesssim(2^nr)^{-4}
}
 for every $r \ge R_1$. Thus, 
\begin{align*}
\abs{v_0(r) - \ell_0} = \abs{ \sum_{n \ge 0}(v_1(2^{n+1} r) - v_1(2^n r)) } \lesssim r^{-4} \sum_{n \ge 0} 2^{-4n} \lesssim r^{-2},
\end{align*}
completing the proof.
\end{proof}

\textbf{Step $3$:} The final step in the proof of  Proposition~\ref{prop: r1} is to conclude that $\vec u(t, r)  \equiv (0, 0)$. We divide the final step into two cases depending on whether  $\ell_0$ in Lemma~\ref{lem: space decay} is zero or nonzero. 
\\
 \textbf{Case 1: $\ell_0=0$ implies $\vec u(t)\equiv(0, 0)$:}

We formulate the above as a lemma: 
\begin{lem} \label{lem: ell=0}Let $\vec u(t)$ be as in Proposition~\ref{prop: r1} and let $\ell_0\in \R$ be as in Lemma~\ref{lem: space decay}.  If $\ell_0 = 0$ then $\vec u(t)  \equiv (0, 0)$. 
\end{lem}
To prove the lemma, we will first prove the following  preliminary claim, which says that if $\ell_0 = 0$ then $(u_0, u_1)$ is compactly supported. We conclude the proof of the lemma by showing that the only solution with pre-compact trajectory and compactly supported initial data is  $\vec u(t) = (0, 0)$. 
\begin{claim}\label{claim: comp supp}Let $\ell_0$ be as in Lemma~\ref{lem: space decay}. If $\ell_0=0$ then $(u_0, u_1)$ is compactly supported. 
\end{claim}
\begin{proof}
If $\ell_0 = 0$, then  for $r \ge R_1$ we have
\EQ{ \label{ell0 0}
\abs{v_0(r)}  = O (r^{-4}) \mas r \to \infty\\
\abs{v_1(r)}  = O( r^{-2}) \mas r \to \infty
}
Hence  for $r_0 \ge R_1$, 
\EQ{\label{lower bound}
\abs{v_0(2^nr_0)} + \abs{v_1(2^nr_0)} \lesssim (2^nr_0)^{-4} +(2^nr_0)^{-2}
}
On the other hand, the difference estimates~\eqref{v0 diff est} and~\eqref{v1 diff est} and~\eqref{ell0 0} yield
\ant{
&\abs{v_0(2^{n+1} r_0)} \ge (1-C_1(2^nr_0)^{-12}) \abs{v_0(2^nr_0)} - C_1(2^nr_0)^{-5} \abs{v_1(2^nr_0)}\\
&\abs{v_1(2^{n+1} r_0)} \ge (1-C_1(2^nr_0)^{-6}) \abs{v_0(2^nr_0)} - C_1(2^nr_0)^{-13} \abs{v_1(2^nr_0)}
}
For large $r_0$ we  combine the above two lines to deduce that
\ant{
\abs{v_0(2^{n+1} r_0)} + \abs{v_1(2^{n+1} r_0)} \ge (1- 2C_1r_0^{-5})( \abs{v_0(2^{n} r_0)} + \abs{v_1(2^{n} r_0)})
}
Now fix $r_0$ large enough so that $2C_1r_0^{-5} < \frac{1}{4}$.  By an  inductive argument we conclude that 
\begin{align*}
( \abs{v_0(2^{n} r_0)} + \abs{v_1(2^{n} r_0)}) \ge \left(\frac{3}{4}\right)^{n} ( \abs{v_0(r_0)} + \abs{v_1( r_0)})
\end{align*}
Next, use~\eqref{lower bound} to estimate the left-hand-side above. We have 
\ant{
\left(\frac{3}{4}\right)^{n} ( \abs{v_0(r_0)} + \abs{v_1( r_0)}) \lesssim 2^{-2n} r_0^{-2}
}
which yields 
\ant{
3^n ( \abs{v_0(r_0)} + \abs{v_1( r_0)}) \lesssim 1, 
}
However, this is  impossible unless $(v_0(r_0), v_1(r_0)) = (0, 0)$. Therefore, $$\vec v(r_0):= (v_0(r_0), v_1(r_0)) = (0, 0).$$ To make this statement about $(u_0, u_1)$ we recall that~\eqref{v u project} implies that 
\ant{
\| \pi_{r_0} \vec u(0)\|_{ \dot{H}^1 \times L^2(r \ge r_0)} = 0.
}
 Lemma~\ref{lem: R ineq} then gives 
\ant{
\| \pi_{r_0}^{\perp} \vec u(0)\|_{ \dot{H}^1 \times L^2(r \ge r_0)} = 0,
}
and thus 
\ant{
\|  \vec u(0)\|_{ \dot{H}^1 \times L^2(r \ge r_0)} = 0,
}
which completes the proof since we know that $\lim_{r \to \infty} u_0(r) = 0$. 
\end{proof}

\begin{proof}[Proof of Lemma~\ref{lem: ell=0}] Suppose that $\ell_0 = 0$. Then, by Claim~\ref{claim: comp supp}, $(u_0, u_1)$ must be compactly supported. Assume that $(u_0, u_1) \not \equiv (0, 0)$.  We can then define $\rho_0>0$ by 
\ant{
\rho_0:= \inf\left\{ \rho \,  : \,   \left\| \vec u(0)\right\|_{H^1 \times L^2(r \ge \rho)} = 0\right\}
}
Let $\e>0$ be a small number (to be determined below). Find $\rho_1 = \rho_1( \e)$ with $\frac{ 1}{2} \rho_0 < \rho_1 < \rho_0$  so that 
\begin{align*}
0< \| \vec u(0)\|_{\dot{H}^1 \times L^2(r \ge \rho_1)}^2  < \e \le \de_1^2
\end{align*}
with $\de_1$ as in~\eqref{R1 de1}. We have 
\EQ{\label{small}
3 \rho_1^{-3} v_0^2(\rho_1) + \rho_1^{-1}v_1^2(\rho_1)&+ \int_{\rho_1}^\I \left( \frac{1}{r} \p_r v_0( r) \right)^2 \, dr + \int_{\rho_1}^\I \left( \p_r v_1(r) \right)^2 \, dr  = \\
& = \| \pi_{\rho_1} \vec u(0)\|^2_{\dot{H}^1 \times L^2(r \ge \rho_1)}+\| \pi_{\rho_1}^{\perp} \vec u(0)\|^2_{\dot{H}^1 \times L^2(r \ge \rho_1)} \\
&= \|  \vec u(0)\|^2_{\dot{H}^1 \times L^2(r \ge \rho_1)} < \e
}
Using Lemma~\ref{lem: v pro} with $R= \rho_1$ gives 
\EQ{\label{v rho}
\left( \int_{\rho_1}^\I \left[\left( \frac{1}{r} \p_r v_0(r) \right)^2 + \left( \p_r v_1( r) \right)^2 \right]\, dr\right)^{\frac{1}{2}}  \lesssim \rho_1^{-\frac{11}{2}} \abs{v_0( \rho_1)}^3 + \rho_1^{-\frac{5}{2}}\abs{v_1( \rho_1)}^3
}
Since $v_0( \rho_0) = v_1( \rho_0) = 0$ we can argue as in Corollary~\ref{cor: diff est} and Corollary~\ref{cor: diff est 2} to obtain
\ant{
&\abs{v_0( \rho_1)} = \abs{v_0(\rho_1) - v_0( \rho_0)} \le C_1\, \e\, (\abs{v_0( \rho_1)} + \abs{v_1( \rho_1)})  \\
&\abs{v_1( \rho_1)} = \abs{v_1(\rho_1) - v_1( \rho_0)} \le C_2\, \e\, (\abs{v_0( \rho_1)} + \abs{v_1( \rho_1)})
}
where above we used that $\frac{1}{2} \rho_0< \rho_1 < \rho_0$ to find constants $C_1, C_2$ depending only $\rho_0$ which is fixed, and the uniform constant in~\eqref{v rho}, but not on~$\e$. Combining the above estimates  we get
\EQ{
(\abs{v_0( \rho_1)} + \abs{v_1( \rho_1)}) \le C_3 \e (\abs{v_0( \rho_1)} + \abs{v_1( \rho_1)})
}
which implies that $\abs{v_0( \rho_1)} = \abs{v_1( \rho_1)} = 0$ after choosing $\e>0$ small enough. By~\eqref{v rho} and~\eqref{small} we see that 
\begin{align*}
\| \vec u(0)\|_{\dot{H}^1 \times L^2(r \ge \rho_1)} = 0.
\end{align*}
But this contradicts the definition of $\rho_0$ because $\rho_1< \rho_0$.
\end{proof}
We have finished the proof of Proposition~\ref{prop: r1} in the case where $\ell_0 = 0$. We next move on to  the case $\ell_0 \neq0$.

 \textbf{Case 2: $\ell_0 \neq 0$ cannot happen:} 
To complete the proof of Proposition~\ref{prop: r1} we prove that $\ell_0 \neq 0$ is impossible. Indeed, we show that if $\ell_0 \neq0$ then the solution $\vec u(t)$ to ~\eqref{u eq} (resp.~\eqref{u eq wms}, resp.~\eqref{u eq wmh}) with the compactness property as in Proposition~\ref{prop: r1} must be equal to a nonzero stationary solution, $\fy_{\ell_0}$, of~\eqref{eq e} (resp.~\eqref{hm 5 s}, resp.~\eqref{hm 5 h}). 

On the other hand, we know by Lemma~\ref{lem: ell} (resp. Lemma~\ref{lem: hm}) that $(\fy_{\ell_0}, 0) \not \in \dot{H}^{\frac{3}{2}}(\R^5)$ for any $\ell_0 \neq 0$, which gives us  contradiction since by construction our solution $\vec u(t) \in  \dot{H}^{\frac{3}{2}} \times \dot{H}^{\frac{1}{2}}$.

The basic idea is to linearize about the elliptic solution $\fy_{\ell_0}$ given by Lemma~\ref{lem: hm} (resp. Lemma~\ref{lem: hm}) with the same leading order spacial asymptotics as the critical data $\vec u(0)$.  In Steps $1, 2$ we proved that 
\ant{
r^3u_0(r) = \ell_0 + O(r^{-4}) \mas r \to \infty.
}
Now, let $\fy_{\ell_0}(r)$ be the solution from either Lemma~\ref{lem: ell} or Lemma~\ref{lem: hm} depending of course on which equation $\vec u(0)$ is initial data for.   Hence $\fy_{\ell_0}$ solves the relevant elliptic equation and satisfies
\begin{align*}
\fy_{\ell_0}(r)= \ell_0 + O(r^{-4}) \mas r \to \infty.
\end{align*}
Next define $ \vec w_{\ell_0}(0) = ( w_{\ell_0, 0}, w_{\ell_0, 1})$ by 
\EQ{\label{w l def}
&w_{\ell_0, 0}(r) := u_0(r) - \fy_{\ell_0}(r)\\
&w_{\ell_0, 1}(r):=  u_1(r),
}
and for all  $t \in I_{\max}( \vec u)$  set 
\EQ{\label{w l t def}
w_{\ell_0}(t, r):=  u(t, r) - \fy_{\ell_0}(r).
}
 We  record various properties of $\vec w_{\ell_0} = ( w_{\ell_0}, \p_t w_{\ell_0})$. First, we have
\EQ{\label{w l limits}
&v_{\ell_0, 0}(r):= r^3 w_{\ell_0, 0}(r) = O(r^{-4}) \quad \textrm{as} \, \, \,  r \to \infty\\
&v_{\ell_0, 1}(r):= r \int_r^{\infty} w_{\ell_0, 1}(\rho) \rho, d\rho = O(r^{-2}) \quad \textrm{as} \, \, \, r \to \infty
}
Next, we write down the equation for $\vec w_{\ell_0}$. If $\vec u(t)$ solves~\eqref{u eq} and $\fy_{\ell_0}$ solves~\eqref{eq e} then $\vec w_{\ell_0}(t)$ solves
\EQ{
\label{w eq}
\p_{tt}w_{\ell_0} - \p_{rr} w_{\ell_0} - \frac{4}{r} \p_r w_{\ell_0} &= 3 \fy_{\ell_0}^2w_{\ell_0} + 3 \fy_{\ell_0}w_{\ell_0}^2 + w_{\ell_0}^3\\
& =: \NN_{\textrm{cubic}}(  \fy_{\ell_0}, w_{\ell_0})
}
If $\vec u(t)$ solves~\eqref{u eq wms} or~\eqref{u eq wmh} then $\vec w_{\ell_0}(t)$ solves an equation of the form 
\EQ{
\label{w eq wm}
&\p_{tt}w_{\ell_0} - \p_{rr} w_{\ell_0} - \frac{4}{r} \p_r w_{\ell_0} =  \NN_{\textrm{w.m.}}(r,  \fy_{\ell_0}, w_{\ell_0}) \\
& \NN_{ \textrm{w.m.}}(r,  \fy_{\ell_0}, w_{\ell_0})  := - 2\frac{\CC(2 r \fy_{\ell_0}) - 1}{r^2}  w_{\ell_0} - \Sm(2r \fy_{\ell_0})\frac{ \CC(2 r w_{\ell_0}) - 1}{r^3}  \\
& \qquad \qquad \qquad  \qquad -\CC(2r \fy_{\ell_0})\frac{( \Sm(2r w_{\ell_0})- 2r w_{\ell_0})}{r^3}
}
where if $\vec u(t)$ solves the $\Sp^3$ target equation then $\CC = \cos$, $\Sm = \sin$ and if $\vec u(t)$ solves the $\Hp^3$ target equation we have $\CC = \cosh$, $\Sm = \sinh$. In either case,  we have the estimates
\EQ{ \label{nn wm}
\abs{\NN_{\textrm{w.m.}}( r, \fy_{\ell_0}, w_{\ell_0})} \lesssim  \abs{\fy_{\ell}}^2 \abs{w_{\ell_0}}  +  \abs{ \fy_{\ell_0}} \abs{ w_{\ell_0}}^2 + \abs{ w_{\ell_0}}^3
}
where again we have used our $L^{\infty}$ control over $r u$ and $r \fy_{\ell}$ in the case of the $\Hp^3$ target equation. 

The crucial point here is that  by construction $\vec w_{\ell_0}$ inherits the main conclusions from Lemma~\ref{lem: compact} because $\fy_{\ell_0}$ is stationary, i.e., 
 \EQ{
 \|\vec w_{\ell_0}(t)\|_{\dot H^1 \times L^2(r \ge R+ \abs{t})} \to 0 \mas \abs{t} \to \infty
 }
W are now in position to prove, much as in the case $\ell_0 = 0$, that we must have $\vec w_{\ell_0} \equiv (0, 0)$. We state this conclusion as a lemma. 
\begin{lem}\label{lem: wl 0}
Suppose $\ell_0 \neq 0$. Let   $\vec w_{\ell_0}$ be as in~\eqref{w l def}, \eqref{w l t def}. Then, $\vec w_{\ell_0} \equiv (0, 0)$, that is, 
$\vec u(0) = ( \fy_{\ell_0}, 0)$ where $\fy_{\ell_0}$ is given by either Lemma~\ref{lem: ell} or Lemma~\ref{lem: hm} (depending of course on which equation $\vec u(t)$ solves).   This means that  $\vec u(0) \notin \dot H^{\frac{3}{2}} \times \dot{H}^{\frac{1}{2}}$ which is a contradiction. 
\end{lem}
The proof of Lemma~\ref{lem: wl 0} will be  very similar to the argument that was presented in the previous steps which completed the  case $\ell_0 = 0$. We will thus omit many details here. 


We have  established the asymptotic behavior of $\vec w_{\ell_0}$ given in~\eqref{w l limits}.  However we recall that one of the keys to the argument in the previous steps was the estimate~\eqref{R ineq}, as this gave a quantitative restriction on the proximity of  $\vec u(0)$ to the plane $P(R)$. Here we establish a similar estimate  for $\vec w_{\ell_0}$. As before  modify the Cauchy problems~\eqref{w eq} and~\eqref{w eq wm} on the interior of the cone $\Cc_R(t, r):=\{r \le R + \abs{t}\}$ for large $R$. As in the relevant sections in~\cite{KLS, L13, DKM5} we alter the right-hand-side of~\eqref{w eq} and~\eqref{w eq wm}. 

For each $R>0$ we define $\fy_{\ell_0, R}$ by
\EQ{ 
\fy_{\ell_0,R}(t, r):= \begin{cases} \fy_{\ell_0}(R+ \abs{t}) \mfor 0 \le r \le R + \abs{t}\\ \fy_{\ell_0}(r) \mfor r \ge R+ \abs{t}\end{cases}
}
where $\fy_{\ell_0}$ is the relevant elliptic solution depending on whether we are dealing with solutions to~\eqref{u eq},~\eqref{u eq wms}, or~\eqref{u eq wmh}. 
Next, set
\begin{align*} 
& \NN_{R, \textrm{cubic}}(t, r, w_{\ell_0}):=   \NN_{ \textrm{cubic}}(\fy_{\ell_0, R}, w_{\ell_0})\\ 
&\NN_{R, \textrm{w.m.}}(t,  r, w_{\ell_0}):= \begin{cases}  \NN_{\textrm{w.m.}}(R+\abs{t}, \fy_{\ell_0, R}, w_{\ell_0})\mfor 0 \le r \le R+ \abs{t} \\  \NN_{\textrm{w.m.}}(r, \fy_{ \ell_0, R}, w_{\ell_0}) \mfor r \ge R+ \abs{t} \end{cases}
\end{align*}
Using \eqref{fy ell} and \eqref{hm ell s} together with~\eqref{nn wm}, we can deduce that for $R$ large enough we have  
\begin{align} \label{N1R}
&\abs{\NN_{R, \textrm{cubic}}(t, r, w_{\ell_0})} \lesssim \begin{cases} (R+\abs{t})^{-6} \abs{w_{\ell_0}} + (R+\abs{t})^{-3} \abs{w_{\ell_0}}^2  + \abs{w_{\ell_0}}^3, \\ \mfor 0 \le r \le R + \abs{t}\\ r^{-6} \abs{w_{\ell_0}}^2 + r^{-3} \abs{w_{\ell_0}}^2  + \abs{w_{\ell_0}}^3 \mfor r \ge R + \abs{t}\end{cases}\\
&\abs{\NN_{R, \textrm{cubic}}(t, r, w_{\ell_0})}  \lesssim \begin{cases} (R+\abs{t})^{-6} \abs{w_{\ell_0}} + (R+\abs{t})^{-3} \abs{w_{\ell_0}}^2  + \abs{w_{\ell_0}}^3, \\ \mfor 0 \le r \le R + \abs{t}\\ r^{-6} \abs{w_{\ell_0}}^2 + r^{-3} \abs{w_{\ell_0}}^2  + \abs{w_{\ell_0}}^3 \mfor r \ge R + \abs{t}\end{cases} \label{N2R}
\end{align}
Because we have the same estimates for $\NN_{R, \textrm{cubic}}$ and $\NN_{R, \textrm{w.m.}}$ above we will simply write $\NN_{R}$ for both of them in what follows. We consider a modified Cauchy problem. As in the set-up for Lemma~\ref{lem: h sd} we define  a smooth radial function $\chi \in C^{\infty}(\R^5)$ with $\chi(r) = 1$ for $r \ge 1$ and $\chi(r) = 0$ on $r \le 1/2$. Rescale to define $\chi_R(r):= \chi(r/R)$ and for each $R>0$ we consider: 
\EQ{ \label{w eq R}
&w_{tt}- w_{rr} - \frac{4}{r} w_r =\ti{\NN}_R(t, r,w) \\
&\ti{\NN}_R(t, r, w) :=  \chi_R\NN_{R}(t, r, w)\\
& \vec w(0) = (w_0, w_1) \in \dot{H}^1 \times L^2(\R^5) 
}
The idea  here is that we have introduced spacial decay  as well as time integrability into the potential terms that  will then allow these to be treated in a perturbative manner. We have also introduced the cut-off $\chi_R$. As in the Lemma~\ref{lem: h sd} this removes the super-critical nature of the nonlinearity and allows us to treat the right-hand-side perturbatively in the \emph{energy space}. This is an analog of Lemma~\ref{lem: h sd} but here we have linearized about a nontrivial elliptic solution $\fy_{\ell_0}$. As before we define  the norm $Z(I)$ by
\ant{
\|\vec w\|_{Z(I)} = \|w\|_{L^{2}_t(I; L^{5}_x(\R^5))} + \|\vec w(t) \|_{L^{\infty}_t(I; \dot H^1 \times L^2)}
}
\begin{lem}\label{lem: w cp} There exists $R_2>0$ and  $\de_2>0$ small enough such that for every $R>R_2$ and all data $\vec w(0) = (w_0, w_1) \in \dot{H}^1 \times L^2(\R^5)$ with 
\ant{
\|\vec w(0) \|_{\dot{H}^1 \times L^2(\R^5)}   < \de_2
} 
there exists a unique global solution $\vec w(t) \in \dot{H}^1 \times L^2$ to \eqref{w eq R}. Moreover, $\vec w(t) $ satisfies 
\EQ{ \label{w a pri}
\|w\|_{Z(\R))} \lesssim \|\vec w(0)\|_{\dot{H}^1 \times L^2(\R^5)} \lesssim \de_2
}
Finally, if we denote the free evolution with the same data by $w_L(t):=S(t) \vec h(0)\in \dot{H}^1 \times L^2(\R^5)$, then we have 
\EQ{\label{w wL}
\sup_{t \in \R}\|\vec w(t) - \vec w_L(t)\|_{\dot{H}^1 \times L^2}& \lesssim R^{-4}\|\vec w(0)\|_{\dot{H}^1 \times L^2} + R^{-5/2}\|\vec w(0)\|_{\dot{H}^1 \times L^2}^2 \\
& \quad+R^{-1} \|\vec w(0)\|_{\dot{H}^1 \times L^2}^3 
}
\end{lem}
\begin{proof} The proof follows from~\eqref{N1R} and~\eqref{N2R} and is very similar to the proof of Lemma~\ref{lem: h sd}. In particular by a standard argument it suffices to control $\ti{\NN}_R$ in $L^1_tL^2_x$ using the estimates ~\eqref{N1R} and~\eqref{N2R}. We omit the details. See for example~\cite[Lemma~$5.17$]{L13} for a more detailed argument of a similar result. 
\end{proof}

With Lemma~\ref{lem: w cp} in hand, we argue exactly as in the proof of Lemma~\ref{lem: R ineq} to prove the key inequality: 
\begin{lem} \label{lem: w R ineq} There exists $R_2>0$ such  that for all $R>R_2$ and for all $t \in I_{\max}(u)$ we have 
\begin{align*}
\| \pi^{\perp}_R \vec w_{\ell_0}(t) \|_{\dot{H}^1 \times L^2(r \ge R)} &\lesssim R^{-4} \| \pi_R \vec w_{\ell_0}(t)\|_{\dot{H}^1 \times L^2(r \ge R)} +R^{-\frac{5}{2}} \| \pi_R \vec w_{\ell_0}(t)\|_{\dot{H}^1 \times L^2(r \ge R)}^2\\
&+  R^{-1} \| \pi_R \vec w_{\ell_0}(t)\|_{\dot{H}^1 \times L^2(r \ge R)}^3 
\end{align*}
where $\pi_R$,  $\pi_R^{\perp}$ are as in Proposition~\ref{prop: ext en}.  We note that the constant above is uniform in $t \in I_{\max}(u)$. 
\end{lem}

Next, we prove that $(\p_r  w_{\ell_0, 0}, w_{\ell_0, 1})$ must be compactly supported. 
\begin{claim}\label{claim: w comp supp} Let $\vec w_{\ell_0}$ be as in~\eqref{w l def}. Then $(\p_r  w_{\ell_0, 0}, w_{\ell_0, 1})$ must be compactly supported.
\end{claim}
\begin{proof}[Proof of Claim~\ref{claim: w comp supp}]
To prove the  claim, we pass to the $\vec v_{\ell_0}$ formulation.  With $(v_{\ell_0, 0}, v_{\ell_0, 1})$ defined as in~\eqref{w l limits} we can  conclude that for all $R>R_2$ large enough we have 
\EQ{\label{v l ineq}
\int_R^{\infty} &\left(\frac{1}{r} \p_r v_{\ell_0, 0}(r)\right)^2 \, dr + \int_R^{\I} (\p_r v_{\ell_0,1}( r))^2 \, dr   
\lesssim  R^{-19}v_{\ell_0,0}^2(R) + R^{-11} v_{\ell_0,0}^4(R)\\
&\quad+ R^{-11}v_{\ell_0,0}^6(R)  
+R^{-17}v_{\ell_0,1}^2(R)
+ R^{-7} v_{\ell_0,1}^4(R) + R^{-5}v_{\ell_0,1}^6( R) \\
&\lesssim  R^{-11}(v_{\ell_0, 0}^2(R) + v_{\ell_0, 1}^2(R))
}
where the first inequality above follows by rephrasing Lemma~\ref{lem: w R ineq} in terms of $ \vec v_{\ell_0} = (v_{\ell_0, 0}, v_{\ell_0, 1})$ by using~\eqref{v u project},  and the last line following from the decay estimates in \eqref{w l limits}.  

Next, arguing as in Corollary~\ref{cor: diff est}, we can prove difference estimates, i.e, for all $R_2 \le r \le r' \le 2r$ we have
\ali{
&\abs{v_{\ell_0, 0}(r) - v_{\ell_0, 0}(r')} \lesssim r^{-4}(v_{\ell_0, 0}^2(r) + v_{\ell_0, 1}^2(r))^{\frac{1}{2}},\\
&\abs{v_{\ell_0, 1}(r) - v_{\ell_0, 1}(r')} \lesssim r^{-5}(v_{\ell_0, 0}^2(r) + v_{\ell_0, 1}^2(r))^{\frac{1}{2}}.
}
Defining the vector  $\vec v_{\ell_0}  = (v_{\ell_0, 0},  v_{\ell_0, 1})$, 
where
\begin{align*}
\abs{\vec v_{\ell_0}(r)}^2:= v_{\ell_0, 0}^2(r) + v_{\ell_0, 1}^2(r)
\end{align*}
this means that  
\ant{
&\abs{\vec v_{\ell_0}(r) - \vec v_{\ell_0}(r')} \lesssim r^{-4}\abs{\vec v_{\ell_0}(r)}.
}
We can then prove that for fixed $r_0 \ge R_2$ large enough,
\begin{align*}
\abs{\vec v_{\ell_0}(2^{n+1}r_0)} \ge \frac{3}{4} \abs{\vec v_{\ell_0}(2^{n}r_0)}.
\end{align*}
Iterating this, we see that for each $n \in \N$, we have
\ant{
\abs{\vec v_{\ell_0}(2^{n}r_0)} \ge \left(\frac{3}{4}\right)^{n} \abs{\vec v_{\ell_0}(r_0)}.
}
On the other hand, using~\eqref{w l limits} we have 
\ant{
\abs{ \vec v_{\ell_0}(2^n r_0)} \lesssim (2^n r_0)^{-2}.
}
Combining the last two lines we obtain
\ant{
3^n \abs{ \vec v_{\ell_0}(r_0)} \lesssim 1,
}
which then implies that $\vec v_{\ell_0}(r_0) = (0, 0)$. Plugging this last fact into~\eqref{v l ineq} yields
\ant{
\int_{r_0}^{\infty} \left(\frac{1}{r} \p_r v_{\ell_0, 0}(r)\right)^2 \, dr + \int_{r_0}^{\I} (\p_r v_{\ell_0,1}( r))^2 \, dr    = 0.
}
Therefore, 
\begin{multline*}
\| \vec u_{\ell_0}\|_{\dot{H}^1 \times L^2(r \ge r_0)}^2= \\= \int_{r_0}^{\infty} \left(\frac{1}{r} \p_r v_{\ell_0, 0}(r)\right)^2 \, dr + \int_{r_0}^{\I} (\p_r v_{\ell_0,1}( r))^2 \, dr + 3r_0^{-3}v_{\ell_0, 0}^2(r_0) + r_0^{-1} v_{\ell_0, 1}^2(r_0) = 0
\end{multline*}
which implies that $(\p_r u_{\ell_0, 0}, u_{\ell_0, 1})$ must be compactly supported. 
\end{proof}

We can now complete the proof of Lemma~\ref{lem: wl 0} by showing that in fact,  $\vec w_{\ell_0} \equiv (0, 0)$.
\begin{proof}[Proof of Lemma~\ref{lem: wl 0}]The proof follows the same argument as the proof of Lemma~\ref{lem: ell=0}. Suppose that $$(\p_r w_{\ell_0,0}, w_{\ell_0,1}) \neq (0, 0).$$  
By Claim~\ref{claim: w comp supp},   $(\p_r w_{\ell_0,0}, w_{\ell_0,1})$ is compactly supported. Then we can define
 $\rho_0>0$ by 
 \ant{
 \rho_0 := \inf \left\{ \rho>0 \, : \,  \|\vec w_{\ell_0}\|_{ \dot H^1 \times L^2(r \ge \rho)} =  0\right\}
 }
Let $\e>0$ be a small number (to be determined below). Let $\rho_1 = \rho_1( \e)$ be chosen so that 
\begin{align} \label{rho1 def}
&\frac{ 1}{2} \rho_0 < \rho_1 < \rho_0, \mand  \rho_0 - \rho_1 < \e, \\
&0< \| \vec u_{\ell_0}(0)\|_{\dot{H}^1 \times L^2(r \ge \rho_1)}^2  < \de_2^2 \label{rho1 def2}
\end{align}
with $\de_2$ is as in Lemma~\ref{lem: w cp}. With $(v_{\ell_0, 0}, v_{\ell, 1})$ as above,  
\EQ{\label{v ell small}
3 \rho_1^{-3} v_{\ell_0,0}^2(\rho_1) + \rho_1^{-1}v_{\ell_0, 1}^2(\rho_1)&+ \int_{\rho_1}^\I \left( \frac{1}{r} \p_r v_{\ell_0,0}( r) \right)^2 \, dr + \int_{\rho_1}^\I \left( \p_r v_{\ell_0,1}( r) \right)^2 \, dr  = \\
& = \| \pi_{\rho_1} \vec u_{\ell_0}(0)\|^2_{\dot{H}^1 \times L^2(r \ge \rho_1)}+\| \pi_{\rho_1}^{\perp} \vec u_{\ell_0}(0)\|^2_{\dot{H}^1 \times L^2(r \ge \rho_1)} \\
&= \|  \vec u_{\ell_0}(0)\|^2_{\dot{H}^1 \times L^2(r \ge \rho_1)} < \de^2_2
}
Setting  $R= \rho_1$ in~\eqref{v l ineq} then yields 
\EQ{\label{vl rho1}
\left( \int_{\rho_1}^\I \left[\left( \frac{1}{r} \p_r v_{\ell_0,0}(r) \right)^2 + \left( \p_r v_{\ell_0,1}( r) \right)^2 \right]\, dr\right)^{\frac{1}{2}} & \lesssim \rho_1^{-\frac{11}{2}}\left( \abs{v_{\ell_0,0}( \rho_1)}^2 + \abs{v_{\ell_0,1}( \rho_1)}^2\right)^{\frac{1}{2}} \\
&\lesssim \rho_0^{-\frac{11}{2}}\left( \abs{v_{\ell_0,0}( \rho_1)}^2 + \abs{v_{\ell_0,1}( \rho_1)}^2\right)^{\frac{1}{2}} 
}
where the assumption $\frac{1}{2} \rho_0< \rho_1< \rho_0$ enters in the last line above. Because $v_{\ell_0, 0}( \rho_0) = v_{\ell_0,1}( \rho_0) = 0$ we can argue exactly as in Corollary~\ref{cor: diff est} to obtain
\ant{
&\abs{v_{\ell_0,0}( \rho_1)}^2 = \abs{v_{\ell_0,0}(\rho_1) - v_{\ell_0,0}( \rho_0)}^2 \le C_1\, \e^{3}\, (\abs{v_{\ell_0,0}( \rho_1)}^2 + \abs{v_{\ell_0,1}( \rho_1)}^2)\\
&\abs{v_{\ell_0,1}( \rho_1)}^2 = \abs{v_{\ell_0,1}(\rho_1) - v_{\ell_0,1}( \rho_0)}^2 \le C_2\, \e\, (\abs{v_{\ell_0,0}( \rho_1)}^2 + \abs{v_{\ell_0,1}( \rho_1)}^2)
}
with $C_1, C_2$ depending only $\rho_0$ which is fixed, and the uniform constant in~\eqref{vl rho1}, but not on~$\e$. 
Combining above  gives 
\EQ{
(\abs{v_{\ell_0,0}( \rho_1)}^2 + \abs{v_{\ell_0,1}( \rho_1)}^2) \le C_3 \e (\abs{v_{\ell_0,0}( \rho_1)}^2 + \abs{v_{\ell_0,1}( \rho_1)}^2),
}
This shows  that $\abs{v_{\ell_0,0}( \rho_1)} = \abs{v_{\ell_0,1}( \rho_1)} = 0$ once $\e>0$ is chosen small enough. By~\eqref{vl rho1} and the equalities in~\eqref{v ell small} we see that 
\ant{
\| \vec w_{\ell_0}(0)\|_{\dot{H}^1 \times L^2(r \ge \rho_1)} = 0. 
}
But this  contradicts the definition of $\rho_0$ since $\rho_1< \rho_0$. Thus, $(\p_r w_{\ell_0,0}, w_{\ell_0,1})  \equiv (0, 0)$. Since $w_{\ell_0}(r) \to 0$ as $r \to \infty$ we  also deduce that $(w_{\ell_0,0}, w_{\ell_0,1})  \equiv (0, 0)$.
\end{proof}

To clarify  this lengthy argument  we summarize the proof of Proposition~\ref{prop: r1}. 
\begin{proof}[Proof of Proposition~\ref{prop: r1}]
Let $\vec u(t)$ be a solution to~\eqref{u eq} (resp.~\eqref{u eq wms}, resp.~\eqref{u eq wmh}) with the compactness property as in Proposition~\ref{prop: r1}.  
By Lemma~\ref{lem: space decay} there exists $\ell_0 \in \R$ so that 
\EQ{
&\abs{r^3 u_0(r) - \ell_0 } = O(r^{-4}) \mas r \to \infty\\
&\abs{r \int_r^{\infty} u_1( \rho) \rho \, d\rho} = O(r^{-2}) \mas r \to \infty 
}
If $\ell_0 \neq 0$ then by Lemma~\ref{lem: wl 0}, $\vec u(0) = (u_0, u_1) = (\fy_{\ell_0}, 0)$ where $\fy_{\ell_0}$ is a nonzero solution to~\eqref{eq e} (resp. ~\eqref{hm 5 s}, resp.~\eqref{hm 5 h}) given by Lemma~\ref{lem: ell} (resp. Lemma~\ref{lem: hm}). However, this is impossible since $\fy_{\ell_0} \notin \dot{H}^{\frac{3}{2}}(\R^5)$, while on the other hand we know that $\vec u(0) \in \dot{H}^{\frac{3}{2}} \times \dot H^{\frac{1}{2}}$. 

Thus, we must have $\ell_0 = 0$.  By Lemma~\ref{lem: ell=0} we can conclude that $\vec u(0)= (0, 0)$, proving Proposition~\ref{prop: r1}. 
\end{proof}

\bibliographystyle{plain}
\bibliography{researchbib}

 \bigskip
 
 \vspace{1in}
 \bigskip 
 
\centerline{\scshape Benjamin Dodson}
\smallskip
{\footnotesize
 \centerline{Department of Mathematics, Johns Hopkins University}
\centerline{404 Krieger Hall, 3400 N. Charles Street,
 Baltimore, MD 21218 U.S.A.}
\centerline{\email{dodson@math.jhu.edu}}
} 

 \centerline{\scshape Andrew Lawrie}
\smallskip
{\footnotesize
 \centerline{Department of Mathematics, The University of California, Berkeley}
\centerline{970 Evans Hall \#3840, Berkeley, CA 94720, U.S.A.}
\centerline{\email{ alawrie@math.berkeley.edu}}
} 

\end{document}